\documentclass[nointlimits]{amsart}
\usepackage[utf8]{inputenc}
\usepackage[T1]{fontenc}
\usepackage{lmodern}
\usepackage{amsmath, amssymb, amsthm, enumerate}
\usepackage{natbib}
\usepackage{color}
\usepackage{hyperref}
\usepackage{microtype}
\usepackage{hyphenat}
\usepackage[normalem]{ulem} 

\newcommand{\TITLE}{Optimal Sobolev inequalities in the hyperbolic space}
\hypersetup{
	unicode=true,
	breaklinks=true,
	pdftitle={\TITLE},
	pdfauthor={Zdeněk Mihula}
	}

\usepackage{enumitem}

\numberwithin{equation}{section}
\theoremstyle{plain}
\newtheorem{theorem}{Theorem}[section]

\newtheorem{proposition}{Proposition}[section]

\theoremstyle{definition}
\newtheorem{remark}{Remark}[section]

\makeatletter
\let\c@corollary=\c@theorem
\let\c@proposition=\c@theorem
\let\c@remark=\c@theorem
\makeatother

\title{\TITLE}
\author{Zden\v ek Mihula}
\address{Zden\v ek Mihula, Czech Technical University in Prague, Faculty of Electrical Engineering, Department of Mathematics, Technick\'a~2, 166~27 Praha~6, Czech Republic}
\email{mihulzde@fel.cvut.cz}
\urladdr{\href{https://orcid.org/0000-0001-6962-7635}{0000-0001-6962-7635}}

\newcommand{\R}{\mathbb{R}}
\newcommand{\N}{\mathbb{N}}
\newcommand{\Hn}{\mathbb{H}^n}
\newcommand{\rn}{\R^n}

\newcommand{\B}{\mathbb{B}}

\DeclareMathOperator{\distH}{dist}

\DeclareMathOperator{\arctanh}{arctanh}
\newcommand{\gradH}[1][{}]{\operatorname{\nabla}_g^{#1}}
\newcommand{\lapH}[1][{}]{\operatorname{\Delta}_g^{#1}}

\newcommand{\Sm}[1][{m}]{S_{#1}}
\newcommand{\Tm}[1][{m}]{T_{#1}}

\newcommand{\dV}{V_g}

\newcommand{\absH}[1]{|#1|_g}

\newcommand{\dx}{{\fam0 d}}
\renewcommand{\d}[1]{\,\dx #1}
\newcommand{\dover}[1]{\,\frac{\dx #1}{#1}}

\newcommand{\M}{\mathfrak{M}}
\newcommand{\Mpl}{\M^+}

\newcommand{\A}{\mathbb{A}}

\DeclareMathOperator*{\esssup}{ess\,sup}

\iffalse
    \usepackage{refcheck}
    \def\crefname#1#2#3{}
    \def\cref#1{\{\getrefs#1,\relax\}}
    \def\getrefs#1,#2\relax{%
        \ref{#1}%
        \ifx\relax#2\relax\else
            , \getrefs#2\relax
        \fi
    }
\else
\usepackage[capitalise]{cleveref} 
\fi

\begin{document}
\setcitestyle{numbers}
\bibliographystyle{plainnat}

\thanks{This research was partly supported by grant no.~23-04720S of the Czech Science Foundation and by project OPVVV CAAS CZ.02.1.01/0.0/0.0/16\_019/0000778}

\begin{abstract}
We find the optimal function norm on the left-hand side of the $m$th order Sobolev type inequality $\|u\|_{Y(\Hn)} \leq C \|\gradH[m] u\|_{X(\Hn)}$ in the $n$-dimensional hyperbolic space $\Hn$, $1\leq m < n$. The optimal function norm in the inequality among all rearrangement\hyp{}invariant function norms is completely characterized. A variety of concrete examples of optimal function norms is provided. The examples include delicate limiting cases and, especially when $m\geq3$, seem to provide new, improved inequalities in these limiting cases. 
\end{abstract}

\maketitle

\section{Introduction}
Sobolev type inequalities have indisputably played an essential role in mathematics for decades and have been studied in diverse settings and from various different perspectives. In this paper, we will study a (higher order) Sobolev type inequality in the Poincar\'e ball model of the $n$-dimensional ($n\geq2$) hyperbolic space $\Hn$. Our main focus will be on the optimality of function norms in the inequality. More precisely, we will consider the Sobolev type inequality
\begin{equation}\label{intro:Sob_ineq}
\|u\|_{Y(\Hn)} \leq C \|\gradH[m] u\|_{X(\Hn)} \quad \text{for every $u\in V_0^m X(\Hn)$},
\end{equation}
where $1\leq m < n$, $\|\cdot\|_{Y(\Hn)}$ and $\|\cdot\|_{X(\Hn)}$ are rearrangement\hyp{}invariant function norms (such as the Lebesgue, Lorentz, or Orlicz ones), and $V_0^m X(\Hn)$ is a Sobolev type space of functions in $\Hn$ whose $m$th order hyperbolic gradient $\gradH[m]$, defined as
\begin{equation}\label{intro:mth_order_gradient}
\gradH[m] = 
\begin{cases}
\lapH[\frac{m}{2}]\quad&\text{if $m$ is even},\\
\gradH\lapH[\lfloor \frac{m}{2} \rfloor]\quad&\text{if $m$ is odd},
\end{cases}
\end{equation}
belongs to $X(\Hn)$, and that together with their lower order derivatives vanish at infinity (see \cref{sec:prel} for precise definitions). Here $\gradH$ is the hyperbolic gradient and $\lapH$ is the Laplace--Beltrami operator in $\Hn$, which can be expressed at $x\in\Hn$ in terms of their Euclidean counterparts as
\begin{align*}
\gradH &= \left(\frac{1 - |x|^2}{2}\right)^2 \nabla\\
\intertext{and}
\lapH &= \left(\frac{1 - |x|^2}{2}\right)^2 \Delta + (n - 2)\frac{1 - |x|^2}{2}\sum_{j = 1}^n x_j\frac{\partial}{\partial x_j}.
\end{align*}

The main results of this paper are contained in \cref{sec:main}, and they can be described as follows. Given $X(\Hn)$, we will characterize the optimal (i.e., the smallest) rearrangement\hyp{}invariant function space $Y(\Hn)$ such that \eqref{intro:Sob_ineq} is valid. Noteworthily, the most general theorem in that regard (\cref{thm:optimal_target_general_theorem}) does not impose any unnecessary assumptions on $X(\Hn)$; in other words, it completely characterizes the optimal (i.e., the strongest) rearrangement\hyp{}invariant function norm on the left-hand side of \eqref{intro:Sob_ineq}. When $X(\Hn)$ is not, loosely speaking, too close to $L^1(\Hn)$ or $L^\infty(\Hn)$, we can considerably simplify the general description of the optimal function space (\cref{thm:optimal_target_general_theorem_simplified_M_bounded_X,thm:optimal_target_general_theorem_simplified_M_bounded_Xasoc,thm:optimal_target_general_theorem_simplified_M_bounded_on_both_X_and_Xasoc}). Much as general theorems are important, concrete examples are often important just as much, if not more. Therefore, we also provide a lot of concrete examples (\cref{thm:examples_LZ}) of optimal rearrangement\hyp{}invariant function norms in \eqref{intro:Sob_ineq}. In the examples, $X(\Hn)$ is from the class of so-called Lorentz--Zygmund spaces, which appears to be a reasonable trade-off between simplicity and generality. Although the class of Lorentz--Zygmund spaces is rich enough to contain not only the Lebesgue and Lorentz spaces but also Orlicz spaces of logarithmic and exponential types, it is still not overly complicated. Furthemore, we compare the optimal rearrangement\hyp{}invariant function norm in \eqref{intro:Sob_ineq} with that in the Euclidean counterpart of \eqref{intro:Sob_ineq} (\cref{thm:optimal_target_eucl_and_supercritical}\ref{thm:optimal_target_eucl_and_supercritical_subcritical_item} and \cref{rem:optimal_target_eucl_and_supercritical}). Finally, we also characterize when the optimal function space in \eqref{intro:Sob_ineq} is contained in $L^\infty(\Hn)$ (\cref{thm:optimal_target_eucl_and_supercritical}\ref{thm:optimal_target_eucl_and_supercritical_supercritical_item}).

Loosely speaking, the most interesting things happen when $X(\Hn)$ is ``close to'' $L^1(\Hn)$, $L^\frac{n}{m}(\Hn)$, or $L^\infty(\Hn)$ (even more so when $m\geq3$). We illustrate it here with three particular examples of limiting higher order Sobolev type inequalities (see \cref{thm:examples_LZ} for more), which we will obtain in this paper and which, to the best of the author's knowledge, are completely new and missing in the literature; moreover, the rearrangement\hyp{}invariant function norms on the left-hand sides are optimal and cannot be replaced by essentially stronger rearrangement\hyp{}invariant function norms. First, for every $u\in V_0^m L^1(\Hn)$, we have
\begin{equation*}
\int_0^1 t^{\frac{n - m}{n} - 1} u^*(t) \d{t} + \sup_{t\in[1, \infty)} \frac{t u^{**}(t)}{(1 + \log t)^{\frac{m - 1}{2}}} \leq C \|\gradH[m] u\|_{L^1(\Hn)}
\end{equation*}
when $m\geq3$ is odd, and
\begin{equation*}
\sup_{t\in(0, 1)} t^{\frac{n - m}{n}} u^{**}(t) + \sup_{t\in[1, \infty)} \frac{t u^{**}(t)}{(1 + \log t)^{\frac{m}{2}}} \leq C \|\gradH[m] u\|_{L^1(\Hn)}
\end{equation*}
when $m$ is even. Here $u^*$ and $u^{**}$ stand for the nonincreasing and maximal nonincreasing rearrangements of $u$, respectively (see \cref{sec:prel} for precise definitions). Second, for every $u\in V_0^m L^\frac{n}{m}(\Hn)$, we have
\begin{equation*}
\Bigg( \int_0^1 \Big(\frac{u^{**}(t)}{1 - \log(t)}\Big)^{\frac{n}{m}}\dover{t} + \int_1^\infty u^{**}(t)^{\frac{n}{m}} \d{t} \Bigg)^{\frac{m}{n}}\leq C \|\gradH[m] u\|_{L^\frac{n}{m}(\Hn)}.
\end{equation*}
From the point of view of function norms, this inequality improves Moser--Trudinger inequalities in $\Hn$ (see \cite{NN:20} and references therein) in a similar way that Br\'ezis-- Wainger's and Hansson's limiting Sobolev inequalities (\cite{BW:80, H:79}) improve Moser--Trudinger inequalities in $\rn$. Third, when $X(\Hn) = L^\infty(\Hn)$, there is no rearrangement\hyp{}invariant function space $Y(\Hn)$ such that \eqref{intro:Sob_ineq} is valid. Nevertheless, we have the following. Let $\alpha_0\leq0$ and $\alpha_\infty > \lceil m/2 \rceil$. For every $u\in V_0^m L^{\infty, \infty; [\alpha_0, \alpha_\infty]}(\Hn)$, we have
\begin{equation*}
\|u\|_{L^\infty(\Hn)} + \sup_{t\in[1, \infty)} u^{**}(t) (1 + \log t)^{\alpha_\infty - \lceil \frac{m}{2} \rceil} \leq C \|\gradH[m] u\|_{L^{(\infty, \infty; [\alpha_0, \alpha_\infty])}(\Hn)},
\end{equation*}
where the norm of $\gradH[m] u$ on the right-hand side is equal to
\begin{equation*}
\sup_{t\in(0, 1)}(\gradH[m] u)^{**}(t)(1 - \log t)^{\alpha_0}+ \sup_{t\in[1, \infty)}(\gradH[m] u)^{**}(t)(1 + \log t)^{\alpha_\infty};
\end{equation*}
moreover, the restriction on the parameters $\alpha_0$ and $\alpha_\infty$ is unavoidable (see \cref{rem:examples_LZ}\ref{rem:examples_LZ_p=q=infty_nonexistence_of_target_item}). The function space $L^{\infty, \infty; [\alpha_0, \alpha_\infty]}(\Hn)$ is an example of a Lorentz--Zygmund space that is neither a Lebesgue space nor a Lorentz space. However, it is equivalent to the exponential type Orlicz space generated by a Young function that is equivalent to the function $t\mapsto \exp(-t^{-1/\alpha_\infty})$ near $0$ and to the function $t\mapsto \exp(t^{-1/\alpha_0})$ near infinity (when $\alpha_0 = 0$, the Young function is constantly equal to $\infty$ near infinity).

To achieve our goal, not only will we make use of rearrangement techniques, but we will also need to carefully interconnect them with contemporary as well as classical theory of function spaces and some aspects of interpolation theory. We will also need to put a considerable amount of effort into handling various compositions of different Hardy type operators with kernels. Rearrangement techniques were recently successfully used in connection with Sobolev type inequalities in $\Hn$ in \cite{NN:20, N:18, N:20, N:21} (see also references therein), where the main focus was on optimal constants. In this paper, our focus is on optimality of function norms (not of constants). We will obtain very general Sobolev type inequalities in $\Hn$ without imposing various restrictions on the function norms. This generality allows us to handle even delicate limiting situations, of which we have already seen some examples. The tradeoff is, however, that we will obtain our Sobolev inequalities in $\Hn$ with some constants, not comparing them with the optimal ones. The aforementioned papers contain very nice introductory sections on Sobolev type inequalities in the hyperbolic space, and the interested reader is referred to them. Whereas the study of optimal function norms in Sobolev type inequalities in $\Hn$ is new, optimal function norms in Sobolev type inequalities in $\rn$ have been studied for quite a long time (e.g., \cite{ACPS:18, BC:21, CP:16, CKP:08, CPS:15, CPS:20, KP:06, M:21, M:21b} and references therein). As far as the higher order ($m\geq3$) version of \eqref{intro:Sob_ineq} is concerned, we will obtain the optimal function norm in the Sobolev inequality by iterating its optimal lower order versions and by proving that the optimality has been  preserved. The innovative idea of obtaining higher order Sobolev type inequalities by iteration of optimal lower order ones appeared in \cite{CPS:15}, where higher order Sobolev type inequalities on domains in $\rn$ equipped with a finite absolutely continuous measure related to the isoperimetric profile of the domain were thoroughly studied. In this paper, significant difficulties, which we will encounter and will have to deal with, stem from a combination of two facts: The measure of $\Hn$ is infinite, and the differential operator considered in this paper leads to iteration of Hardy operators of different types. In particular, this combination means that we will have to get under our control the behavior of a sum of operators, neither of which rules over the other in general (see \cref{rem:neither_Hardy_is_better_in_general}). Moreover, the infinite measure also means that different function spaces are far less often nested, which makes their intersections and sums considerably more complicated.

We conclude this introductory section by briefly describing the structure of this paper. \cref{sec:prel} contains precise definitions and preliminaries, aimed at making this paper self-contained to a reasonable extent. \cref{sec:main} contains the main result of this paper, which are later proved in \cref{sec:proofs_main}. \cref{sec:aux_results} contains results that are auxiliary in nature for us, equipping and preparing us for proofs of the main results, but readers interested in the theory of rearrangement\hyp{}invariant function spaces and of operators acting on them may find them of independent interest.

\section{Preliminaries}\label{sec:prel}

\subsection{The hyperbolic space}
Throughout the paper, we assume that $n\in\N$, $n\geq2$, and that $\Hn$ stands for the Poincar\'e ball model of the $n$-dimensional hyperbolic space; that is, $\Hn$ is the unit ball in $\rn$ endowed with the Riemannian metric $g_x$ that is at point $x\in\rn$, $|x| < 1$, defined as
\begin{equation*}
g_x(\cdot, \cdot) = \Big( \frac{2}{1 - |x|^2} \Big)^2\langle \cdot, \cdot \rangle.
\end{equation*}
Here $|\cdot|$ and $\langle \cdot, \cdot \rangle$ stand for the Euclidean norm and inner product, respectively. The hyperbolic volume element $\d{\dV}(x)$ is
\begin{equation*}
\d{\dV(x)} = \Big( \frac{2}{1 - |x|^2} \Big)^n \d{x}.
\end{equation*}

We denote by $\distH(x)$ the hyperbolic distance of the point $x\in\Hn$ from the origin; that is, (e.g., \cite[p.~185]{L:18})
\begin{equation*}
\distH(x) = \log\frac{1 + |x|}{1 - |x|} = 2 \arctanh|x|.
\end{equation*}
Furthermore, we denote by $\B_r$ the (open) hyperbolic ball centered at the origin with a radius $r$, that is,
\begin{equation*}
\B_r = \{x\in\Hn\colon \distH(x) < r\},
\end{equation*}
and by $V(r)$ the hyperbolic volume of $\B_r$, that is,
\begin{equation}\label{prel:volume_of_hyperbolic_ball}
	V(r) = n\omega_n\int_0^r (\sinh t)^{n-1}\d{t}.
\end{equation}
Here $\omega_n$ is the (Euclidean) volume of the unit ball in $\rn$. It is easy to see that $V\colon[0, \infty) \to [0,\infty)$ is a $C^\infty$ increasing bijection, whose derivative is positive in $(0, \infty)$. It follows that its inverse function, which we shall denote by $\varrho$, belongs to $\mathcal C^\infty(0, \infty)$. A set $M\subseteq \Hn$ is bounded if $M \subseteq \B_r$ for some $r\in(0, \infty)$.

\subsection{Rearrangement\hyp{}invariant function spaces}
In this subsection, we recall some parts of the theory of rearrangement\hyp{}invariant function spaces, which will be needed in the paper. Throughout this subsection, $(R, \mu)$ is a $\sigma$-finite nonatomic measure space such that $\mu(R) = \infty$. Because it will always be unambiguous what the measure $\mu$ is, it is omitted in the notation.

We set
\begin{align*}
\M(R)&= \{f\colon \text{$f$ is a $\mu$-measurable function on $R$ with values in $[-\infty,\infty]$}\}\\
\intertext{and}
\Mpl(R)&= \{f \in \M(R)\colon f \geq 0\ \text{$\mu$-a.e.~on $R$}\}.
\end{align*}

The \emph{nonincreasing rearrangement} $f^*\colon(0,\infty) \to [0, \infty]$ of a function $f\in \M(R)$  is defined as
\begin{equation*}
f^*(t)=\inf\{\lambda\in(0,\infty): \mu(\{x\in R\colon|f(x)|>\lambda\})\leq t\},\ t\in(0,\infty).
\end{equation*}
The function $f^*$ is nonincreasing and the functions $f$ and $f^*$ are equimeasurable; that is, $\mu(\{x\in R\colon|f(x)|>\lambda\})=|\{t\in (0, \infty)\colon f^*(t)>\lambda\}|$ for every $\lambda\in(0,\infty)$. Here $|\cdot|$ on the right-hand side stands for the Lebesgue measure. 

The \emph{Hardy-Littlewood inequality} (\cite[Chapter~2, Theorem~2.2]{BS}) tells us that, for every $f,g\in\Mpl(R)$,
\begin{equation}\label{prel:ri:HL}
\int_R fg\d{\mu} \leq \int_0^\infty f^*(t)g^*(t)\d{t}.
\end{equation}
In particular, by taking $g=\chi_E$ in \eqref{prel:ri:HL}, we obtain
\begin{equation}\label{prel:ri:HLg=chiE}
\int_E f\d{\mu}\leq\int_0^{\mu(E)}f^*(t)\d{t}
\end{equation}
for every $\mu$-measurable $E\subseteq R$.

The \emph{maximal nonincreasing rearrangement} $f^{**} \colon (0,\infty) \to [0, \infty]$ of a function $f\in \M(R)$ is
defined as
\begin{equation*}
f^{**}(t)=\frac1t\int_0^ t f^{*}(s)\d{s},\ t\in(0,\infty).
\end{equation*}
The function $f^{**}$ is nonincreasing and we have
\begin{equation}\label{prel:ri:twostarsdominateonestar}
f^*\leq f^{**}\quad\text{for every $f\in\M(R)$}.
\end{equation}
Unlike $f\mapsto f^*$, the operator $f\mapsto f^{**}$ is subadditive. More precisely, we have (\cite[Chapter~2,~(3.10)]{BS}):
\begin{equation}\label{prel:ri:subadditivityofdoublestar}
\int_0^t(f+g)^*(s)\d{s}\leq\int_0^tf^*(s)\d{s}+\int_0^tg^*(s)\d{s} \quad \text{for every $t\in(0, \infty)$}
\end{equation}
and every $f,g\in\Mpl(R)$.

A functional $\|\cdot\|_{X(0, \infty)}\colon\Mpl(0,\infty)\to[0,\infty]$ is called a \emph{rearrangement\hyp{}invariant function norm} if, for all $f$, $g$ and $\{f_k\}_{k=1}^\infty$ in $\Mpl(0,\infty)$, and every $\alpha\in[0,\infty)$:
\begin{itemize}
\item[(P1)] $\|f\|_{X(0,\infty)}=0$ if and only if $f=0$ a.e.~on $(0,\infty)$; $\|\alpha f\|_{X(0,\infty)}= \alpha \|f\|_{X(0,\infty)}$;  $\|f+g\|_{X(0,\infty)}\leq \|f\|_{X(0,\infty)} + \|g\|_{X(0,\infty)}$;
\item[(P2)] $\|f\|_{X(0,\infty)}\leq\|g\|_{X(0,\infty)}$ if $ f\leq g$ a.e.~on $(0,\infty)$;
\item[(P3)] $\|f_k\|_{X(0,\infty)} \nearrow \|f\|_{X(0,\infty)}$ if $f_k \nearrow f$ a.e.~on $(0,\infty)$;
\item[(P4)] $\|\chi_E\|_{X(0,\infty)}<\infty$ for every measurable $E\subseteq(0, \infty)$ of finite measure;
\item[(P5)] for every measurable $E\subseteq(0, \infty)$ of finite measure, there is a positive finite constant $C_{E,X}$, possibly depending on $E$ and  $\|\cdot\|_{X(0,\infty)}$ but not on $f$, such that $\int_E f(t)\d{t} \leq C_{E,X} \|f\|_{X(0,\infty)}$;
\item[(P6)] $\|f\|_{X(0,\infty)} = \|f^*\|_{X(0,\infty)}$.
\end{itemize}

With every rearrangement\hyp{}invariant function norm $\|\cdot\|_{X(0,\infty)}$, we associate another functional $\|\cdot\|_{X'(0,\infty)}$ defined as
\begin{equation}\label{prel:ri:normX'}
\|f\|_{X'(0, \infty)}= \sup_{\substack{g\in{\Mpl(0,\infty)}\\\|g\|_{X(0,\infty)}\leq1}}\int_0^\infty f(t)g(t)\d{t},\ f\in\Mpl(0, \infty).
\end{equation}
The functional $\|\cdot\|_{X'(0, \infty)}$ is also a rearrangement\hyp{}invariant function norm (\cite[Chapter~2, Proposition~4.2]{BS}), and it is called the \emph{associate function norm} of $\|\cdot\|_{X(0, \infty)}$. Furthermore, we always have (\cite[Chapter~1, Theorem~2.7]{BS})
\begin{equation}\label{prel:ri:normX''}
\|f\|_{X(0,\infty)}= \sup_{\substack{g\in{\Mpl(0,\infty)}\\\|g\|_{X'(0,\infty)}\leq1}}\int_0^\infty f(t)g(t)\d{t} \quad\text{for every $f\in\Mpl(0,\infty)$},
\end{equation}
that is,
\begin{equation}\label{prel:ri:X''=X}
\|\cdot \|_{(X')'(0, \infty)} = \|\cdot \|_{X(0, \infty)}.
\end{equation}
The supremum in \eqref{prel:ri:normX''} does not change when the functions involved are replaced by their nonincreasing rearrangements (\cite[Chapter~2, Proposition~4.2]{BS}), that is,
\begin{equation}\label{prel:ri:normX''down}
\|f\|_{X(0,\infty)}= \sup_{\substack{g\in{\Mpl(0,\infty)}\\\|g\|_{X'(0,\infty)}\leq1}}\int_0^\infty f^*(t)g^*(t)\d{t} \quad\text{for every $f\in\Mpl(0,\infty)$}.
\end{equation}

Given a rearrangement\hyp{}invariant function norm $\|\cdot\|_{X(0, \infty)}$, we define the \emph{rearrangement\hyp{}invariant function space} $X(R)$ as the linear set
\begin{equation*}
X(R)=\{f\in\M(R)\colon \|f^*\|_{X(0,\infty)}<\infty\}
\end{equation*}
endowed with the norm
\begin{equation*}
\|f\|_{X(R)} = \|f^*\|_{X(0, \infty)},\ f\in\M(R).
\end{equation*}
It is a Banach space (\cite[Chapter~1, Theorem~1.6]{BS}). Note that $f\in\M(R)$ belongs to $X(R)$ if and only if $\|f\|_{X(R)}<\infty$, and that $\|f\|_{X(R)} = \||f|\|_{X(R)}$.

We have 
\begin{equation}\label{prel:ri:holder}
\int_R fg\d{\mu}\leq \|f\|_{X(R)}\|g\|_{X'(R)}\quad\text{for every $f,g\in\Mpl(R)$}.
\end{equation}
We shall refer to \eqref{prel:ri:holder} as the H\"older inequality.

In view of \eqref{prel:ri:X''=X}, rearrangement\hyp{}invariant function spaces are uniquely determined by their associate function spaces. Consequently, statements like ``let $X(R)$ be \emph{the} rearrangement\hyp{}invariant function space whose associate function norm is \dots'' are well justified.

We say that a rearrangement\hyp{}invariant function space $X(R)$ is \emph{embedded in} a rearrangement\hyp{}invariant function space $Y(R)$, and we write $X(R)\hookrightarrow Y(R)$, if there is a positive constant $C$ such that $\|f\|_{Y(R)}\leq C\|f\|_{X(R)}$ for every $f\in\M(R)$. We have (\cite[Chapter~1, Theorem~1.8]{BS}) $X(0, \infty)\hookrightarrow Y(0, \infty)$ if and only if $X(0, \infty)\subseteq Y(0, \infty)$. If $X(R)\hookrightarrow Y(R)$ and $Y(R)\hookrightarrow X(R)$ simultaneously, we write that $X(R)=Y(R)$. We clearly have
\begin{align}
X(R)\hookrightarrow Y(R) \quad&\text{if and only if}\quad X(0, \infty)\hookrightarrow Y(0, \infty). \notag\\
\intertext{Furthermore}
X(0, \infty)\hookrightarrow Y(0, \infty)\quad&\text{if and only if}\quad Y'(0, \infty)\hookrightarrow X'(0, \infty), \label{prel:ri:XtoYiffY'toX'}
\end{align}
with the same embedding constants.

The \emph{dilation operator} $D_a$, $a>0$, defined as
\begin{equation*}
D_af(t)= f\Big( \frac{t}{a} \Big),\ f\in\Mpl(0, \infty),\ t\in(0, \infty),
\end{equation*}
is bounded on every rearrangement\hyp{}invariant function space $X(0,\infty)$. More precisely, we have (\cite[Chapter~3, Proposition~5.11]{BS})
\begin{equation}\label{prel:ri:dilation}
\|D_af\|_{X(0,\infty)}\leq\max\{1,a\}\|f\|_{X(0,\infty)}\quad\text{for every $f\in\M(0,\infty)$}.
\end{equation}

We will also need to consider sums and intersections of rearrangement\hyp{}invariant function spaces on a few occasions. When $\|\cdot\|_{X(0, \infty)}$ and $\|\cdot\|_{Y(0, \infty)}$ are rearrangement\hyp{}invariant function norms, then so are $\|\cdot\|_{X(0, \infty) \cap Y(0, \infty)}$ and $\|\cdot\|_{(X + Y)(0, \infty)}$ defined as
\begin{align*}
\|f\|_{X(0, \infty) \cap Y(0, \infty)} &= \max\{\|f\|_{X(0, \infty)}, \|f\|_{Y(0, \infty)}\},\ f\in\Mpl(0, \infty),
\intertext{and}
\|f\|_{(X + Y)(0, \infty)} &= \inf_{f=g+h}(\|g\|_{X(0, \infty)} + \|h\|_{Y(0, \infty)}),\ f\in\Mpl(0, \infty),
\end{align*}
where the infimum extends over all possible decompositions $f=g+h$, $g,h\in\Mpl(0, \infty)$. Furthermore, we have (e.g., \cite[Lemma~1.12]{CNS:03})
\begin{align}
(X(0, \infty) \cap Y(0, \infty))' = (X' + Y')(0, \infty) \notag \\
\intertext{and}
(X + Y)'(0, \infty) = X'(0, \infty) \cap Y'(0, \infty) \label{prel:ri:dual_sum_and_inter},
\end{align}
with equality of norms.

\subsection{Examples of rearrangement\hyp{}invariant function spaces}
Textbook examples of rearrangement\hyp{}invariant function spaces are the \emph{Lebesgue spaces} $L^p(R)$, $p\in[1, \infty]$. The corresponding rearrangement\hyp{}invariant function norm $\|\cdot\|_{L^p(0, \infty)}$ is defined as
\begin{equation*}
\|f\|_{L^p(0, \infty)} = \begin{cases}
	\Big(\int_0^\infty f(t)^p \d{t}\Big)^\frac1{p}  \quad&\text{if $p\in[1, \infty)$},\\
	\esssup_{t\in(0, \infty)}f(t) \quad&\text{if $p = \infty$},
\end{cases}
\end{equation*}
for $f\in\Mpl(0, \infty)$. The layer cake formula (e.g., \cite[Chapter~2, Proposition~1.8]{BS}) shows that the Lebesgue norm $\|\cdot\|_{L^p(R)}$ as defined here (i.e., $\|f\|_{L^p(R)} = \|f^*\|_{L^p(0, \infty)}$) coincides with its standard definition without rearrangements.

The \emph{Lorentz spaces} $L^{p, q}(R)$ are an important generalization of Lebesgue spaces, where either $p\in(1, \infty)$ and $q\in[1, \infty]$ or $p = q = 1$ or $p = q = \infty$. The corresponding rearrangement\hyp{}invariant function norm $\|\cdot\|_{L^{p,q}(0, \infty)}$ is defined as
\begin{equation*}
\|f\|_{L^{p,q}(0, \infty)} = \|t^{\frac1{p} - \frac1{q}}f^*(t)\|_{L^q(0, \infty)},\ f\in\Mpl(0, \infty).
\end{equation*}
However, we need to be more careful here. The functional $\|\cdot\|_{L^{p,q}(0, \infty)}$ is not a rearrangement\hyp{}invariant function norm if $1< p < q \leq\infty$, because it is not subadditive. When this is the case, the functional $\|\cdot\|_{L^{p,q}(0, \infty)}$ is merely \emph{equivalent to a rearrangement\hyp{}invariant function norm}. More precisely, when either $p\in(1, \infty)$ and $q\in[1 ,\infty]$ or $p=q=\infty$, there are positive constants $C_1$ and $C_2$ such that
\begin{equation*}
C_1\|f\|_{L^{(p,q)}(0, \infty)} \leq \|f\|_{L^{p,q}(0, \infty)} \leq C_2\|f\|_{L^{(p,q)}(0, \infty)} \quad \text{for every $f\in\Mpl(0, \infty)$},
\end{equation*}
and the functional $\|f\|_{L^{(p,q)}(0, \infty)} = \|f^{**}\|_{L^{p,q}(0, \infty)}$, $f\in\Mpl(0, \infty)$, is a rearrangement\hyp{}invariant function norm. The interested reader can find more information in \cite[Chapter~4, Section~4]{BS}. Since we will usually not look for exact values of constants, we will use $\|\cdot\|_{L^{p,q}(0, \infty)}$ instead of $\|\cdot\|_{L^{(p,q)}(0, \infty)}$ even when $1< p < q \leq\infty$. Note that we have
\begin{equation*}
\|\cdot\|_{L^p(0, \infty)} = \|\cdot\|_{L^{p,p}(0, \infty)} \quad \text{for $p\in[1, \infty]$}.
\end{equation*}
Furthermore, when $p\in(1, \infty)$ and $1\leq q_1 < q_2\leq \infty$, we have
\begin{equation*}
L^{p, q_1}(R) \subsetneq L^{p, q_2}(R).
\end{equation*}

The Lorentz spaces $L^{p,q}(R)$ belong to the class of \emph{Lorentz--Zygmund spaces} $L^{p, q; \A, \B}(R)$, where $p,q\in[1, \infty]$, $\A = [\alpha_0, \alpha_\infty]\in\R^2$, and $\B = [\beta_0, \beta_\infty]\in\R^2$. Let $\ell$ and $\ell\ell$ be the functions defined as $\ell(t)=1+|\log t|$ and $\ell\ell(t)=\ell(\ell(t))$, $t\in(0,\infty)$. The functional $\|\cdot\|_{L^{p,q;\A, \B}(0, \infty)}$ is defined as
\begin{equation*}
\|f\|_{L^{p,q;\A, \B}(0, \infty)}=\left\|t^{\frac1{p}-\frac1{q}}\ell^\A(t)\ell\ell^\B(t)f^*(t)\right\|_{L^q(0,\infty)},\ f\in\Mpl(0, \infty),
\end{equation*}
where the functions $\ell^\A$ and $\ell\ell^\B$ are defined as
\begin{equation*}
\ell^\A(t)=\begin{cases}
\ell^{\alpha_0}(t),\ &t\in(0,1),\\
\ell^{\alpha_\infty}(t),\ &t\in[1,\infty),
\end{cases}
\end{equation*}
and
\begin{equation*}
\ell\ell^\A(t)=\begin{cases}
\ell\ell^{\alpha_0}(t),\ &t\in(0,1),\\
\ell\ell^{\alpha_\infty}(t),\ &t\in[1,\infty).
\end{cases}
\end{equation*}
As with Lorentz spaces, there is also functional $\|\cdot\|_{L^{(p,q;\A, \B)}(0, \infty)}$ defined as
\begin{equation*}
\|f\|_{L^{(p,q;\A, \B)}(0, \infty)} = \|f^{**}\|_{L^{p,q;\A, \B}(0, \infty)},\ f\in\Mpl(0, \infty).
\end{equation*}
When $p\in(1, \infty]$, the functionals $\|\cdot\|_{L^{p,q;\A, \B}(0, \infty)}$ and $\|\cdot\|_{L^{(p,q;\A, \B)}(0, \infty)}$ are equivalent. When $\A = \B = [0,0]$, we have $\|\cdot\|_{L^{p,q;\A, \B}(0, \infty)} = \|\cdot\|_{L^{p,q}(0, \infty)}$. When $\B = [0, 0]$, we will omit $\B$ in the notation. Not only does the class of Lorentz--Zygmund spaces include Lorentz (and Lebesgue) spaces, but it also includes some \emph{Orlicz spaces}; namely those of logarithmic ($p=q\in[1, \infty)$) and exponential ($p = q = \infty$) types (see \cite[Section~8]{OP:99} for more detail). The functional $\|\cdot\|_{L^{p,q;\A}(0, \infty)}$ is equivalent to a rearrangement\hyp{}invariant function norm if and only if (\cite[Theorem~7.1]{OP:99}) one of the conditions
\begin{itemize}
\item $p = q = 1$, $\alpha_0 \geq 0$, and $\alpha_\infty \leq 0$;
\item $p\in(1, \infty)$ and $q\in[1, \infty]$;
\item $p=\infty$, $q\in[1, \infty)$, and $\alpha_0+\frac1{q}<0$;
\item $p=q=\infty$ and $\alpha_0\leq 0$
\end{itemize}
is satisfied. When we say that $X(R) = L^{p,q; \A}(R)$ is a Lorentz--Zygmund space, we will implicitly assume that one of the conditions is satisfied and treat $X(R)$ as a rearrangement\hyp{}invariant function space. Assuming that either one of the first two conditions is satisfied or $p=q=\infty$, $\alpha_0\leq0$, and $\alpha_\infty\geq0$, we have (\cite[Theorems~6.2~and~6.6]{OP:99})
\begin{equation}\label{thm:examples_LZ_X_asoc}
(L^{p, q; \A}(0, \infty))' = L^{p', q'; -\A}(0, \infty).
\end{equation}

Lorentz--Zygmund spaces $L^{p, q; \A, \B}(R)$ are special instances of \emph{classical Lorentz spaces} $\Lambda^q_v(R)$, $q\in[1, \infty]$, for suitable choices of the weight function $v$. Let $q\in[1, \infty]$ and $v\in\Mpl(0, \infty)$. The functional $\|\cdot\|_{\Lambda^q_v(0, \infty)}$ is defined as
\begin{equation*}
\|f\|_{\Lambda^q_v(0, \infty)} = \|f^*v\|_{L^q(0, \infty)},\ f\in\Mpl(0, \infty).
\end{equation*}
For a general weight $v$, the functional $\|\cdot\|_{\Lambda^q_v(0, \infty)}$ need not be a rearrangement\hyp{}invariant function norm nor be equivalent to one. The interested reader can find characterizations when $\|\cdot\|_{\Lambda^q_v(0, \infty)}$ is equivalent to a rearrangement\hyp{}invariant function norm in \citep{S:90} ($q\in(1,\infty)$), \citep{CGS:96} ($q=1$) or \citep{GS:14} ($q\in(1,\infty]$).

\subsection{Hardy operators}\label{sec:prel:sub:hardy}
Let $\alpha\in[0, n)$. We define the function $\phi_\alpha\colon (0, \infty) \to (0, \infty)$ as
\begin{equation*}
\phi_\alpha(t) = \min\{t^{-1 + \frac{\alpha}{n}}, t^{-1}\},\ t\in(0, \infty).
\end{equation*}
We also define the operators $R_\alpha$ and $H_\alpha$, which map $\Mpl(0, \infty)$ into $\Mpl(0, \infty)$, as
\begin{align*}
R_\alpha f(t) &= \phi_\alpha (t)\int_0^t f(s)\d{s} ,\ t\in(0, \infty), \\
\intertext{and}
H_\alpha f(t) &= \int_t^\infty f(s) \phi_\alpha (s)\d{s},\ t\in(0, \infty),
\end{align*}
where $f\in\Mpl(0, \infty)$. The operators $R_\alpha$ and $H_\alpha$ are formally adjoint in the sense that
\begin{equation}\label{prel:ri:Rk_Hk_self_adjoint}
\int_0^\infty f(t) R_\alpha g(t) \d{t} = \int_0^\infty H_\alpha f(t) g(t) \d{t} \quad \text{for every $f,g\in\Mpl(0, \infty)$}.
\end{equation}
When $\alpha = 0$, we set (cf.~\cite[pp.~150-154]{BS})
\begin{equation*}
P = R_0 \quad \text{and} \quad Q = H_0.
\end{equation*}
Note that we have $f^{**} = Pf^*$ for every $f\in\Mpl(0, \infty)$.

The adjointness of $P$ and $Q$ combined with \eqref{prel:ri:normX'} implies that
\begin{equation}\label{prel:ri:P_bounded_iff_Q}
\parbox{0.85\textwidth}{the operator $P$ is bounded on $X(0, \infty)$ if and only if the operator $Q$ is bounded on $X'(0, \infty)$;}
\end{equation}
moreover, their operator norms are the same. Furthermore, the boundedness of either on a rearrangement\hyp{}invariant function space over $(0, \infty)$ is equivalent to the boundedness of its version restricted to nonincreasing functions. More precisely, it follows from the Hardy--Littlewood inequality \eqref{prel:ri:HL} that
\begin{equation}\label{prel:ri:P_bounded_iff_double_star}
\parbox{0.85\textwidth}{the operator $P$ is bounded on $X(0, \infty)$ if and only if the operator $f \mapsto f^{**}$ is bounded on $X(0, \infty)$;}
\end{equation}
moreover, their operator norms are the same. As for the operator $Q$, it follows from \eqref{prel:ri:Rk_Hk_self_adjoint} combined with \eqref{prel:ri:normX''down} and the fact that the function $Pf^*$ is nonincreasing that
\begin{equation}\label{prel:ri:Q_bounded_iff_Qfstar}
\parbox{0.85\textwidth}{the operator $Q$ is bounded on $X(0, \infty)$ if and only if the operator $f \mapsto Qf^*$ is bounded on $X(0, \infty)$;}
\end{equation}
moreover, their operator norms are the same.

\subsection{Sobolev spaces in the hyperbolic space}
Let $m\in\N$. The \emph{$m$th order hyperbolic gradient} $\gradH[m]$ is defined by \eqref{intro:mth_order_gradient}, where we formally set $\lapH[0]u = u$ and $\lapH[j] = \lapH(\lapH[j-1])$, $j\in\N$. Let $X(\Hn)$ be a rearrangement\hyp{}invariant function space. When $m$ is odd, $\gradH[m]u\in X(\Hn)$ means $\absH{\gradH[m]u}\in X(\Hn)$, and we shall write $\|\gradH[m]u\|_{X(\Hn)}$ for short, instead of $\|\absH{\gradH[m]u}\|_{X(\Hn)}$. We define the \emph{$m$th order Sobolev space $V^m X(\Hn)$} and its subspace $V_0^m X(\Hn)$, consisting of those functions that vanish in a suitable sense, as follows. When $m = 1$, we define
\begin{equation*}
V^1 X(\Hn) = \big\{ u: \text{$u$ is weakly differentiable and $\absH{\gradH u}\in X(\Hn)$} \big\}
\end{equation*}
and
\begin{align*}
V_0^1 X(\Hn) = \Big\{ u\in V^1 X(\Hn):\ &Eu\ \text{is weakly differentiable on $\rn$}\ \text{and}\\
&\dV(\{x\in\Hn: |u(x)| > \lambda\}) < \infty \quad \text{for every $\lambda>0$} \Big\},
\end{align*}
where $Eu$ is the extension of $u$ outside $\Hn$ by $0$. When $m\geq2$, we define
\begin{align*}
V^m X(\Hn) = \Big\{ u:\ &\text{$u$ is $m$-times weakly differentiable}, \gradH[m]u\in X(\Hn),\ \text{and}\\
&\text{$\absH{\gradH(\lapH[j-1]u)}, \lapH[j]u \in L^2_{loc}(\Hn)$ for every $j = 1,\dots, \Big\lfloor \frac{m}{2} \Big\rfloor$}
\Big\}
\end{align*}
and
\begin{align*}
V_0^m X(\Hn) = \Big\{ u\in V^m X(\Hn):\ &\{x\in\Hn: |\lapH[j]u(x)|>\lambda\}\ \text{is bounded}\\
&\text{for all $\lambda>0$ and $j = 0, \dots, \Big\lceil \frac{m}{2} - 1 \Big\rceil$} \Big\}.
\end{align*}
We have $\mathcal C_0^m(\Hn)\subseteq V_0^m X(\Hn)$. Note that our definition of the Sobolev spaces $V^m X(\Hn)$ and $V_0^m X(\Hn)$ for $m\geq2$ is slightly more restrictive than the corresponding definition of suitable higher-order Sobolev spaces on $\rn$ when the full gradient is considered (see~\cite{ACPS:18, CP:98, M:21}). Our definition of $V^m X(\Hn)$ and $V_0^m X(\Hn)$ for $m\geq2$ is such that we may utilize the following pointwise estimate on the nonincreasing rearrangement of $u$. For every $u\in V_0^2 X(\Hn)$, we have (\citep[Proof of Proposition~2.2]{NN:20}, cf.~\cite{C:04})
\begin{equation}\label{prel:potential_estimate}
u^*(t) \leq \int_t^\infty \frac{s(-\lapH u)^{**}(s)}{(n\omega_n \sinh(\varrho(s))^{n-1})^2}\d{s} \quad \text{for every $t\in(0,\infty)$},
\end{equation}
in which $\varrho$ is the inverse function to the function $V$ defined by \eqref{prel:volume_of_hyperbolic_ball}.

\subsection*{Notation} We write $A\lesssim B$, where $A$ and $B$ are nonnegative expressions, if there is a positive constant $c$ such that $A\leq c\cdot B$. The constant $c$ may depend on some parameters in the expression. If not stated explicitly,  what the constant may depend on and what it may not should be obvious from the context. We also write $A\gtrsim B$ with the obvious meaning, and $A\approx B$ when $A\lesssim B$ and $A\gtrsim B$ simultaneously.

\section{Main Results}\label{sec:main}
Before we exhibit the main results of this paper, we first need to make a few definitions. Provided that it makes sense, we will denote be $T^j$ the $j$th iterate of an operator $T$; that is,
\begin{equation*}
T^j = \overbrace{T\circ\cdots\circ T}^{\text{$j$ times}}.
\end{equation*}
The zeroth iterate is to be interpreted as the identity operator.

Given $m\in\N$, $m < n$, we define the operators $\Tm$ and $\Sm$ as
\begin{align}
\Tm &= \begin{cases}
	(H_2\circ P)^k \circ H_1 \quad &\text{if $m$ is odd},\\
	(H_2\circ P)^{k + 1} \quad &\text{if $m$ is even},	
\end{cases} \label{operator_Tm_def}\\
\intertext{and}
\Sm &= \begin{cases}
		H_2^k \circ H_1 \quad &\text{if $m$ is odd},\\
		H_2^{k + 1} \quad &\text{if $m$ is even},
	\end{cases} \label{operator_Sm_def}
\end{align}
where
\begin{equation}\label{E:def_k}
k = \Big\lceil \frac{m}{2} - 1 \Big\rceil.
\end{equation}
The letter $k$ will always mean the same in the rest. Recall that the operators used in the definitions of $\Tm$ and $\Sm$ were defined in \cref{sec:prel:sub:hardy}.

For each $j\in\N_0$, we define a nonnegative kernel $K_j(\cdot,\cdot)$ on $\{(a,b)\in(0, \infty)^2: a\leq b\}$ as
\begin{equation}\label{E:def_kernel_K}
K_j(a,b) = \begin{cases}
1\quad \quad &\text{if $0<a\leq b \leq 1$}, \\
\log(eb)^j \quad &\text{if $0<a\leq 1 \leq b$}, \\
\log(\frac{b}{a})^j \quad &\text{if $1< a \leq b$}.
\end{cases}
\end{equation}
Note that
\begin{equation}\label{E:kernel_K_monotonicity_first_var}
\text{the function $(0, b]\ni a\mapsto K_j(a,b)$ is nonincreasing for each $b>0$}
\end{equation}
and
\begin{equation}\label{E:kernel_K_monotonicity_second_var}
\text{the function $[a, \infty)\ni b\mapsto K_j(a,b)$ is nondecreasing for each $a>0$}.
\end{equation}

We say that a rearrangement\hyp{}invariant function space $Y(\Hn)$ is \emph{the optimal target space for $X(\Hn)$ in \eqref{intro:Sob_ineq}} if $Y(\Hn)$ is the smallest rearrangement\hyp{}invariant function space that makes the inequality valid. More precisely, the inequality is valid with $X(\Hn)$ and $Y(\Hn)$, and, if the inequality is valid with $Y(\Hn)$ replaced by a rearrangement\hyp{}invariant function space $Z(\Hn)$, then $Y(\Hn)\hookrightarrow Z(\Hn)$. In other words, $\|\cdot\|_{Y(\Hn)}$ is the strongest (up to equivalences) rearrangement\hyp{}invariant function norm on the left-hand side of \eqref{intro:Sob_ineq} with which the inequality is still valid.

Our first main result is the following so-called reduction principle for the Sobolev type inequality \eqref{intro:Sob_ineq}. It rephrases the question of whether \eqref{intro:Sob_ineq} is valid in terms of boundedness of suitable operators.
\begin{theorem}\label{thm:reduction_principle}
Let $\|\cdot\|_{X(0, \infty)}$ and $\|\cdot\|_{Y(0, \infty)}$ be rearrangement\hyp{}invariant function norms, $m\in\N$, $m < n$. Consider the following three statements.
\begin{enumerate}[label=(\roman*), ref=(\roman*)]
	\item\label{thm:reduction_principle_Sobolev_item} There is a constant $C_1 > 0$ such that
		\begin{equation}\label{thm:reduction_principle_Sob_ineq}
			\|u\|_{Y(\Hn)} \leq C_1 \|\gradH[m] u\|_{X(\Hn)} \quad \text{for every $u \in V_0^m X(\Hn)$}.
		\end{equation}
	\item\label{thm:reduction_principle_Hardy_item} There is a constant $C_2 > 0$ such that
		\begin{equation}\label{thm:reduction_principle_Hardy_ineq}
			\|\Tm f\|_{Y(0, \infty)} \leq C_2 \|f\|_{X(0, \infty)} \quad \text{for every $f\in\Mpl(0, \infty)$}.
		\end{equation}
	\item\label{thm:reduction_principle_Hardy_simplified_item} There is a constant $C_3 > 0$ such that
	\begin{equation}\label{thm:reduction_principle_Hardy_ineq_simplified}
			\|\Sm f\|_{Y(0, \infty)} \leq C_3 \|f\|_{X(0, \infty)} \quad \text{for every $f\in\Mpl(0, \infty)$}.
		\end{equation}
\end{enumerate}
The statements \ref{thm:reduction_principle_Sobolev_item} and \ref{thm:reduction_principle_Hardy_item} are equivalent. They imply \ref{thm:reduction_principle_Hardy_simplified_item}. If either $m\geq2$ and the operator $f\mapsto f^{**}$ is bounded on $X(0, \infty)$ or $m = 1$, then \ref{thm:reduction_principle_Hardy_simplified_item} implies the first two statements; in other words, if that is the case, then all three statements are equivalent.
\end{theorem}

\begin{remark}\label{rem:reduction_principle_Hardy_ineq_equiv_to_reduced}\ 
\begin{enumerate}[label=(\roman*), ref=(\roman*)]
\item\label{rem:reduction_principle_Hardy_ineq_equiv_to_reduced_item} When $m$ is even, the validity of \eqref{thm:reduction_principle_Hardy_ineq} is equivalent to the validity of the same inequality restricted only to nonincreasing functions $f\in\Mpl(0, \infty)$ (with the same multiplicative constant). That follows immediately from the Hardy--Littlewood inequality \eqref{prel:ri:HL}. Remarkably enough, the same is true even when $m$ is odd\textemdash this time we need to increase the constant in the unrestricted version, however. We will see that in the proof of \cref{thm:reduction_principle}. To achieve that, we will use the remarkable result \citep[Theorem~9.5]{CPS:15} and its generalization \citep[Theorem~3.10]{Pe:20} to the interval $(0, \infty)$. The result ensures that, for $\alpha,\beta\in[0, n)$,
\begin{align*}
\|R_\alpha(R_\beta f^*)^*\|_{X'(0, \infty)} &\approx \|(R_\alpha\circ R_\beta) f^*\|_{X'(0, \infty)} \\
&\approx  \sup_{\substack{g\in\Mpl(0, \infty)\\ \|g\|_{X(0, \infty)}\leq1}} \int_0^\infty g^*(t) (R_\alpha\circ R_\beta) f^*(t) \d{t}
\end{align*}
for every $f\in\Mpl(0, \infty)$ and every rearrangement\hyp{}invariant function norm $\|\cdot\|_{X(0, \infty)}$, despite the fact that the functions $R_\beta f^*$ and $(R_\alpha\circ R_\beta) f^*$ are hardly ever nonincreasing. Furthermore, such an equivalence of restricted and unrestricted versions is also true for \eqref{thm:reduction_principle_Hardy_ineq_simplified}. In the case of \eqref{thm:reduction_principle_Hardy_ineq_simplified}, the Hardy--Littlewood inequality is not enough even for $m$ even.
\item When $m\geq3$, the operators $\Tm$ and $\Sm$ are defined as iterations of operators. Nevertheless, we may express their rearrangement\hyp{}invariant function norms in closed forms, up to equivalence, by means of operators with kernels.
Let $\|\cdot\|_{Y(0, \infty)}$ be a rearrangement\hyp{}invariant function norm. It follows from \cref{prop:properties_of_iterated_operators,prop:iterated_operators_closed_form} that, for every $f\in\Mpl(0, \infty)$,
\begin{align*}
\|\Tm f\|_{Y(0, \infty)} &\approx \Big\| \phi_m(t) \int_0^t f(s) K_k(s,t) \d{s} \Big\|_{Y(0, \infty)} \\
&\quad+ \Big\| \int_t^\infty Pf(s) \phi_m(s) K_k(t,s) \d{s} \Big\|_{Y(0, \infty)}
\end{align*}
when $m\geq4$ is even and
\begin{align*}
\|\Tm f\|_{Y(0, \infty)} &\approx \Big\| \phi_{m-1}(t) \int_0^t f(s) \phi_1(s)s K_{k-1}(s,t)  \d{s} \Big\|_{Y(0, \infty)}\\
&\quad+ \Big\| \int_t^\infty f(s) \phi_m(s) K_k(t,s) \d{s} \Big\|_{Y(0, \infty)}
\end{align*}
when $m\geq3$ is odd; and that
\begin{equation*}
\|\Sm f\|_{Y(0, \infty)} \approx \Big\| \int_t^\infty f(s) \phi_m(s) K_k(t,s) \d{s} \Big\|_{Y(0, \infty)},
\end{equation*}
where $k$ is defined by \eqref{E:def_k}. The multiplicative constants depend only on $m$ and $n$.
\end{enumerate}
\end{remark}

We now exhibit several theorems describing the optimal target space in the Sobolev inequality \eqref{intro:Sob_ineq}. We do not start with the most general one (\cref{thm:optimal_target_general_theorem}), however. Instead, we start with less general ones. Even though they impose some unnecessary assumptions on $X(\Hn)$, they can often be used instead of the most general one and provide simpler descriptions of the optimal target space. In general, their extra assumptions are not mutually exclusive, and so they provide different equivalent descriptions of the optimal target space in a lot of situations. The following theorem can be used either when $m = 1$ or when, loosely speaking, $X(\Hn)$ is not ``too close to $L^1(\Hn)$''.
\begin{theorem}\label{thm:optimal_target_general_theorem_simplified_M_bounded_X}
Let $\|\cdot\|_{X(0, \infty)}$ be a rearrangement\hyp{}invariant function norm, $m\in\N$, $m < n$. Assume that
\begin{equation}\label{thm:optimal_target_general_theorem_condition}
\frac{(1 + \log t)^{\lceil m/2 - 1\rceil}}{t}\chi_{(1, \infty)}(t) \in X'(0, \infty)
\end{equation}
and that one of the conditions
\begin{itemize}
\item $m = 1$;
\item $m \geq 2$ and the operator $f\mapsto f^{**}$ is bounded on $X(0, \infty)$
\end{itemize}
is satisfied. Set $k = \lceil m/2 - 1 \rceil$ and 
\begin{equation}\label{thm:optimal_target_general_theorem_beta}
\beta = \begin{cases}1 \quad &\text{if $m$ is odd},\\
	2 \quad &\text{if $m$ is even}.
\end{cases}
\end{equation}
Let $Y_{m, X}(\Hn)$ be the rearrangement\hyp{}invariant function space whose associate function norm is defined as
\begin{equation}\label{thm:optimal_target_general_theorem_simplified_M_bounded_X_asoc_norm}
\|g\|_{Y_{m, X}'(0, \infty)} = \begin{cases}
	\|R_m g^*\|_{X'(0, \infty)} \quad &\text{if $m\in\{1, 2\}$},\\
	\|(R_\beta \circ R_2^k) g^*\|_{X'(0, \infty)} \quad &\text{if $m\geq3$},
\end{cases}
\end{equation}
for $g\in\Mpl(0, \infty)$. The rearrangement\hyp{}invariant function space $Y_{m, X}(\Hn)$ is the optimal target space for $X(\Hn)$ in \eqref{intro:Sob_ineq}.

Furthermore, if \eqref{thm:optimal_target_general_theorem_condition} is not satisfied, there is no rearrangement\hyp{}invariant function space $Y(\Hn)$ for which \eqref{intro:Sob_ineq} with $X(\Hn)$ is valid.
\end{theorem}

\begin{remark}\ 
\begin{enumerate}[label=(\roman*), ref=(\roman*)]
\item The condition \eqref{thm:optimal_target_general_theorem_condition} is necessary for the existence of any rearrangement\hyp{}invariant target function space $Y(\Hn)$ in \eqref{intro:Sob_ineq}. When $X(\Hn) = L^p(\Hn)$ is a Lebesgue space (or, more generally, a Lorentz space $L^{p, q}(\Hn)$), it is satisfied if and only if $p\in[1, \infty)$.
\item When $X(\Hn) = L^p(\Hn)$ is a Lebesgue space (or, more generally, a Lorentz space $L^{p, q}(\Hn)$), the operator $f\mapsto f^{**}$ is bounded on $X(0,\infty)$ if and only if $p\in(1, \infty]$.
\item The boundedness of the operator $f\mapsto f^{**}$ on $X(0, \infty)$ does not imply the validity of \eqref{thm:optimal_target_general_theorem_condition} (e.g., consider $X(0, \infty) = L^\infty(0, \infty)$).
\item When $m\geq3$, the rearrangement\hyp{}invariant function norm defined by \eqref{thm:optimal_target_general_theorem_simplified_M_bounded_X_asoc_norm} is equivalent to a rearrangement\hyp{}invariant function norm defined as
\begin{equation*}
\Mpl(0,\infty) \ni g \mapsto \Big\| \phi_m(t) \int_0^t g^*(s) K_k(s,t) \d{s} \Big\|_{X'(0, \infty)}.
\end{equation*}
See \cref{rem:operator_induced_norms_closed_forms}\ref{rem:operator_induced_norms_closed_forms:item} for more detail.
\end{enumerate}
\end{remark}

Now is the time to finally see some concrete examples. In the following list of examples, $X(\Hn)$ is a Lorentz--Zygmund space $L^{p, q; \A}(\Hn)$. Recall that Lebesgue spaces $L^p(\Hn)$ coincide with $L^{p, p; [0, 0]}(\Hn)$, and Lorentz spaces $L^{p,q}(\Hn)$ coincide with $L^{p, q; [0, 0]}(\Hn)$. Although the following theorem is shown now, all theorems describing the optimal target space in \eqref{intro:Sob_ineq} from this section are used in its proof.
\begin{theorem}\label{thm:examples_LZ}
Let $X(\Hn) = L^{p, q; \A}(\Hn)$ be a Lorentz--Zygmund space, where $\A = [\alpha_0, \alpha_\infty]$. Let $m\in\N$, $m < n$. Set
\begin{align*}
Y_{m, X}(\Hn) = \begin{cases}
\Lambda^q_{v_1}(\Hn) \quad &\text{if either $m = 1$, $p=q=1$, $\alpha_0\geq0$, and $\alpha_\infty\leq 0$} \\
	&\text{\hphantom{if }or $p\in(1, \frac{n}{m})$}, \\
Z_1(\Hn) \quad &\text{if $m\geq3$ is odd, $p=q=1$, and $\alpha_0=\alpha_\infty=0$}, \\
Z_2(\Hn) \quad &\text{if $m$ is even, $p=q=1$, and $\alpha_0=\alpha_\infty=0$}, \\
\Lambda^q_{v_2}(\Hn) \quad &\text{if $p = \frac{n}{m}$ and $\alpha_0 < \frac1{q'}$}, \\
\Lambda^q_{v_3}(\Hn) \quad &\text{if $p = \frac{n}{m}$, $q\in(1, \infty]$, and $\alpha_0 = \frac1{q'}$}, \\
L^\infty(\Hn) \cap X(\Hn) \quad &\text{if either $p = \frac{n}{m}$, $q = 1$, and $\alpha_0 \geq0$}, \\
	&\text{\hphantom{if }or $p = \frac{n}{m}$, $q \in (1, \infty]$, and $\alpha_0 > \frac1{q'}$}, \\
	&\text{\hphantom{if }or $p\in(\frac{n}{m}, \infty)$}, \\
Z_3(\Hn) \quad &\text{if $p=q=\infty$, $\alpha_0\leq 0$, and $\alpha_\infty > \lceil \frac{m}{2} \rceil$}, \\
\end{cases}
\end{align*}
where
\begin{align*}
v_1(t) &= t^{\frac{n- mp}{np} - \frac1{q}}\ell^{\alpha_0}(t)\chi_{(0, 1)}(t) + t^{\frac1{p} - \frac1{q}}\ell^{\alpha_\infty}(t)\chi_{[1, \infty)}(t), \\
v_2(t) &= t^{-\frac1{q}}\ell^{\alpha_0 - 1}(t)\chi_{(0, 1)}(t) + t^{\frac{m}{n} - \frac1{q}}\ell^{\alpha_\infty}(t)\chi_{[1, \infty)}(t), \\
v_3(t) &= t^{-\frac1{q}}\ell^{-\frac1{q}}(t)\ell\ell^{-1}(t)\chi_{(0, 1)}(t) + t^{\frac{m}{n} - \frac1{q}}\ell^{\alpha_\infty}(t)\chi_{[1, \infty)}(t), \\
\intertext{for $t\in(0, \infty)$, and where}
\|f\|_{Z_1(0, \infty)} &= \int_0^1 t^{\frac{n - m}{n} - 1} f^*(t) \d{t} + \sup_{t\in[1, \infty)} t\ell^{-\frac{m - 1}{2}}(t)f^{**}(t), \\
\|f\|_{Z_2(0, \infty)} &= \sup_{t\in(0, 1)} t^{\frac{n - m}{n}} f^{**}(t) + \sup_{t\in[1, \infty)} t\ell^{-\frac{m}{2}}(t)f^{**}(t), \\
\|f\|_{Z_3(0, \infty)} &= \|f\|_{L^\infty(0, \infty)} + \sup_{t\in[1, \infty)} \ell^{\alpha_\infty - \lceil \frac{m}{2} \rceil}(t)f^{**}(t),
\end{align*}
for $f\in\Mpl(0, \infty)$.

The rearrangement\hyp{}invariant function space $Y_{m, X}(\Hn)$ is the optimal target space for $X(\Hn)$ in \eqref{intro:Sob_ineq}.
\end{theorem}

\begin{remark}\label{rem:examples_LZ}\ 
\begin{enumerate}[label=(\roman*), ref=(\roman*)]
\item It can be verified that the functionals $\|\cdot\|_{Z_j(0, \infty)}$, $j\in\{1, 2, 3\}$, are really rearrangement\hyp{}invariant function norms. Furthermore, it follows from weighted Hardy's inequality \cite[Theorem~1]{M:72} that the norm $\|\cdot\|_{Z_3(0, \infty)}$ is equivalent to the functional
\begin{equation*}
\Mpl(0, \infty) \ni f\mapsto \|f\|_{L^\infty(0, \infty)} + \sup_{t\in[1, \infty)} \ell^{\alpha_\infty - k - 1}(t)f^{*}(t).
\end{equation*}
As for $\|\cdot\|_{Z_1(0, \infty)}$ and $\|\cdot\|_{Z_2(0, \infty)}$, it is possible to equivalently swap $f^*$ and $f^{**}$ back and forth in the first terms, but it is not possible in the second ones. Moreover, $\|\cdot\|_{Z_1(0, \infty)}$ is still a norm after the swap, but $\|\cdot\|_{Z_2(0, \infty)}$ is not subadditive anymore after the swap.
\item The optimal function spaces from \cref{thm:examples_LZ} can be expressed as intersections of rearrangement\hyp{}invariant function spaces\textemdash namely
\begin{align*}
\Lambda^q_{v_1}(\Hn) &= L^{\frac{np}{n - mp}, q; [\alpha_0, \alpha_\infty]}(\Hn) \cap L^{p,q; [\alpha_0, \alpha_\infty]}(\Hn), \\
Z_1(\Hn) &= L^{\frac{n}{n - m}, 1}(\Hn) \cap L^{(1, \infty; [0, -\frac{m - 1}{2}])}(\Hn), \\
Z_2(\Hn) &= L^{\frac{n}{n - m}, \infty}(\Hn) \cap L^{(1, \infty; [0, -\frac{m}{2}])}(\Hn), \\
\Lambda^q_{v_2}(\Hn) &= L^{\infty, q; [\alpha_0 - 1, \alpha_\infty]}(\Hn) \cap L^{\frac{n}{m}, q; [\alpha_0, \alpha_\infty]}(\Hn), \\
\Lambda^q_{v_3}(\Hn) &= L^{\infty, q; [-\frac1{q}, \alpha_\infty], [-1, 0]}(\Hn) \cap L^{\frac{n}{m}, q; [1-\frac1{q}, \alpha_\infty]}(\Hn), \\
Z_3(\Hn) &= L^\infty(\Hn) \cap L^{\infty, \infty; [\alpha_0, \alpha_\infty - \lceil \frac{m}{2} \rceil]}(\Hn).
\end{align*}
The description of $Y_{m, X}(\Hn)$ given in \cref{thm:examples_LZ} is in a sense ``more revealing'', however, because it provides a clearer insight into how the optimal norm behaves locally and globally.

\item\label{rem:examples_LZ_p=q=infty_nonexistence_of_target_item}  When $\alpha_0\leq0$ (otherwise $L^{\infty, \infty; [\alpha_0, \alpha_\infty]}(\Hn) = \{0\}$) and $\alpha_\infty\leq \lceil m/2 \rceil$, there is no rearrangement\hyp{}invariant function space $Y(\Hn)$ for which \eqref{intro:Sob_ineq} with $X(\Hn) = L^{\infty, \infty; [\alpha_0, \alpha_\infty]}(\Hn)$ is valid. The reason is that the condition \eqref{thm:optimal_target_general_theorem_condition} is not satisfied. In particular, there is no rearrangement\hyp{}invariant function space $Y(\Hn)$ for which \eqref{intro:Sob_ineq} with $X(\Hn) = L^\infty(\Hn)$ is valid.
\end{enumerate}
\end{remark}

The following description of the optimal target space in \eqref{intro:Sob_ineq} can be used either when $m = 2$ or when, loosely speaking, $X(\Hn)$ is not ``too close to $L^\infty(\Hn)$''.
\begin{theorem}\label{thm:optimal_target_general_theorem_simplified_M_bounded_Xasoc}
Let $\|\cdot\|_{X(0, \infty)}$ be a rearrangement\hyp{}invariant function norm, $m\in\N$, $m < n$. Assume that \eqref{thm:optimal_target_general_theorem_condition} and one of the conditions
\begin{itemize}
\item $m = 2$;
\item $m \neq 2$ and the operator $f\mapsto f^{**}$ is bounded on $X'(0, \infty)$
\end{itemize}
are satisfied. Set $k = \lceil m/2 - 1 \rceil$. Let $Y_{m, X}(\Hn)$ be the rearrangement\hyp{}invariant function space whose associate function norm is defined as
\begin{equation}\label{thm:optimal_target_general_theorem_simplified_M_bounded_Xasoc_asoc_norm}
\|g\|_{Y_{m, X}'(0, \infty)} = \begin{cases}
	\|(R_1 \circ H_2^k)g^{**}\|_{X'(0, \infty)} \quad &\text{if $m$ is odd},\\
	\|H_2^{k + 1} g^{**}\|_{X'(0, \infty)} \quad &\text{if $m$ is even},
\end{cases}
\end{equation}
for $g\in\Mpl(0, \infty)$. The rearrangement\hyp{}invariant function space $Y_{m, X}(\Hn)$ is the optimal target space for $X(\Hn)$ in \eqref{intro:Sob_ineq}.

Furthermore, if \eqref{thm:optimal_target_general_theorem_condition} is not satisfied, there is no rearrangement\hyp{}invariant function space $Y(\Hn)$ for which \eqref{intro:Sob_ineq} with $X(\Hn)$ is valid.
\end{theorem}

\begin{remark}\label{rem:boundedness_P_on_Xasoc_implies_necessary_cond}\ 
\begin{enumerate}[label=(\roman*), ref=(\roman*)]
\item\label{rem:boundedness_P_on_Xasoc_implies_necessary_cond_explanation_item} When $m \neq 2$, the assumption that \eqref{thm:optimal_target_general_theorem_condition} is valid in \cref{thm:optimal_target_general_theorem_simplified_M_bounded_Xasoc} is actually vacuously satisfied. The reason is that the boundedness of the operator $f\mapsto f^{**}$ on $X'(0, \infty)$ implies the validity of \eqref{thm:optimal_target_general_theorem_condition}. Indeed, we have
\begin{align*}
\|P^{k + 1}\chi_{(0, 1)}\|_{X'(0, \infty)} \leq \|P\|^{k+1}_{X'(0, \infty)} \|\chi_{(0, 1)}\|_{X'(0, \infty)} < \infty \\
\intertext{and}
P^{k + 1}\chi_{(0, 1)}(t) \approx \chi_{(0, 1)}(t) + \frac{(1 + \log t)^k}{t}\chi_{[1, \infty)}(t) \quad \text{for every $t\in(0, \infty)$}.
\end{align*}
The assumption is included in the statement, nevertheless, because its validity is necessary (and sufficient) for the existence of a rearrangement\hyp{}invariant target function space in \eqref{intro:Sob_ineq}.
\item If $X(\Hn)$ is a Lebesgue space $L^p(\Hn)$ (or, more generally, a Lorentz space $L^{p, q}(\Hn)$, or, even more generally, a Lorentz--Zygmund space $L^{p, q; \A}(\Hn)$), then the operator $f\mapsto f^{**}$ is bounded on $X'(0, \infty)$ if and only if $p\in[1, \infty)$.
\item When $m\geq4$ is even, the rearrangement\hyp{}invariant function norm defined by \eqref{thm:optimal_target_general_theorem_simplified_M_bounded_Xasoc_asoc_norm} is equivalent to a rearrangement\hyp{}invariant function norm defined as
\begin{equation*}
\Mpl(0, \infty) \ni g\mapsto \Big\| \int_t^\infty g^{**}(s) \phi_{m}(s) K_k(t,s) \d{s} \Big\|_{X'(0, \infty)}.
\end{equation*}
When $m\geq3$ is odd, the rearrangement\hyp{}invariant function norm defined by \eqref{thm:optimal_target_general_theorem_simplified_M_bounded_Xasoc_asoc_norm} is equivalent to a rearrangement\hyp{}invariant function norm defined as
\begin{align*}
\Mpl(0, \infty) \ni g\mapsto &\Big\| \phi_1(t) \int_0^t g^{**}(s) s \phi_{m-1}(s) \d{s} \Big\|_{X'(0, \infty)} \\
&\,+ \Big\| \phi_1(t) t \int_t^\infty g^{**}(s) \phi_{m-1}(s) K_{k-1}(t,s) \d{s} \Big\|_{X'(0, \infty)}.
\end{align*}
See \cref{rem:operator_induced_norms_closed_forms}\ref{rem:operator_induced_norms_closed_forms:item} for more detail.
\end{enumerate}
\end{remark}

The following theorem can be used, loosely speaking, when $X(\Hn)$ is not ``too close to $L^1(\Hn)$ nor to $L^\infty(\Hn)$''.
\begin{theorem}\label{thm:optimal_target_general_theorem_simplified_M_bounded_on_both_X_and_Xasoc}
Let $\|\cdot\|_{X(0, \infty)}$ be a rearrangement\hyp{}invariant function norm, $m\in\N$, $3\leq m < n$. Assume that the operator $f\mapsto f^{**}$ is bounded on both $X(0, \infty)$ and $X'(0, \infty)$. Let $Y_{m, X}(\Hn)$ be the rearrangement\hyp{}invariant function space whose associate function norm is defined by
\begin{equation*}
\|g\|_{Y_{m, X}'(0, \infty)} = \|R_m g^*\|_{X'(0, \infty)},\ g\in\Mpl(0, \infty).
\end{equation*}
The rearrangement\hyp{}invariant function space $Y_{m, X}(\Hn)$ is the optimal target space for $X(\Hn)$ in \eqref{intro:Sob_ineq}.
\end{theorem}

\begin{remark}\label{rem:optimal_target_general_theorem_simplified_M_bounded_on_both_X_and_Xasoc_m_one_two_already_known}\ 
\begin{enumerate}[label=(\roman*), ref=(\roman*)]
\item\label{rem:optimal_target_general_theorem_simplified_M_bounded_on_both_X_and_Xasoc_m_one_two_already_known_m_1_2_item} Owing to \cref{thm:optimal_target_general_theorem_simplified_M_bounded_X}, the conclusion of \cref{thm:optimal_target_general_theorem_simplified_M_bounded_on_both_X_and_Xasoc} is valid for $m\in\{1, 2\}$, too, under less restrictive assumptions.
\item It is known (e.g., \cite{BS, C:97, EMMP:20}) that the boundedness of the operator $f\mapsto f^{**}$ on both $X(0, \infty)$ and $X'(0, \infty)$ is equivalent to the boundedness of the Riesz transforms on $X(\rn)$.
\end{enumerate}
\end{remark}

The following theorem is divided into two parts. The first one characterizes when the optimal target space in \eqref{intro:Sob_ineq} can be described in terms of the optimal target space in a suitable Sobolev type inequality in $\rn$ (see \cref{rem:optimal_target_eucl_and_supercritical}\ref{rem:optimal_target_eucl_and_supercritical_Rn_comparison_item}). The second one characterizes when the optimal target space in \eqref{intro:Sob_ineq} is contained in $L^\infty(\Hn)$.
\begin{theorem}\label{thm:optimal_target_eucl_and_supercritical}
Let $\|\cdot\|_{X(0, \infty)}$ be a rearrangement\hyp{}invariant function norm, $m\in\N$, $m < n$. Assume that the operator $f\mapsto f^{**}$ is bounded on $X'(0, \infty)$. If $m\geq2$, assume that it is also bounded on $X(0, \infty)$.
\begin{enumerate}[label=(\roman*), ref=(\roman*)]
\item\label{thm:optimal_target_eucl_and_supercritical_subcritical_item} If
\begin{equation}\label{thm:optimal_target_eucl_and_supercritical_p_(1_n/m)_condition}
t^{-1 + \frac{m}{n}}\chi_{(1, \infty)}(t) \in X'(0, \infty),
\end{equation}
then the rearrangement\hyp{}invariant function space
\begin{equation*}
Y_{m, X}(\Hn) = Z_{m, X}(\Hn) \cap X(\Hn),
\end{equation*}
where $Z_{m,X}(\Hn)$ is the rearrangement\hyp{}invariant function space whose associate function norm is defined as
\begin{equation}\label{thm:optimal_target_eucl_and_supercritical_p_(1_n/m)_Z_asoc}
\|g\|_{Z_{m, X}'(0, \infty)}  = \|t^{\frac{m}{n}}g^{**}(t)\|_{X'(0, \infty)},\ g\in\Mpl(0, \infty),
\end{equation}
is the optimal target space for $X(\Hn)$ in \eqref{intro:Sob_ineq}. 

On the other hand, if \eqref{thm:optimal_target_general_theorem_condition} is satisfied but \eqref{thm:optimal_target_eucl_and_supercritical_p_(1_n/m)_condition} is not, there is an optimal target space for $X(\Hn)$ in \eqref{intro:Sob_ineq}, but it is not equivalent to $Z_{m, X}(\Hn) \cap X(\Hn)$.
\item\label{thm:optimal_target_eucl_and_supercritical_supercritical_item}  If
\begin{equation}\label{thm:optimal_target_eucl_and_supercritical_p_(n/m_infty)_condition}
t^{-1 + \frac{m}{n}}\chi_{(0, 1)}(t) +  \frac{(1 + \log t)^{\lceil m/2 - 1\rceil}}{t}\chi_{(1, \infty)}(t) \in X'(0, \infty),
\end{equation}
then the rearrangement\hyp{}invariant function space
\begin{equation*}
Y_{m, X}(\Hn) = L^\infty(\Hn) \cap X(\Hn)
\end{equation*}
is the optimal target space for $X(\Hn)$ in \eqref{intro:Sob_ineq}.

On the other hand, if \eqref{thm:optimal_target_general_theorem_condition} is satisfied but \eqref{thm:optimal_target_eucl_and_supercritical_p_(n/m_infty)_condition} is not, there is an (optimal) target space for $X(\Hn)$ in \eqref{intro:Sob_ineq}, but \eqref{intro:Sob_ineq} is not valid with any rearrangement\hyp{}invariant function space $Y(\Hn)$ such that $Y(\Hn) \subseteq L^\infty(\Hn)$.
\end{enumerate}
\end{theorem}

\begin{remark}\label{rem:optimal_target_eucl_and_supercritical}\ 
\begin{enumerate}[label=(\roman*), ref=(\roman*)]
\item Let $X(\Hn) = L^{p, q}(\Hn)$ be a Lorentz space. The condition \eqref{thm:optimal_target_eucl_and_supercritical_p_(1_n/m)_condition} is satisfied if and only if either $p = n/m$ and $q = 1$ or $p\in[1, n/m)$. The condition \eqref{thm:optimal_target_eucl_and_supercritical_p_(n/m_infty)_condition} is satisfied if and only if either $p = n/m$ and $q = 1$ or $p\in(n/m, \infty)$.
\item\label{rem:optimal_target_eucl_and_supercritical_Rn_comparison_item} Let $Z_{m, X}(\rn)$ denote the rearrangement\hyp{}invariant function space in the Euclidean space $\rn$ whose associate function norm is defined by \eqref{thm:optimal_target_eucl_and_supercritical_p_(1_n/m)_Z_asoc}. If \eqref{thm:optimal_target_eucl_and_supercritical_p_(1_n/m)_condition} is satisfied, then $Z_{m, X}(\rn)$ is the optimal target space for $X(\rn)$ in the Sobolev inequality
\begin{equation*}
\|u\|_{Z(\rn)} \leq C \|D^m u\|_{X(\rn)},
\end{equation*}
where $D^m$ is the vector of all $m$th order partial derivatives of $u$, and where $u$ is a $m$ times weakly differentiable function in $\rn$ whose weak derivatives up to order $m - 1$ vanish at infinity in a suitable sense (see \cite{M:21}). In contrast with the hyperbolic space $\Hn$, it is not possible to intersect $Z_{m, X}(\rn)$ with $X(\rn)$. Furthermore, if, in addition, the maximal operator is bounded on $X(\rn)$, then $Z_{m, X}(\rn)$ is also the optimal target space for $X(\rn)$ and the $m$th order Riesz potential in $\rn$ (see \cite{EMMP:20}).
\item In general, neither \eqref{thm:optimal_target_eucl_and_supercritical_p_(1_n/m)_condition} nor \eqref{thm:optimal_target_eucl_and_supercritical_p_(n/m_infty)_condition} implies that the operator $f\mapsto f^{**}$ is bounded on $X(0, \infty)$ or $X'(0, \infty)$. For example, \eqref{thm:optimal_target_eucl_and_supercritical_p_(1_n/m)_condition} is satisfied for $X(0, \infty) = L^1(0, \infty)$, but the operator $f\mapsto f^{**}$ is not bounded on $L^1(0, \infty)$. It is also satisfied for $X(0, \infty) = L^1(0, \infty) \cap L^\infty(0, \infty)$, but the operator is not bounded on $X'(0, \infty)$. Similar counterexamples show that \eqref{thm:optimal_target_eucl_and_supercritical_p_(n/m_infty)_condition} does not imply the boundedness  of the operator on $X(0, \infty)$ or $X'(0, \infty)$.
\end{enumerate}
\end{remark}

We conclude this section with the most general theorem describing the optimal target space in \eqref{intro:Sob_ineq}. It completely characterizes optimal target spaces in the Sobolev inequality. For example, when $m\geq3$ and $X(\Hn) = L^1(\Hn)$, we need to use it because none of the previous theorems can be used.
\begin{theorem}\label{thm:optimal_target_general_theorem}
Let $\|\cdot\|_{X(0, \infty)}$ be a rearrangement\hyp{}invariant function norm, $m\in\N$, $m < n$. Assume that \eqref{thm:optimal_target_general_theorem_condition} is satisfied. Set $k = \lceil m/2 - 1 \rceil$.
Let $Y_{m, X}(\Hn)$ be the rearrangement\hyp{}invariant function space whose associate function norm is defined as
\begin{equation}\label{thm:optimal_target_general_theorem_asoc_norm}
\|g\|_{Y_{m, X}'(0, \infty)} = \begin{cases}
	\|(R_1\circ (H_2\circ P)^k)g^*\|_{X'(0, \infty)} \quad &\text{if $m$ is odd}, \\
	\|(H_2\circ P)^{k+1}g^*\|_{X'(0, \infty)} \quad &\text{if $m$ is even},
\end{cases}
\end{equation}
for $g\in\Mpl(0, \infty)$. The rearrangement\hyp{}invariant function space $Y_{m, X}(\Hn)$ is the optimal target space for $X(\Hn)$ in \eqref{intro:Sob_ineq}.

Furthermore, if \eqref{thm:optimal_target_general_theorem_condition} is not satisfied, there is no rearrangement\hyp{}invariant function space $Y(\Hn)$ for which \eqref{intro:Sob_ineq} with $X(\Hn)$ is valid.
\end{theorem}

\begin{remark}
When $m\geq4$ is even, the rearrangement\hyp{}invariant function norm defined by \eqref{thm:optimal_target_general_theorem_asoc_norm} is equivalent to a rearrangement\hyp{}invariant function norm defined as
\begin{align*}
\Mpl(0, \infty) \ni g\mapsto &\Big\| \phi_m(t) \int_0^t g^*(s) K_k(s,t) \d{s} \Big\|_{X'(0, \infty)} \\
&\,+ \Big\| \int_t^\infty g^{**}(s) \phi_m(s) K_k(t,s) \d{s} \Big\|_{X'(0, \infty)}.
\end{align*}
When $m\geq3$ is odd, the rearrangement\hyp{}invariant function norm defined by \eqref{thm:optimal_target_general_theorem_asoc_norm} is equivalent to a rearrangement\hyp{}invariant function norm defined as
\begin{align*}
\Mpl(0, \infty) \ni g\mapsto &\Big\| \phi_1(t) \int_0^t g^{**}(s) s \phi_{m-1}(s) \d{s} \Big\|_{X'(0, \infty)}\\
&\,+ \Big\| \phi_1(t) t \int_t^\infty g^{**}(s) \phi_{m-1}(s) K_{k-1}(t,s) \d{s} \Big\|_{X'(0, \infty)}\\
&\,+ \Big\| \phi_m(t) \int_0^t g^*(s) K_k(s,t) \d{s} \Big\|_{X'(0, \infty)}.
\end{align*}
See \cref{rem:operator_induced_norms_closed_forms}\ref{rem:operator_induced_norms_closed_forms:item} for more detail.
\end{remark}

\section{Auxiliary Results}\label{sec:aux_results}
The following P\'olya--Szeg\"o inequality for the hyperbolic gradient, which ensures that a rearrangement\hyp{}invariant function norm of the hyperbolic gradient is not increased by symmetrization, was essentially proved in \citep[Theorem~2.2]{N:20}. It was proved there in the setting of Lorentz spaces, but the proof relies only on rearrangement techniques, the coarea formula and the isoperimetric inequality in the hyperbolic space $\Hn$ (\citep{BDS:15}). Therefore, it can easily be extended to the setting of general rearrangement\hyp{}invariant function norms\textemdash to that end, see \cite[Lemma~4.1]{CP:98} or \cite[Lemma~3.3]{CP:09}\textemdash and we omit its proof.
\begin{proposition}\label{prop:PS_inequality}
Let $\|\cdot\|_{X(0, \infty)}$ be a rearrangement\hyp{}invariant function norm. If $u\in V_0^1X(\Hn)$, then the function $u^*$ is locally absolutely continuous on $(0, \infty)$, and we have
\begin{equation*}
n\omega_n\Big\|-\frac{du^*}{dt}(t)\sinh(\varrho(t))^{n-1}\Big\|_{X(0, \infty)} \leq \| \gradH u\|_{X(\Hn)},
\end{equation*}
where $\varrho$ is the inverse function to the function $V$ defined by \eqref{prel:volume_of_hyperbolic_ball}.
\end{proposition}

The following proposition is a step toward getting the behavior of Hardy type operators that we will need to handle under our control.
\begin{proposition}\label{prop:properties_of_iterated_operators}
Let $f,g,\in\Mpl(0, \infty)$, $\alpha, \beta\in[0, n)$, and $j \in \N_0$.
\begin{enumerate}[label=(\roman*), ref=(\roman*)]
	\item\label{prop:properties_of_iterated_operators_H2P_self_adjoint_item} We have
	\begin{align}\label{prop:properties_of_iterated_operators_H2P_self_adjoint}
	\Big( \frac{n-\alpha}{n} \Big)^j \int_0^\infty (H_\alpha\circ P)^j f(t) g(t) \d{t}	&\leq \int_0^\infty f(t) (H_\alpha\circ P)^j g(t) \d{t} \notag\\
	&\leq \Big( \frac{n}{n-\alpha} \Big)^j \int_0^\infty (H_\alpha\circ P)^j f(t) g(t) \d{t}.
	\end{align}
	\item\label{prop:properties_of_iterated_operators_composition_pointwise_item} For every $t\in(0, \infty)$, we have
	\begin{align}
		(R_\beta \circ R_\alpha^j)f(t) &= \frac{\phi_\beta(t)}{j!}\int_0^t f(s) \Big( \int_{s}^t \phi_\alpha \Big)^j \d{s} \label{prop:properties_of_iterated_operators_Ri_R_jk}\\
		\intertext{and}
		(H_\alpha^j \circ H_\beta)f(t) &= \frac1{j!}\int_t^\infty f(s) \phi_\beta(s) \Big( \int_t^s \phi_\alpha \Big)^j \d{s}. \label{prop:properties_of_iterated_operators_Hjk_Hi}
	\end{align}
	\item\label{prop:properties_of_iterated_operators_HaPk_item} If $\alpha\in(0, n)$, then we have
	\begin{align}
		\|(H_\alpha\circ P)^{j + 1} f\|_{X(0, \infty)} &\approx \|(H_\alpha^{j+1} \circ P)f\|_{X(0, \infty)} + \|R_\alpha^{j+1} f\|_{X(0, \infty)} \label{prop:properties_of_iterated_operators_(Hj_P)k}\\
		\intertext{and}
		\|(R_\beta\circ(H_\alpha\circ P)^{j+1}) f\|_{X(0, \infty)} &\approx \|(R_\beta \circ H_\alpha^{j+1} \circ P)f\|_{X(0, \infty)} \notag\\
		&\quad + \|(R_\beta \circ R_\alpha^{j+1}) f\|_{X(0, \infty)} \label{prop:properties_of_iterated_operators_Ri_(Hj_P)k}
	\end{align}
	for every rearrangement\hyp{}invariant function norm $\|\cdot\|_{X(0, \infty)}$. The multiplicative constants depend only on $\alpha, \beta$, $n$ and $j$.
\end{enumerate}
\end{proposition}
\begin{proof}
We start with \ref{prop:properties_of_iterated_operators_H2P_self_adjoint_item}. By Fubini theorem, we have
\begin{equation*}
(H_\alpha \circ P)f(t) = \Big( \int_t^\infty \frac{\phi_\alpha(s)}{s} \d{s}\Big) \Big(\int_0^t f(s) \d{s}\Big) + \int_t^\infty f(s) \Big(\int_s^\infty \frac{\phi_\alpha(\tau)}{\tau} \d{\tau}\Big) \d{s}.
\end{equation*}
Moreover, we have
\begin{equation*}
\int_t^\infty \frac{\phi_\alpha(s)}{s} \d{s} = \begin{cases}
\frac{n}{n - \alpha}t^{-1 + \frac{\alpha}{n}} - \frac{\alpha}{n - \alpha}\quad &\text{if $t\in(0, 1)$}, \\
t^{-1} \quad &\text{if $t\in[1, \infty)$},
\end{cases}
\end{equation*}
for every $t\in(0, \infty)$, and so
\begin{equation*}
\phi_\alpha(t) \leq \int_t^\infty \frac{\phi_\alpha(s)}{s} \d{s} \leq \frac{n}{n-\alpha} \phi_\alpha(t) \quad \text{for every $t\in(0, \infty)$}.
\end{equation*}
Hence we have
\begin{equation}\label{prop:properties_of_iterated_operators_Ha_P_approx_Ra_+_Ha}
(R_\alpha + H_\alpha) f(t) \leq (H_\alpha\circ P) f(t) \leq \frac{n}{n-\alpha} (R_\alpha + H_\alpha) f(t),
\end{equation}
which combined with \eqref{prel:ri:Rk_Hk_self_adjoint} implies \eqref{prop:properties_of_iterated_operators_H2P_self_adjoint}.

As for \ref{prop:properties_of_iterated_operators_composition_pointwise_item}, note that both identities are trivially true when $j = 0$. Let $j\in\N$. Thanks to straightforward computations involving the Fubini theorem and a change of variables, we have
\begin{align}
R_\alpha^jf(t) &= \frac{\phi_\alpha(t)}{(j - 1)!}\int_0^t f(s) \Big( \int_s^t \phi_\alpha \Big)^{j-1} \d{s} \label{prop:properties_of_iterated_operators_Ra_k}\\ 
\intertext{and}
H_\alpha^kf(t) &= \frac1{(j - 1)!}\int_t^\infty f(s) \phi_\alpha(s) \Big( \int_t^s \phi_\alpha \Big)^{j-1} \d{s} \label{prop:properties_of_iterated_operators_Ha_k}
\end{align}
for every $t\in(0, \infty)$. Using that, the Fubini theorem and the same change of variables again, we obtain
\begin{align*}
(R_\beta\circ R_\alpha^j) f(t) &= \frac{\phi_\beta(t)}{(j-1)!}\int_0^t \phi_\alpha(s) \Big( \int_0^s f(\tau) \Big( \int_\tau^s \phi_\alpha \Big)^{j-1} \d{\tau} \Big) \d{s} \\
&= \frac{\phi_\beta(t)}{(j-1)!}\int_0^t f(\tau) \Big( \int_\tau^t \phi_\alpha(s) \Big( \int_\tau^s \phi_\alpha \Big)^{j-1} \d{s} \Big) \d{\tau} \\
&= \frac{\phi_\beta(t)}{j (j-1)!}\int_0^t f(\tau) \Big( \int_\tau^t \phi_\alpha \Big)^{j} \d{\tau}.
\end{align*}
The identity \eqref{prop:properties_of_iterated_operators_Hjk_Hi} can be obtained in a similar way.

It remains to prove \ref{prop:properties_of_iterated_operators_HaPk_item}. We start with some preliminary computations. First, we clearly have
\begin{equation}\label{prop:properties_of_iterated_operators_kernel_phi_alpha_explicit}
\int_a^b \phi_\alpha = \begin{cases}
	\frac{n}{\alpha}(b^\frac{\alpha}{n} - a^\frac{\alpha}{n})  \quad &\text{if $0\leq a \leq b \leq 1$}, \\
	\frac{n}{\alpha}(1 - a^\frac{\alpha}{n}) + \log b \quad &\text{if $0\leq a \leq 1 \leq b < \infty$}, \\
	\log \frac{b}{a} \quad &\text{if $1 \leq a \leq b < \infty$}.
\end{cases}
\end{equation}
Using that, it is not hard to see that
\begin{equation}\label{prop:properties_of_iterated_operators_aux_upper}
\int_0^a \Big( \int_s^b \phi_\alpha \Big)^i \d{s} \lesssim a \Big( \int_a^{2b} \phi_\alpha \Big)^i \approx a \Big( \int_{\frac{a}{2}}^{b} \phi_\alpha \Big)^i \quad \text{for every $0 < a \leq b < \infty$}
\end{equation}
and every $i\in\N_0$, in which the multiplicative constants depend only on $\alpha$,  $n$ and $i$. The equivalence is an immediate consequence of a change of variables, and so only the inequality needs to be verified. That is straightforward when $a\in(0, 1]$, in which case we simply estimate $\int_s^b \phi_\alpha$ by $\int_0^b \phi_\alpha$. When $a\in[1,b]$, we have
\begin{equation*}
\int_0^a \Bigg( \int_s^b \phi_\alpha \Bigg)^i \d{s} \leq \int_0^1 \Bigg( \int_0^b \phi_\alpha \Bigg)^i \d{s} + \int_1^a \log^i\Big( \frac{b}{s} \Big) \d{s}.
\end{equation*}
As for the first term, using the fact that
\begin{equation}\label{prop:properties_of_iterated_operators_monotonicity}
\text{the function $(0, b] \ni t \mapsto t \log^i(e^i b/t)$ is nondecreasing}
\end{equation}
(increasing, actually), we obtain
\begin{equation*}
\int_0^1 \Bigg( \int_0^b \phi_\alpha \Bigg)^i \d{s} \approx (1 + \log b)^i \approx \log^i(e^i b) \leq a\log^i\Big( \frac{e^i b}{a} \Big).
\end{equation*}
As for the second term, we have
\begin{align*}
\int_1^a \log^i\Big( \frac{b}{s} \Big) \d{s} &= b\int_\frac{b}{a}^b \frac{\log^i \tau}{\tau^2} \d{\tau} \leq b\int_\frac{b}{a}^\infty \frac{\log^i \tau}{\tau^2} \d{\tau} \\
&\approx a\Big( 1 + \log^i\Big( \frac{b}{a} \Big) \Big) \approx a\log^i\Big( \frac{e^i b}{a} \Big).
\end{align*}
Hence
\begin{equation*}
\int_0^a \Bigg( \int_s^b \phi_\alpha \Bigg)^i \d{s} \lesssim a\log^i\Big( \frac{e^i b}{a} \Big) \approx a\log^i\Big( \frac{2 b}{a} \Big) \approx a \Big( \int_a^{2b} \phi_\alpha \Big)^i.
\end{equation*}

Next, we claim that
\begin{equation}\label{prop:properties_of_iterated_operators_Ri_H_jk_P_pointwise_upper_estimate}
(R_\beta \circ H_\alpha^{j + 1} \circ P)f(t) \lesssim (R_\beta\circ R_\alpha^{j + 1})f(2t) + (H_\alpha^{j + 1} \circ H_\beta\circ P)f(t/2)
\end{equation}
for every $t\in(0, \infty)$, in which the multiplicative constant depends only on $\alpha$, $\beta$, $n$ and $j$. To that end, using \eqref{prop:properties_of_iterated_operators_Ha_k} and the Fubini theorem, we have
\begin{align}
(R_\beta \circ H_\alpha^{j + 1} \circ P)f(t) &\approx \phi_\beta(t) \int_0^t \Bigg( \int_s^\infty Pf(\tau) \phi_\alpha(\tau) \Big( \int_s^\tau \phi_\alpha \Big)^j \d{\tau} \Bigg) \d{s} \notag\\
&= \phi_\beta(t) \int_0^t Pf(\tau) \phi_\alpha(\tau) \Bigg( \int_0^\tau \Big( \int_s^\tau \phi_\alpha \Big)^j \d{s} \Bigg) \d{\tau} \notag\\
&\quad + \phi_\beta(t) \int_t^\infty Pf(\tau) \phi_\alpha(\tau) \Bigg( \int_0^t \Big( \int_s^\tau \phi_\alpha \Big)^j \d{s} \Bigg) \d{\tau} \label{prop:properties_of_iterated_operators_Ri_H_jk_P_Fubini}
\end{align}
for every $t\in(0, \infty)$. Now, using the Fubini theorem again together with \eqref{prop:properties_of_iterated_operators_aux_upper}, we have
\begingroup
\allowdisplaybreaks
\begin{align*}
&\phi_\beta(t) \int_0^t Pf(\tau) \phi_\alpha(\tau) \Bigg( \int_0^\tau \Big( \int_s^\tau \phi_\alpha \Big)^j \d{s} \Bigg) \d{\tau} \\
&= \phi_\beta(t)\int_0^t f(x) \int_x^t \frac{\phi_\alpha(\tau)}{\tau}\Big( \int_0^\tau \Big( \int_s^\tau \phi_\alpha \Big)^j \d{s} \Big) \d{\tau} \d{x} \\
&\lesssim \phi_\beta(t)\int_0^t f(x) \int_x^t \phi_\alpha(\tau) \Big( \int_\tau^{2\tau} \phi_\alpha \Big)^j \d{\tau} \d{x} \\
&\leq \phi_\beta(t)\int_0^t f(x) \int_x^{2t} \phi_\alpha(\tau) \Big( \int_\tau^{2t} \phi_\alpha \Big)^j \d{\tau} \d{x} \\
&\approx \phi_\beta(t)\int_0^t f(x) \Big( \int_x^{2t} \phi_\alpha \Big)^{j + 1} \d{x} \\
&\lesssim \phi_\beta(2t)\int_0^{2t} f(x) \Big( \int_x^{2t} \phi_\alpha \Big)^{j + 1} \d{x} \\
&\approx  (R_\beta \circ R_\alpha^{j + 1})f(2t),
\end{align*}
where we used \eqref{prop:properties_of_iterated_operators_Ri_R_jk} in the last step. Furthermore, using the fact that
\begin{equation*}
\int_{\frac{a}{2}}^a \phi_\alpha \approx a\phi_\alpha(a) \quad \text{for every $a\in(0, \infty)$},
\end{equation*}
in which the multiplicative constants depend only on $\alpha$ and $n$, the fact that the function $(0, \infty)\ni t \mapsto t\phi_\beta(t)$ is nondecreasing, and \eqref{prop:properties_of_iterated_operators_aux_upper} again, we obtain
\begin{align*}
&\phi_\beta(t) \int_t^\infty Pf(\tau) \phi_\alpha(\tau) \Bigg( \int_0^t \Big( \int_s^\tau \phi_\alpha \Big)^j \d{s} \Bigg) \d{\tau} \\
&\lesssim t\phi_\beta(t) \int_t^\infty Pf(\tau) \phi_\alpha(\tau) \Big( \int_{\frac{t}{2}}^\tau \phi_\alpha \Big)^j \d{\tau} \\
&\approx  t\phi_\beta(t) \int_t^\infty Pf(\tau) \frac1{\tau}\Big( \int_{\frac{\tau}{2}}^\tau \phi_\alpha \Big) \Big( \int_{\frac{t}{2}}^\tau \phi_\alpha \Big)^j \d{\tau} \\
&\leq \int_t^\infty Pf(\tau) \phi_\beta(\tau) \Big( \int_{\frac{t}{2}}^\tau \phi_\alpha \Big)^{j + 1} \d{\tau} \\
&\leq \int_{\frac{t}{2}}^\infty Pf(\tau) \phi_\beta(\tau) \Big( \int_{\frac{t}{2}}^\tau \phi_\alpha \Big)^{j + 1} \d{\tau} \\
&\approx (H_\alpha^{j + 1}\circ H_\beta \circ P)f\Big( \frac{t}{2} \Big),
\end{align*}
\endgroup
where we used \eqref{prop:properties_of_iterated_operators_Hjk_Hi} in the last step. Combining the last two chains of inequalities with \eqref{prop:properties_of_iterated_operators_Ri_H_jk_P_Fubini}, we obtain \eqref{prop:properties_of_iterated_operators_Ri_H_jk_P_pointwise_upper_estimate}. Finally, we claim that
\begin{align}
(H_\alpha^{j + 1} \circ P)f(t) + R_\alpha^{j + 1} f(t) &\lesssim (H_\alpha \circ P)^{j + 1} f(t) \notag\\
&\lesssim (H_\alpha^{j + 1} \circ P)f(t/2^j) + R_\alpha^{j + 1} f(2^jt), \label{prop:properties_of_iterated_operators_(Hj_P)k_pointwise}
\end{align}
for every $t\in(0, \infty)$, in which the multiplicative constants depend only on $\alpha$, $n$ and $j$. We shall prove that by induction on $j$. If $j = 0$, \eqref{prop:properties_of_iterated_operators_(Hj_P)k_pointwise} reads as
\begin{equation*}
(H_\alpha \circ P)f(t) + R_\alpha f(t) \lesssim (H_\alpha \circ P) f(t) \lesssim (H_\alpha \circ P)f(t) + R_\alpha f(t),
\end{equation*}
which is clearly true thanks to \eqref{prop:properties_of_iterated_operators_Ha_P_approx_Ra_+_Ha}. Now, let us assume that \eqref{prop:properties_of_iterated_operators_(Hj_P)k_pointwise} is valid for some $j\in\N_0$. Using \eqref{prop:properties_of_iterated_operators_Ha_P_approx_Ra_+_Ha}, we have
\begin{equation*}
(H_\alpha \circ P)^{j+2} f(t) \approx ((R_\alpha + H_\alpha)\circ(H_\alpha \circ P)^{j+1}) f(t) \quad \text{for every $t\in(0, \infty)$}.
\end{equation*}
On the one hand, using the inductive hypothesis, we have
\begin{align*}
((R_\alpha + H_\alpha)\circ(H_\alpha \circ P)^{j + 1}) f(t) &\gtrsim ((R_\alpha + H_\alpha)\circ(H_\alpha^{j + 1}\circ P + R_\alpha^{j + 1}))f(t) \\
&\geq R_\alpha^{j + 2}f(t) + (H_\alpha^{j + 2}\circ P)f(t)
\end{align*}
for every $t\in(0, \infty)$. Hence, we have established the desired lower bound. On the other hand, using the inductive hypothesis again together with changes of variables, we obtain
\begin{align*}
&((R_\alpha + H_\alpha)\circ(H_\alpha \circ P)^{j+1}) f(t) \\
&\lesssim (R_\alpha + H_\alpha)((H_\alpha^{j+1} \circ P)f(\cdot/2^j))(t) + (R_\alpha + H_\alpha)(R_\alpha^{j+1} f(2^j\,\cdot))(t) \\
&\approx ((R_\alpha + H_\alpha)\circ(H_\alpha^{j+1} \circ P))f(t/2^j) + ((R_\alpha + H_\alpha)\circ R_\alpha^{j+1})f(2^jt) \\
&= R_\alpha^{j + 2}f(2^{j}t) + (H_\alpha^{j+2} \circ P)f(t/2^j) \\
&\quad + (R_\alpha \circ H_\alpha^{j+1} \circ P)f(t/2^j) + (H_\alpha \circ R_\alpha^{j+1})f(2^jt)
\end{align*}
for every $t\in(0, \infty)$. Combining that with \eqref{prop:properties_of_iterated_operators_Ri_H_jk_P_pointwise_upper_estimate}, we arrive at
\begin{align*}
((R_\alpha + H_\alpha)\circ(H_\alpha \circ P)^{j+1}) f(t) &\lesssim R_\alpha^{j + 2}f(2^jt) + (H_\alpha^{j+2} \circ P)f(t/2^j)\\
&\quad + R_\alpha^{j + 2}f(t/2^{j-1}) + (H_\alpha^{j+2} \circ P)f(t/2^{j+1})\\
&\quad + (H_\alpha \circ R_\alpha^{j+1})f(2^jt) \\
&\lesssim R_\alpha^{j + 2}f(2^{j+1} t) + (H_\alpha^{j+2} \circ P)f(t/2^{j+1}) \\
&\quad + (H_\alpha \circ R_\alpha^{j+1})f(2^{j}t)
\end{align*}
for every $t\in(0, \infty)$, in which the multiplicative constants depend only on $\alpha$, $n$ and $j$. In view of that, in order to establish the desired upper bound in \eqref{prop:properties_of_iterated_operators_(Hj_P)k_pointwise}, it is sufficient to prove that
\begin{equation*}
(H_\alpha \circ R_\alpha^{j+1})f(t) \lesssim R_\alpha^{j + 2}f(2t) + (H_\alpha^{j+2} \circ P)f(t/2) \quad \text{for every $t\in(0 ,\infty)$},
\end{equation*}
in which the multiplicative constants depend only on $\alpha$, $n$ and $j$. To that end, using the Fubini theorem and \eqref{prop:properties_of_iterated_operators_Ra_k}, we have
\begin{align*}
(H_\alpha \circ R_\alpha^{j+1})f(t) &\approx \int_0^t f(\tau) \int_t^\infty \phi_\alpha(s)^2 \Big( \int_\tau^s \phi_\alpha \Big)^j \d{s} \d{\tau} \\
&\quad + \int_t^\infty f(\tau) \int_\tau^\infty \phi_\alpha(s)^2 \Big( \int_\tau^s \phi_\alpha \Big)^j \d{s} \d{\tau}.
\end{align*}
We claim that
\begin{equation}\label{prop:properties_of_iterated_operators_aux_upper2}
\int_a^\infty \phi_\alpha^2(s) \Big( \int_b^s \phi_\alpha \Big)^i \d{s} \lesssim \phi_\alpha(a) \Big( \int_b^{2a} \phi_\alpha \Big)^{i+1} \approx \phi_\alpha(a) \Big( \int_{\frac{b}{2}}^a \phi_\alpha \Big)^{i+1}
\end{equation}
for every $0 < b \leq a < \infty$ and every $i\in\N_0$, in which the multiplicative constants depend only on $\alpha$, $n$ and $i$. The equivalence is an immediate consequence of a change of variables. If $0< b \leq a \leq 1$, then
\begin{align*}
\int_a^\infty \phi_\alpha^2(s) \Big( \int_b^s \phi_\alpha \Big)^i \d{s} &\lesssim \phi_\alpha(a)\int_a^1 s^{-1 + \frac{\alpha}{n}} s^\frac{i\alpha}{n} \d{s} \\
&\quad + \int_1^\infty \frac{(1 + \log s)^i}{s^2} \d{s} \\
&\approx \phi_\alpha(a)(a^{\frac{(i+1)\alpha}{n}} - 1) + 1 \\
&\lesssim \phi_\alpha(a)a^{\frac{(i+1)\alpha}{n}} \approx \phi_\alpha(a) \Big( \int_{\frac{a}{2}}^a \phi_\alpha \Big)^{i + 1} \\
&\leq \phi_\alpha(a) \Big( \int_{\frac{b}{2}}^a \phi_\alpha \Big)^{i + 1},
\end{align*}
whereas, if $1\leq b\leq a < \infty$, then
\begin{align*}
\int_a^\infty \phi_\alpha^2(s) \Big( \int_b^s \phi_\alpha \Big)^i \d{s} &= \int_a^\infty \frac{\log^i(\frac{s}{b})}{s^2} \d{s} \\
&= \frac1{b} \int_{\frac{a}{b}}^\infty \frac{\log^i(\tau)}{\tau^2} \d{\tau} \\
&\approx \frac1{a} \Big( 1 + \log^i\Big( \frac{a}{b} \Big) \Big) \\
&\approx \frac1{a}\log^i\Big( \frac{2a}{b} \Big) \lesssim \frac1{a}\log^{i + 1}\Big( \frac{2a}{b} \Big) \\
&= \phi_\alpha(a) \Big( \int_b^{2a} \phi_\alpha \Big)^{i+1}.
\end{align*}
The remain case $0< b \leq 1 \leq a < \infty$ can be proved similarly. Using \eqref{prop:properties_of_iterated_operators_aux_upper2} and \eqref{prop:properties_of_iterated_operators_Ra_k}, we have
\begin{align*}
&\int_0^t f(\tau) \int_t^\infty \phi_\alpha(s)^2 \Big( \int_\tau^s \phi_\alpha \Big)^j \d{s} \d{\tau} \\
&\lesssim \phi_\alpha(t) \int_0^t f(\tau) \Big( \int_\tau^{2t} \phi_\alpha \Big)^{j+1} \d{\tau} \\
&\lesssim R_\alpha^{j+2}  f(2t).
\end{align*}

Thanks to \eqref{prop:properties_of_iterated_operators_aux_upper2}, \eqref{prop:properties_of_iterated_operators_Ha_k} and \eqref{prop:properties_of_iterated_operators_Ha_P_approx_Ra_+_Ha}, we also have
\begin{align*}
&\int_t^\infty f(\tau) \int_\tau^\infty \phi_\alpha(s)^2 \Big( \int_\tau^s \phi_\alpha \Big)^j \d{s} \d{\tau} \\
&\lesssim \int_t^\infty f(\tau) \phi_\alpha(\tau) \Big( \int_{\frac{\tau}{2}}^\tau \phi_\alpha \Big)^{j+1} \d{\tau} \\
&\leq \int_{\frac{t}{2}}^\infty f(\tau) \phi_\alpha(\tau) \Big( \int_{\frac{t}{2}}^\tau \phi_\alpha \Big)^{j+1} \d{\tau} \\
&\approx H_\alpha^{j+2}f\Big( \frac{t}{2} \Big) \leq (H_\alpha^{j+2}\circ P) f\Big( \frac{t}{2} \Big),
\end{align*}
which concludes the inductive step.

At last, having \eqref{prop:properties_of_iterated_operators_(Hj_P)k_pointwise} at our disposal, the rest is easy. The desired equivalences \eqref{prop:properties_of_iterated_operators_(Hj_P)k} and \eqref{prop:properties_of_iterated_operators_Ri_(Hj_P)k} follow from \eqref{prop:properties_of_iterated_operators_(Hj_P)k_pointwise} owing to the boundedness of the dilation operator on rearrangement\hyp{}invariant function spaces (see \eqref{prel:ri:dilation}).
\end{proof}

\begin{remark}\label{rem:neither_Hardy_is_better_in_general}
Let $\alpha\in(0, n)$. When $j\in\N$, $j\geq2$, the operators $H_\alpha^j \circ P$ and $R_\alpha^j$ are in general incomparable in the sense that the boundedness of either from a rearrangement\hyp{}invariant function space $X(0, \infty)$ to a rearrangement\hyp{}invariant function space $Y(0, \infty)$ does not imply that the other is also bounded from $X(0, \infty)$ to $Y(0, \infty)$. That can be illustrated by the following two examples (actually, the first one works also for $j = 1$). It is not hard to see that the operator $R_\alpha^j$ is bounded from the Lorentz--Zygmund space $L^{\infty, \infty; [0, j - 1]}(0, \infty)$ to $L^\infty(0, \infty)$. However, $H_\alpha^j \circ P$ is not bounded from $L^{\infty, \infty; [0, j - 1]}(0, \infty)$ to any rearrangement\hyp{}invariant function space, for $\ell^{[0, 1 - j]} \in L^{\infty, \infty; [0, j - 1]}(0, \infty)$ but $(H_\alpha^j \circ P)(\ell^{[0, 1 - j]})(t) = \infty$ for every $t\in(0, \infty)$. As for the other example, it can be shown that the operator $H_\alpha^j \circ P$ is bounded from $L^1(0, \infty)$ to $L^{\frac{n}{n - 2j}, \infty}(0, \infty) \cap L^{(1, \infty; [0, -1])}(0, \infty)$, whereas $R_\alpha^j$ is bounded from $L^1(0, \infty)$ only to $L^{\frac{n}{n - 2j}, \infty}(0, \infty) \cap L^{(1, \infty; [0, -j])}(0, \infty)$.
\end{remark}

The following proposition tells us that rearrangement\hyp{}invariant function norms of some iterated Hardy operators are equivalent to those of certain non-iterated Hardy-type operators involving the kernel defined by \eqref{E:def_kernel_K}.
\begin{proposition}\label{prop:iterated_operators_closed_form}
Let $\alpha\in(0, n)$, $\beta\in[0, n)$, and $j \in \N_0$. For every rearrangement\hyp{}invariant function norm $\|\cdot\|_{X(0, \infty)}$ and every $f\in\Mpl(0, \infty)$, we have
\begin{align}
\|(R_\beta \circ R_\alpha^j)f\|_{X(0, \infty)} &\approx \Big\| \phi_{\beta + \alpha j}(t) \int_0^t f(s) K_j(s,t) \d{s} \Big\|_{X(0, \infty)}, \label{prop:iterated_operators_closed_form:E-RR}\\
\|(H_\alpha^j \circ H_\beta)f\|_{X(0, \infty)} &\approx \Big\| \int_t^\infty f(s) \phi_{\beta + \alpha j}(s) K_j(t,s) \d{s} \Big\|_{X(0, \infty)}, \label{prop:iterated_operators_closed_form:E-HH}\\
\intertext{and}
\|(R_\beta \circ H_\alpha^{j+1})f\|_{X(0, \infty)} &\approx \Big\| \phi_\beta(t) \int_0^t f(s) s \phi_{(j+1)\alpha}(s) \d{s} \Big\|_{X(0, \infty)} \nonumber\\
&\quad+ \Big\| \phi_\beta(t) t \int_t^\infty f(s) \phi_{(j+1)\alpha}(s) K_j(t,s) \d{s} \Big\|_{X(0, \infty)}. \label{prop:iterated_operators_closed_form:E-RH}
\end{align}
Furthermore, if $\alpha\in (0, \frac{n}{j+1})$, then we also have
\begin{align}
\|((H_\alpha \circ P)^{j+1}\circ H_\beta)f\|_{X(0, \infty)} &\approx \Big\| \phi_{(j+1)\alpha}(t) \int_0^t f(s) \phi_\beta(s)s K_j(s,t)  \d{s} \Big\|_{X(0, \infty)} \nonumber\\
&\quad+ \Big\| \int_t^\infty f(s) \phi_{\beta + (j+1)\alpha}(s) K_{j+1}(t,s) \d{s} \Big\|_{X(0, \infty)}. \label{prop:iterated_operators_closed_form:E-HPH}
\end{align}
The multiplicative constants in these equivalences depend only on $\alpha, \beta$, $n$, and $j$.
\end{proposition}
\begin{proof}
Our point of departure are the following pointwise estimates:
\begin{align}
K_j(2a,b) &\lesssim K_j\Big( a, \frac{b}{2} \Big) \quad \text{for all $0<a\leq \frac{b}{2}$},\label{prop:iterated_operators_closed_form:E1}\\
K_j(a,b) &\lesssim K_{j+1}(a,b) \quad \text{for all $0<a\leq \frac{b}{2}$},\label{prop:iterated_operators_closed_form:E6}\\
K_j(a,2a) &\approx 1 \quad \text{for every $a>0$},\label{prop:iterated_operators_closed_form:E5}\\
\Big( \int_a^b \phi_\alpha \Big)^j &\lesssim b \phi_{j\alpha}(b) K_j(a,b) \quad \text{for all $0<a\leq b$},\label{prop:iterated_operators_closed_form:E2}\\
\Big( \int_a^b \phi_\alpha \Big)^j &\gtrsim b \phi_{j\alpha}(b) K_j\Big( a,\frac{b}{2} \Big) \quad \text{for all $0<a\leq \frac{b}{2}$},\label{prop:iterated_operators_closed_form:E3}
\end{align}
and
\begin{equation}\label{prop:iterated_operators_closed_form:E4}
aK_j(a,b) \lesssim \int_0^a K_j(s,b) \d{s} \lesssim aK_j(a,2b) \quad \text{for all $0<a\leq b$}.
\end{equation}
The inequalities \eqref{prop:iterated_operators_closed_form:E1}--\eqref{prop:iterated_operators_closed_form:E6} and the equivalence \eqref{prop:iterated_operators_closed_form:E5} can be easily proved by analyzing the pointwise behavior of the kernel $K_j$, and we omit their proofs. The inequality \eqref{prop:iterated_operators_closed_form:E2} follows from \eqref{prop:properties_of_iterated_operators_kernel_phi_alpha_explicit}. We now turn our attention to \eqref{prop:iterated_operators_closed_form:E3}. First, for $0<a\leq \frac{b}{2} < b\leq 2$, we have
\begin{align*}
\Big( \int_a^b \phi_\alpha \Big)^j &\geq \Big( \int_{b/2}^b \phi_\alpha \Big)^j \approx \Big( \int_{b/2}^b s^{-1 + \frac{\alpha}{n}} \d{s} \Big)^j \approx b^{\frac{j\alpha}{n}} \approx b \phi_{j\alpha}(b) \\
&= b \phi_{j\alpha}(b) K_j\Big( a,\frac{b}{2} \Big).
\end{align*}
Second, for $0< a \leq 1$ and $b>2$, we have
\begin{equation*}
\Big( \int_a^b \phi_\alpha \Big)^j \geq \Big( \int_1^b s^{-1} \d{s} \Big)^j = \log(b)^j \approx \log\Big(\frac{e}{2}b\Big)^j = b \phi_{j\alpha}(b) K_j\Big( a,\frac{b}{2} \Big).
\end{equation*}
Finally, when $1<a\leq \frac{b}{2}$ and $b > 2$, we have
\begin{equation*}
\Big( \int_a^b \phi_\alpha \Big)^j = \Big( \int_a^b s^{-1} \d{s} \Big)^j = \log\Big( \frac{b}{a} \Big)^j \geq \log\Big( \frac{b}{2a} \Big)^j =  b \phi_{j\alpha}(b) K_j\Big( a,\frac{b}{2} \Big).
\end{equation*}
Putting all three estimates together, we obtain \eqref{prop:iterated_operators_closed_form:E3}. Next, we will prove the second inequality in \eqref{prop:iterated_operators_closed_form:E4}. When $0 < a\leq\min\{1,b\}$, we clearly have
\begin{equation}\label{prop:iterated_operators_closed_form:5}
\int_0^a K_j(s,b) \d{s} = \int_0^a K_j(a,b) \d{s} =  aK_j(a,b) \leq a K_j(a,2b).
\end{equation}
For $1<a\leq b$, we see that
\begin{align}
\int_0^a K_j(s,b) \d{s} &= \log(eb)^j + \int_1^a \log\Big( \frac{b}{s} \Big)^j \d{s} \nonumber\\
&= \log(eb)^j  + b\int_{b/a}^b \frac{\log(\tau)^j}{\tau^2} \d{\tau}. \label{prop:iterated_operators_closed_form:6}
\end{align}
As for the second term, we have
\begin{equation}
b\int_{b/a}^b \frac{\log(\tau)^j}{\tau^2} \d{\tau} \leq b\int_{b/a}^\infty \frac{\log(\tau)^j}{\tau^2} \d{\tau} \lesssim a\log\Big( \frac{2b}{a} \Big)^j = aK_j(a,2b)
\end{equation}
for all $1<a\leq b$. Furthermore, using \eqref{prop:properties_of_iterated_operators_monotonicity}, we obtain
\begin{equation}\label{prop:iterated_operators_closed_form:8}
\log(eb)^j \approx \log(e^jb)^j \leq a\log\Big( e^j\frac{b}{a} \Big)^j \approx a\log\Big( \frac{2b}{a} \Big)^j = aK_j(a,2b)
\end{equation}
for all $1<a\leq b$. Hence, the second inequality in \eqref{prop:iterated_operators_closed_form:E4} follows from \eqref{prop:iterated_operators_closed_form:5}--\eqref{prop:iterated_operators_closed_form:8}. It remains for us to prove the first inequality in \eqref{prop:iterated_operators_closed_form:E4}. In view of \eqref{prop:iterated_operators_closed_form:5}, we only need to prove it for $1<a\leq b$. To this end, note that
\begin{equation*}
\log(eb)^j \geq \log\Big( \frac{b}{a} \Big)^j \quad \text{for all $1<a\leq b$}.
\end{equation*}
Hence, using \eqref{prop:iterated_operators_closed_form:6}, we obtain
\begin{align*}
\int_0^a K_j(s,b) \d{s} &\geq \log(eb)^j + b\log\Big( \frac{b}{a} \Big)^j \int_{b/a}^b \frac1{\tau^2} \d{\tau} = \log(eb)^j + b\log\Big( \frac{b}{a} \Big)^j \frac{a-1}{b} \\
&\geq a\log\Big( \frac{b}{a} \Big)^j
\end{align*}
for all $1<a\leq b$. It follows that the first inequality in \eqref{prop:iterated_operators_closed_form:E4} is true.

Having \eqref{prop:iterated_operators_closed_form:E1}--\eqref{prop:iterated_operators_closed_form:E4} at our disposal, we now establish the desired equivalences. We start with \eqref{prop:iterated_operators_closed_form:E-RR}. On the one hand, using \eqref{prop:properties_of_iterated_operators_Ri_R_jk} and \eqref{prop:iterated_operators_closed_form:E2}, we obtain
\begin{align}
\|(R_\beta \circ R_\alpha^j)f\|_{X(0, \infty)} &\approx \Big\| \phi_\beta(t) \int_0^t f(s) \Big( \int_{s}^t \phi_\alpha \Big)^j \d{s} \Big\|_{X(0, \infty)} \nonumber\\
&\lesssim \Big\| \phi_\beta(t)t\phi_{j\alpha}(t) \int_0^t f(s) K_j(s,t) \d{s} \Big\|_{X(0, \infty)} \nonumber\\
&= \Big\| \phi_{\beta+j\alpha}(t) \int_0^t f(s) K_j(s,t) \d{s} \Big\|_{X(0, \infty)} \label{prop:iterated_operators_closed_form:9}
\end{align}
for every $f\in\Mpl(0, \infty)$. On the other hand, using \eqref{prop:properties_of_iterated_operators_Ri_R_jk}, \eqref{prop:iterated_operators_closed_form:E3}, and \eqref{prel:ri:dilation}, we have
\begin{align}
\|(R_\beta \circ R_\alpha^j)f\|_{X(0, \infty)} &\gtrsim \Big\| \phi_\beta\Big( \frac{t}{2} \Big) \int_0^{t/2} f(s) \Big( \int_{s}^t \phi_\alpha \Big)^j \d{s} \Big\|_{X(0, \infty)} \nonumber\\
&\gtrsim \Big\| \phi_{\beta+j\alpha}\Big( \frac{t}{2} \Big) \int_0^{t/2} f(s) K_j\Big( s, \frac{t}{2} \Big) \d{s} \Big\|_{X(0, \infty)} \nonumber\\
&\approx \Big\| \phi_{\beta+j\alpha}(t) \int_0^t f(s) K_j(s, t) \d{s} \Big\|_{X(0, \infty)} \label{prop:iterated_operators_closed_form:10}
\end{align}
for every $f\in\Mpl(0, \infty)$. Hence, combining \eqref{prop:iterated_operators_closed_form:9} and \eqref{prop:iterated_operators_closed_form:10}, we obtain \eqref{prop:iterated_operators_closed_form:E-RR}. The validity of \eqref{prop:iterated_operators_closed_form:E-HH} can be established in a similar way. On the one hand, thanks to \eqref{prop:properties_of_iterated_operators_Hjk_Hi} and \eqref{prop:iterated_operators_closed_form:E2}, we have
\begin{align*}
\|(H_\alpha^j \circ H_\beta)f\|_{X(0, \infty)} &\approx \Big\| \int_t^\infty f(s) \phi_\beta(s) \Big( \int_t^s \phi_\alpha \Big)^j \d{s} \Big\|_{X(0, \infty)} \\
&\lesssim \Big\| \int_t^\infty f(s) \phi_\beta(s) s \phi_{j\alpha}(s) K_j(t,s) \d{s} \Big\|_{X(0, \infty)} \\
&= \Big\| \int_t^\infty f(s) \phi_{\beta + j \alpha}(s) K_j(t,s) \d{s} \Big\|_{X(0, \infty)}
\end{align*}
for every $f\in\Mpl(0, \infty)$. On the other hand, using \eqref{prop:properties_of_iterated_operators_Hjk_Hi}, \eqref{prop:iterated_operators_closed_form:E3}, \eqref{prop:iterated_operators_closed_form:E1}, and \eqref{prel:ri:dilation}, we obtain
\begin{align*}
\|(H_\alpha^j \circ H_\beta)f\|_{X(0, \infty)} &\gtrsim \Big\| \int_{2t}^\infty f(s) \phi_\beta(s) \Big( \int_t^s \phi_\alpha \Big)^j \d{s} \Big\|_{X(0, \infty)}\\
&\gtrsim \Big\| \int_{2t}^\infty f(s) \phi_{\beta + j\alpha}(s) K_j\Big( t, \frac{s}{2} \Big) \d{s} \Big\|_{X(0, \infty)}\\
&\gtrsim \Big\| \int_{2t}^\infty f(s) \phi_{\beta + j\alpha}(s) K_j\Big( 2t, s \Big) \d{s} \Big\|_{X(0, \infty)} \\
&\approx \Big\| \int_{t}^\infty f(s) \phi_{\beta + j\alpha}(s) K_j(t,s) \d{s} \Big\|_{X(0, \infty)}
\end{align*}
for every $f\in\Mpl(0, \infty)$, which completes the proof of \eqref{prop:iterated_operators_closed_form:E-HH}.

Next, we prove \eqref{prop:iterated_operators_closed_form:E-RH}. For all $f\in\Mpl(0, \infty)$ and $t\in(0, \infty)$, we have
\begin{align}
(R_\beta \circ H_\alpha^{j+1})f(t) &\approx \phi_\beta(t) \int_0^t \int_s^\infty f(\tau) \phi_\alpha(\tau) \Big( \int_s^\tau \phi_\alpha \Big)^j \d{\tau} \d{s} \nonumber\\
&= \phi_\beta(t) \int_0^t \int_s^t f(\tau) \phi_\alpha(\tau) \Big( \int_s^\tau \phi_\alpha \Big)^j \d{\tau} \d{s} \nonumber\\
&\quad+ \phi_\beta(t) \int_0^t \int_t^\infty f(\tau) \phi_\alpha(\tau) \Big( \int_s^\tau \phi_\alpha \Big)^j \d{\tau} \d{s} \label{prop:iterated_operators_closed_form:11}
\end{align}
thanks to \eqref{prop:properties_of_iterated_operators_Hjk_Hi}. For future reference, note that, thanks to \eqref{prop:iterated_operators_closed_form:E2} and \eqref{prop:iterated_operators_closed_form:E4},
\begin{equation}\label{prop:iterated_operators_closed_form:12}
\int_0^a \Big( \int_s^b \phi_\alpha \Big)^j \d{s} \lesssim b\phi_{j\alpha}(b) \int_0^a K_j(s,b) \d{s} \lesssim a b\phi_{j\alpha}(b) K_j(a,2b)
\end{equation}
for all $0<a\leq b$ and that, thanks to \eqref{prop:iterated_operators_closed_form:E3} and \eqref{prop:iterated_operators_closed_form:E4},
\begin{equation}\label{prop:iterated_operators_closed_form:13}
\int_0^a \Big( \int_s^b \phi_\alpha \Big)^j \d{s} \gtrsim b\phi_{j\alpha}(b) \int_0^a K_j\Big( s, \frac{b}{2} \Big) \d{s} \gtrsim a b\phi_{j\alpha}(b) K_j\Big( a, \frac{b}{2} \Big)
\end{equation}
for all $0<a\leq \frac{b}{2}$. Now, using the Fubini theorem, \eqref{prop:iterated_operators_closed_form:12}, and \eqref{prop:iterated_operators_closed_form:E5}, we obtain
\begin{align}
\int_0^t \int_s^t f(\tau) \phi_\alpha(\tau) \Big( \int_s^\tau \phi_\alpha \Big)^j \d{\tau} \d{s} &= \int_0^t f(\tau) \phi_\alpha(\tau) \int_0^\tau \Big( \int_s^\tau \phi_\alpha \Big)^j \d{s} \d{\tau} \nonumber\\
&\lesssim \int_0^t f(\tau) \phi_\alpha(\tau) \tau^2 \phi_{j\alpha}(\tau) \d{\tau} \nonumber\\
&= \int_0^t f(\tau) \tau \phi_{(j+1)\alpha}(\tau) \d{\tau} \label{prop:iterated_operators_closed_form:14}
\end{align}
for all $f\in\Mpl(0, \infty)$ and $t\in(0, \infty)$. Furthermore, using the Fubini theorem, \eqref{prop:iterated_operators_closed_form:13}, and \eqref{prop:iterated_operators_closed_form:E5}, we obtain
\begin{align}
\int_0^t \int_s^t f(\tau) \phi_\alpha(\tau) \Big( \int_s^\tau \phi_\alpha \Big)^j \d{\tau} \d{s} &= \int_0^t f(\tau) \phi_\alpha(\tau) \int_0^\tau \Big( \int_s^\tau \phi_\alpha \Big)^j \d{s} \d{\tau} \nonumber\\
&\geq \int_0^t f(\tau) \phi_\alpha(\tau) \int_0^{\tau/4} \Big( \int_s^\tau \phi_\alpha \Big)^j \d{s} \d{\tau} \nonumber\\
&\gtrsim \int_0^t f(\tau) \tau \phi_{(j+1)\alpha}(\tau) K_j\Big(\frac{\tau}{4}, \frac{\tau}{2}\Big) \d{\tau} \nonumber\\
&\approx \int_0^t f(\tau) \tau \phi_{(j+1)\alpha}(\tau) \d{\tau} \label{prop:iterated_operators_closed_form:15}
\end{align}
for all $f\in\Mpl(0, \infty)$ and $t\in(0, \infty)$. Hence, it follows from \eqref{prop:iterated_operators_closed_form:14} and \eqref{prop:iterated_operators_closed_form:15} that
\begin{align}
&\Big\| \phi_\beta(t) \int_0^t \int_s^t f(\tau) \phi_\alpha(\tau) \Big( \int_s^\tau \phi_\alpha \Big)^j \d{\tau} \d{s}  \Big\|_{X(0,\infty)} \nonumber\\
&\approx \Big\| \phi_\beta(t) \int_0^t f(s) s \phi_{(j+1)\alpha}(s) \d{s}  \Big\|_{X(0,\infty)} \label{prop:iterated_operators_closed_form:16}
\end{align}
for every $f\in\Mpl(0, \infty)$. Furthermore, by Fubini theorem, we have
\begin{equation}\label{prop:iterated_operators_closed_form:17}
\int_0^t \int_t^\infty f(\tau) \phi_\alpha(\tau) \Big( \int_s^\tau \phi_\alpha \Big)^j \d{\tau} \d{s} = \int_t^\infty f(\tau) \phi_\alpha(\tau) \int_0^t \Big( \int_s^\tau \phi_\alpha \Big)^j \d{s} \d{\tau}
\end{equation}
for all $f\in\Mpl(0, \infty)$ and $t\in(0, \infty)$. On the one hand, using \eqref{prop:iterated_operators_closed_form:12}, \eqref{prop:iterated_operators_closed_form:E1}, and \eqref{prel:ri:dilation},
we have
\begingroup
\allowdisplaybreaks
\begin{align}
&\Big\| \phi_\beta(t) \int_t^\infty f(\tau) \phi_\alpha(\tau) \int_0^t \Big( \int_s^\tau \phi_\alpha \Big)^j \d{s} \d{\tau} \Big\|_{X(0, \infty)} \nonumber\\
&\lesssim \Big\| \phi_\beta(t) t \int_t^\infty f(\tau) \phi_{(j+1)\alpha}(\tau) K_j(t, 2\tau) \d{\tau} \Big\|_{X(0, \infty)} \nonumber\\
&\lesssim \Big\| \phi_\beta(t) t \int_t^\infty f(\tau) \phi_{(j+1)\alpha}(\tau) K_j\Big( \frac{t}{2}, \tau \Big) \d{\tau} \Big\|_{X(0, \infty)} \nonumber\\
&\lesssim \Big\| \phi_\beta\Big( \frac{t}{2} \Big) \frac{t}{2} \int_{t/2}^\infty f(\tau) \phi_{(j+1)\alpha}(\tau) K_j\Big( \frac{t}{2}, \tau \Big) \d{\tau} \Big\|_{X(0, \infty)} \nonumber\\
&\approx \Big\| \phi_\beta(t) t \int_t^\infty f(\tau) \phi_{(j+1)\alpha}(\tau) K_j(t, \tau) \d{\tau} \Big\|_{X(0, \infty)} \label{prop:iterated_operators_closed_form:18}
\end{align}
for every $f\in\Mpl(0, \infty)$. On the one hand, using \eqref{prop:iterated_operators_closed_form:13}, \eqref{prop:iterated_operators_closed_form:E1}, and \eqref{prel:ri:dilation}, we have
\begin{align}
&\Big\| \phi_\beta(t) \int_t^\infty f(\tau) \phi_\alpha(\tau) \int_0^t \Big( \int_s^\tau \phi_\alpha \Big)^j \d{s} \d{\tau} \Big\|_{X(0, \infty)} \nonumber\\
&\geq \Big\| \phi_\beta(t) \int_{2t}^\infty f(\tau) \phi_\alpha(\tau) \int_0^t \Big( \int_s^\tau \phi_\alpha \Big)^j \d{s} \d{\tau} \Big\|_{X(0, \infty)} \nonumber\\
&\gtrsim \Big\| \phi_\beta(t) t \int_{2t}^\infty f(\tau) \phi_{(j+1)\alpha}(\tau) K_j\Big( t, \frac{\tau}{2} \Big) \d{\tau} \Big\|_{X(0, \infty)} \nonumber\\
&\gtrsim \Big\| \phi_\beta(t) t \int_{2t}^\infty f(\tau) \phi_{(j+1)\alpha}(\tau) K_j(2t, \tau) \d{\tau} \Big\|_{X(0, \infty)} \nonumber\\
&\approx \Big\| \phi_\beta(t) t \int_t^\infty f(\tau) \phi_{(j+1)\alpha}(\tau) K_j(t, \tau) \d{\tau} \Big\|_{X(0, \infty)} \label{prop:iterated_operators_closed_form:19}
\end{align}
\endgroup
for every $f\in\Mpl(0, \infty)$. Hence, combining \eqref{prop:iterated_operators_closed_form:17}, \eqref{prop:iterated_operators_closed_form:18}, and \eqref{prop:iterated_operators_closed_form:19}, we arrive at
\begin{align}
&\Big\| \phi_\beta(t) \int_0^t \int_t^\infty f(\tau) \phi_\alpha(\tau) \Big( \int_s^\tau \phi_\alpha \Big)^j \d{\tau} \d{s} \Big\|_{X(0, \infty)} \nonumber\\
&\approx \Big\| \phi_\beta(t) t \int_t^\infty f(\tau) \phi_{(j+1)\alpha}(\tau) K_j(t, \tau) \d{\tau} \Big\|_{X(0, \infty)} \label{prop:iterated_operators_closed_form:20}
\end{align}
for every $f\in\Mpl(0, \infty)$. Therefore, \eqref{prop:iterated_operators_closed_form:E-RH} now follows from \eqref{prop:iterated_operators_closed_form:11}, \eqref{prop:iterated_operators_closed_form:16}, and \eqref{prop:iterated_operators_closed_form:20}.

Finally, we prove \eqref{prop:iterated_operators_closed_form:E-HPH}. In the rest of the proof, we assume that $\alpha\in (0, \frac{n}{j+1})$. Note that, thanks to \eqref{prop:properties_of_iterated_operators_(Hj_P)k},
\begin{align}
\|((H_\alpha \circ P)^{j+1}\circ H_\beta)f\|_{X(0, \infty)} &\approx \|(H_\alpha^{j+1} \circ P \circ H_\beta)f\|_{X(0, \infty)} \nonumber\\
&\quad+ \|(R_\alpha^{j+1}\circ H_\beta)f\|_{X(0, \infty)} \label{prop:iterated_operators_closed_form:27}
\end{align}
for every $f\in\Mpl(0, \infty)$. Note that, by the Fubini theorem,
\begin{align}
\int_0^t H_\beta f(s) K_j(s,t) \d{s} &= \int_0^t \int_s^t f(\tau) \phi_\beta(\tau) \d{\tau} K_j(s,t) \d{s} \nonumber\\
&\quad+ H_\beta f(t) \int_0^t K_j(s,t) \d{s} \label{prop:iterated_operators_closed_form:22}
\end{align}
for all $f\in\Mpl(0, \infty)$ and $t\in(0, \infty)$. Using \eqref{prop:iterated_operators_closed_form:E4} and \eqref{prop:iterated_operators_closed_form:E5}, we have
\begin{equation*}
H_\beta f(t) \int_0^t K_j(s,t) \d{s} \lesssim t H_\beta f(t) \quad \text{for all $f\in\Mpl(0, \infty)$ and $t\in(0, \infty)$}.
\end{equation*}
Using this, the fact that the function $(0,\infty)\ni t \mapsto t\phi_{(j+1)\alpha}(t)$ is nondecreasing, \eqref{prop:iterated_operators_closed_form:E5}, \eqref{prop:iterated_operators_closed_form:E1}, \eqref{E:kernel_K_monotonicity_second_var}, and \eqref{prel:ri:dilation}, we have
\begingroup
\allowdisplaybreaks
\begin{align}
\Big\|\phi_{(j+1)\alpha}(t) H_\beta f(t) \int_0^t K_j(s,t) \d{s} \Big\|_{X(0, \infty)} &\lesssim \Big\|\phi_{(j+1)\alpha}(t) t \int_t^\infty f(s) \phi_\beta(s) \d{s} \Big\|_{X(0, \infty)} \nonumber\\
&\leq \Big\|\int_t^\infty f(s) \phi_{\beta + (j+1)\alpha}(s) \d{s} \Big\|_{X(0, \infty)} \nonumber\\
&\lesssim \Big\|\int_t^\infty f(s) \phi_{\beta + (j+1)\alpha}(s) K_{j+1}\Big( \frac{t}{2}, t \Big) \d{s} \Big\|_{X(0, \infty)} \nonumber\\
&\leq \Big\|\int_{t/2}^\infty f(s) \phi_{\beta + (j+1)\alpha}(s) K_{j+1}\Big( \frac{t}{2}, s \Big) \d{s} \Big\|_{X(0, \infty)} \nonumber\\
&\approx \Big\|\int_t^\infty f(s) \phi_{\beta + (j+1)\alpha}(s) K_{j+1}(t, s) \d{s} \Big\|_{X(0, \infty)} \label{prop:iterated_operators_closed_form:23}
\end{align}
for every $f\in\Mpl(0, \infty)$. Furthermore, we also have
\begin{align}
&\Big\|\phi_{(j+1)\alpha}(t) \int_0^t \int_s^t f(\tau) \phi_\beta(\tau) \d{\tau} K_j(s,t) \d{s} \Big\|_{X(0, \infty)} \nonumber\\
&= \Big\|\phi_{(j+1)\alpha}(t) \int_0^t f(\tau) \phi_\beta(\tau) \int_0^\tau K_j(s,t) \d{s} \d{\tau}\Big\|_{X(0, \infty)} \nonumber\\
&\lesssim \Big\|\phi_{(j+1)\alpha}(t) \int_0^t f(\tau) \phi_\beta(\tau) \tau K_j(\tau, 2t) \d{\tau}\Big\|_{X(0, \infty)} \nonumber\\
&\lesssim \Big\|\phi_{(j+1)\alpha}(2t) \int_0^{2t} f(\tau) \phi_\beta(\tau) \tau K_j(\tau, 2t) \d{\tau}\Big\|_{X(0, \infty)} \nonumber\\
&\approx \Big\|\phi_{(j+1)\alpha}(t) \int_0^{t} f(\tau) \phi_\beta(\tau) \tau K_j(\tau, t) \d{\tau}\Big\|_{X(0, \infty)} \label{prop:iterated_operators_closed_form:24}
\end{align}
\endgroup
for every $f\in\Mpl(0, \infty)$, thanks to the Fubini theorem, \eqref{prop:iterated_operators_closed_form:E4}, and \eqref{prel:ri:dilation}. Hence, combining \eqref{prop:iterated_operators_closed_form:E-RR}, \eqref{prop:iterated_operators_closed_form:22}, \eqref{prop:iterated_operators_closed_form:23}, and \eqref{prop:iterated_operators_closed_form:24}, we obtain
\begin{align}
\|(R_\alpha^{j+1}\circ H_\beta)f\|_{X(0, \infty)} &\lesssim \Big\|\int_t^\infty f(s) \phi_{\beta + (j+1)\alpha}(s) K_{j+1}(t, s) \d{s} \Big\|_{X(0, \infty)} \nonumber\\
&\quad+ \Big\|\phi_{(j+1)\alpha}(t) \int_0^{t} f(\tau) \phi_\beta(\tau) \tau K_j(\tau, t) \d{\tau}\Big\|_{X(0, \infty)} \label{prop:iterated_operators_closed_form:25}
\end{align}
for every $f\in\Mpl(0, \infty)$. On the other hand, using \eqref{prop:iterated_operators_closed_form:E-RR}, \eqref{prop:iterated_operators_closed_form:22}, the Fubini theorem, and \eqref{prop:iterated_operators_closed_form:E4}, we also have
\begin{align}
\|(R_\alpha^{j+1}\circ H_\beta)f\|_{X(0, \infty)} &\gtrsim \Big\|\phi_{(j+1)\alpha}(t) \int_0^t f(\tau) \phi_\beta(\tau) \int_0^\tau K_j(s,t) \d{s} \d{\tau}\Big\|_{X(0, \infty)} \nonumber\\
&\gtrsim \Big\|\phi_{(j+1)\alpha}(t) \int_0^{t} f(\tau) \phi_\beta(\tau) \tau K_j(\tau, t) \d{\tau}\Big\|_{X(0, \infty)} \label{prop:iterated_operators_closed_form:26}
\end{align}
for every $f\in\Mpl(0, \infty)$. We now turn our attention to the first term on the right-hand side of \eqref{prop:iterated_operators_closed_form:27}. Note that
\begin{equation}\label{prop:iterated_operators_closed_form:28}
(P \circ H_\beta)f(t) = H_\beta f(t) + \frac1{t}\int_0^t f(s) \phi_\beta(s) s \d{s} 
\end{equation}
for all $f\in\Mpl(0, \infty)$ and $t\in(0, \infty)$ thanks to the Fubini theorem. By \eqref{prop:iterated_operators_closed_form:E-HH} and the Fubini theorem, we have
\begingroup
\allowdisplaybreaks
\begin{align}
&\Big\|  H_\alpha^{j+1} \Big( \frac1{\cdot}\int_0^{\cdot} f(\tau) \phi_\beta(\tau) \tau \d{\tau} \Big)  \Big\|_{X(0, \infty)} \nonumber\\
&\approx \Big\| \int_t^\infty \Big( \frac1{s}\int_0^s f(\tau) \phi_\beta(\tau) \tau \d{\tau} \Big) \phi_{(j+1)\alpha}(s) K_j(t,s) \d{s} \Big\|_{X(0, \infty)} \nonumber\\
&\approx \Big\| \Big( \int_0^t f(\tau) \phi_\beta(\tau) \tau \d{\tau} \Big) \int_t^\infty \frac{\phi_{(j+1)\alpha}(s)}{s} K_j(t,s) \d{s} \Big\|_{X(0, \infty)} \nonumber\\
&\quad+ \Big\| \int_t^\infty \Big( \frac1{s}\int_t^s f(\tau) \phi_\beta(\tau) \tau \d{\tau} \Big) \phi_{(j+1)\alpha}(s) K_j(t,s) \d{s} \Big\|_{X(0, \infty)} \nonumber\\
&= \Big\| \Big( \int_0^t f(\tau) \phi_\beta(\tau) \tau \d{\tau} \Big) \int_t^\infty \frac{\phi_{(j+1)\alpha}(s)}{s} K_j(t,s) \d{s} \Big\|_{X(0, \infty)} \nonumber\\
&\quad+ \Big\| \int_t^\infty f(\tau) \phi_\beta(\tau) \tau \Big( \int_\tau^\infty \frac{\phi_{(j+1)\alpha}(s)}{s} K_j(t,s) \d{s} \Big) \d{\tau} \Big\|_{X(0, \infty)} \label{prop:iterated_operators_closed_form:29}
\end{align}
for every $f\in\Mpl(0, \infty)$. Furthermore, we claim that
\begin{equation}\label{prop:iterated_operators_closed_form:21}
\phi_{(j+1)\alpha}(b)K_j(a,b) \lesssim \int_b^\infty \frac{\phi_{(j+1)\alpha}(s)}{s} K_j(a,s)\d{s} \lesssim \phi_{(j+1)\alpha}(b)K_j(a,2b)
\end{equation}
for all $0<a\leq b$. For now, assume that we have already established \eqref{prop:iterated_operators_closed_form:21}. Now, using \eqref{prop:iterated_operators_closed_form:21}, \eqref{prop:iterated_operators_closed_form:E5} twice, \eqref{E:kernel_K_monotonicity_second_var}, and \eqref{prel:ri:dilation}, we obtain
\begin{align}
&\Big\| \Big( \int_0^t f(\tau) \phi_\beta(\tau) \tau \d{\tau} \Big) \int_t^\infty \frac{\phi_{(j+1)\alpha}(s)}{s} K_j(t,s) \d{s} \Big\|_{X(0, \infty)} \nonumber\\
&\lesssim \Big\| \phi_{(j+1)\alpha}(t)K_j(t,2t) \int_0^t f(\tau) \phi_\beta(\tau) \tau \d{\tau} \Big\|_{X(0, \infty)} \nonumber\\
&\approx \Big\| \phi_{(j+1)\alpha}(t) \int_0^t f(\tau) \phi_\beta(\tau) \tau K_j(\tau,2\tau) \d{\tau} \Big\|_{X(0, \infty)} \nonumber\\
&\lesssim \Big\| \phi_{(j+1)\alpha}(2t) \int_0^{2t} f(\tau) \phi_\beta(\tau) \tau K_j(\tau,2t) \d{\tau} \Big\|_{X(0, \infty)} \nonumber\\
&\approx \Big\| \phi_{(j+1)\alpha}(t) \int_0^t f(\tau) \phi_\beta(\tau) \tau K_j(\tau,t) \d{\tau} \Big\|_{X(0, \infty)} \label{prop:iterated_operators_closed_form:30}
\end{align}
for every $f\in\Mpl(0, \infty)$. Furthermore, using \eqref{prop:iterated_operators_closed_form:21}, \eqref{prop:iterated_operators_closed_form:E1}, \eqref{prop:iterated_operators_closed_form:E6}, and \eqref{prel:ri:dilation}, we also have
\begin{align}
&\Big\| \int_t^\infty f(\tau) \phi_\beta(\tau) \tau \Big( \int_\tau^\infty \frac{\phi_{(j+1)\alpha}(s)}{s} K_j(t,s) \d{s} \Big) \d{\tau} \Big\|_{X(0, \infty)} \nonumber\\
&\lesssim \Big\| \int_t^\infty f(\tau) \phi_{\beta + (j+1)\alpha}(\tau) K_j(t,2\tau) \d{\tau} \Big\|_{X(0, \infty)} \nonumber\\
&\lesssim \Big\| \int_t^\infty f(\tau) \phi_{\beta + (j+1)\alpha}(\tau) K_j\Big( \frac{t}{2}, \tau \Big) \d{\tau} \Big\|_{X(0, \infty)} \nonumber\\
&\lesssim \Big\| \int_t^\infty f(\tau) \phi_{\beta + (j+1)\alpha}(\tau) K_{j+1}\Big( \frac{t}{2}, \tau \Big) \d{\tau} \Big\|_{X(0, \infty)} \nonumber\\
&\leq \Big\| \int_{t/2}^\infty f(\tau) \phi_{\beta + (j+1)\alpha}(\tau) K_{j+1}\Big( \frac{t}{2}, \tau \Big) \d{\tau} \Big\|_{X(0, \infty)} \nonumber\\
&\approx \Big\| \int_t^\infty f(\tau) \phi_{\beta + (j+1)\alpha}(\tau) K_{j+1}(t, \tau) \d{\tau} \Big\|_{X(0, \infty)} \label{prop:iterated_operators_closed_form:32}
\end{align}
\endgroup
for every $f\in\Mpl(0, \infty)$. Hence, combining \eqref{prop:iterated_operators_closed_form:28}, \eqref{prop:iterated_operators_closed_form:E-HH}, \eqref{prop:iterated_operators_closed_form:29}, \eqref{prop:iterated_operators_closed_form:30}, and \eqref{prop:iterated_operators_closed_form:32}, we obtain
\begin{align}
\|(H_\alpha^{j+1} \circ P \circ H_\beta)f\|_{X(0, \infty)} &\lesssim \Big\| \phi_{(j+1)\alpha}(t) \int_0^t f(\tau) \phi_\beta(\tau) \tau K_j(\tau,t) \d{\tau} \Big\|_{X(0, \infty)} \nonumber\\
&\quad+ \Big\| \int_t^\infty f(\tau) \phi_{\beta + (j+1)\alpha}(\tau) K_{j+1}(t, \tau) \d{\tau} \Big\|_{X(0, \infty)} \label{prop:iterated_operators_closed_form:33}
\end{align}
for every $f\in\Mpl(0, \infty)$. On the other hand, using \eqref{prop:iterated_operators_closed_form:28} and \eqref{prop:iterated_operators_closed_form:E-HH}, we have
\begin{align}
\|(H_\alpha^{j+1} \circ P \circ H_\beta)f\|_{X(0, \infty)} &\geq \|(H_\alpha^{j+1} \circ H_\beta)f\|_{X(0, \infty)} \nonumber\\
&\approx \Big\| \int_t^\infty f(s) \phi_{\beta + (j+1)\alpha}(s) K_{j+1}(t, s) \d{s} \Big\|_{X(0, \infty)} \label{prop:iterated_operators_closed_form:34}
\end{align}
for every $f\in\Mpl(0, \infty)$. Finally, combining \eqref{prop:iterated_operators_closed_form:27}, \eqref{prop:iterated_operators_closed_form:25}, \eqref{prop:iterated_operators_closed_form:33}, \eqref{prop:iterated_operators_closed_form:26}, and \eqref{prop:iterated_operators_closed_form:34}, we obtain \eqref{prop:iterated_operators_closed_form:E-HPH}. Therefore, the entire proof will be finished once we establish \eqref{prop:iterated_operators_closed_form:21}. First, assume that $0<a\leq1$. For every $b\geq 1$, we have
\begin{align}
\int_b^\infty \frac{\phi_{(j+1)\alpha}(s)}{s} K_j(a,s)\d{s} &= \int_b^\infty \frac{\log(e s)^j}{s^2}\d{s} \approx \frac{\log(e b)^j}{b} \nonumber\\
&=  \phi_{(j+1)\alpha}(b) K_j(a,b). \label{prop:iterated_operators_closed_form:35}
\end{align}
If $0<a\leq b\leq1$, then
\begin{align}
\int_b^\infty \frac{\phi_{(j+1)\alpha}(s)}{s} K_j(a,s)\d{s} &= \int_b^1 s^{-2 + \frac{(j+1)\alpha}{n}}\d{s} + \int_1^\infty \frac{\log(e s)^j}{s^2}\d{s} \nonumber\\
&\approx b^{-1 + \frac{(j+1)\alpha}{n}} = \phi_{(j+1)\alpha}(b) K_j(a,b). \label{prop:iterated_operators_closed_form:36}
\end{align}
Note that here we use the assumption that $\alpha\in (0, \frac{n}{j+1})$. Combining \eqref{prop:iterated_operators_closed_form:35} and \eqref{prop:iterated_operators_closed_form:36} and using \eqref{E:kernel_K_monotonicity_second_var}, we obtain the validity of \eqref{prop:iterated_operators_closed_form:21} for all $0<a\leq \min\{1,b\}$. At last, let $1<a\leq b$. Note that
\begin{equation}\label{prop:iterated_operators_closed_form:37}
\int_b^\infty \frac{\phi_{(j+1)\alpha}(s)}{s} K_j(a,s)\d{s} = \int_b^\infty \frac{\log(\frac{s}{a})^j}{s^2}\d{s} = \frac1{a} \int_{b/a}^\infty \frac{\log(\tau)^j}{\tau^2}\d{\tau}.
\end{equation}
Now, on the one hand, we have
\begin{equation}\label{prop:iterated_operators_closed_form:38}
\frac1{a} \int_{b/a}^\infty \frac{\log(\tau)^j}{\tau^2}\d{\tau} \leq \frac1{a} \int_{b/a}^\infty \frac{\log(2\tau)^j}{\tau^2}\d{\tau} \approx \frac{\log(\frac{2b}{a})}{b} = \phi_{(j+1)\alpha}(b) K_j(a,2b).
\end{equation}
On the other hand, we have
\begin{equation}\label{prop:iterated_operators_closed_form:39}
\frac1{a} \int_{b/a}^\infty \frac{\log(\tau)^j}{\tau^2}\d{\tau} \geq \frac{\log(\frac{b}{a})}{a} \int_{b/a}^\infty \frac1{\tau^2}\d{\tau} \approx \frac{\log(\frac{b}{a})}{b} = \phi_{(j+1)\alpha}(b) K_j(a,b).
\end{equation}
By combining \eqref{prop:iterated_operators_closed_form:37} with \eqref{prop:iterated_operators_closed_form:38} and \eqref{prop:iterated_operators_closed_form:39}, we establish the validity of \eqref{prop:iterated_operators_closed_form:21} also for all $1 < a\leq b$, which concludes the proof.
\end{proof}

The following proposition characterizes when certain functionals are rearrangement\hyp{}invariant function norms.
\begin{proposition}\label{prop:operator_induced_norms}
Let $\|\cdot\|_{X(0, \infty)}$ be a rearrangement\hyp{}invariant function norm, $m < n$, $m\in\N$. Set $k = \lceil m/2 - 1 \rceil$.
Set
\begin{align*}
\nu_{m, X}(f) &=\begin{cases}
	\|(R_1\circ (H_2\circ P)^k) f^*\|_{X(0, \infty)}\quad &\text{if $m$ is odd}, \\
	\|(H_2\circ P)^{k+1} f^*\|_{X(0, \infty)}\quad &\text{if $m$ is even},
\end{cases}\\
\sigma_{m, X}(f) &= \|(R_\beta\circ R_2^k) f^*\|_{X(0, \infty)}, \\
\lambda_{m, X}(f) &= \|R_m f^*\|_{X(0, \infty)}, \\
\mu_{m, X}(f) &=\begin{cases}
	\|(R_1\circ H_2^k) f^{**}\|_{X(0, \infty)}\quad &\text{if $m$ is odd}, \\
	\|H_2^{k+1} f^{**}\|_{X(0, \infty)}\quad &\text{if $m$ is even},
\end{cases}\\
\end{align*}
for $f\in\Mpl(0, \infty)$, where $\beta$ is defined by \eqref{thm:optimal_target_general_theorem_beta}.
\begin{enumerate}[label=(\roman*), ref=(\roman*)]
\item The functionals $\nu_{m, X}$ and $\sigma_{m, X}$ are rearrangement\hyp{}invariant function norms if and only if
\begin{equation}\label{prop:operator_induced_norms_nu_sigma_condition}
\frac{(1 + \log t)^k}{t}\chi_{(1, \infty)}(t) \in X(0, \infty).
\end{equation}
\item The functional $\lambda_{m, X}$ is a rearrangement\hyp{}invariant function norm if and only if
\begin{equation}\label{prop:operator_induced_norms_lambda_condition}
\frac1{t}\chi_{(1, \infty)}(t) \in X(0, \infty).
\end{equation}
\item When $m$ is odd, the functional $\mu_{m, X}$ is a rearrangement\hyp{}invariant function norm if and only if \eqref{prop:operator_induced_norms_nu_sigma_condition} with $k = 1$ is satisfied. When $m$ is even, the functional is a rearrangement\hyp{}invariant function norm if and only if \eqref{prop:operator_induced_norms_lambda_condition} is satisfied. 
\end{enumerate}
\end{proposition}
\begin{proof}
For future reference, we start with two observations. First, note that, in view of \eqref{prop:properties_of_iterated_operators_(Hj_P)k} and \eqref{prop:properties_of_iterated_operators_Ri_(Hj_P)k}, we have
\begin{equation}\label{prop:operator_induced_norms_sigma_lesssim_nu}
\sigma_{m, X}(f)\lesssim \nu_{m, X}(f) \quad \text{for every $f\in\Mpl(0, \infty)$}, 
\end{equation}
in which the multiplicative constants depend only on $n$ and $m$. Second, we claim that
\begin{equation}\label{prop:operator_induced_norms_lambda_lesssim_sigma}
\lambda_{m, X}(f) \lesssim \sigma_{m, X}(f)  \quad \text{for every $f\in\Mpl(0, \infty)$},
\end{equation}
in which the multiplicative constants depend only on $n$ and $m$. Indeed, using \eqref{prop:iterated_operators_closed_form:E-RR}, \eqref{E:kernel_K_monotonicity_first_var}, \eqref{prop:iterated_operators_closed_form:E5}, and \eqref{prel:ri:dilation}, we obtain
\begin{align*}
\sigma_{m, X}(f) = \|(R_\beta \circ R_2^k)f^*\|_{X(0, \infty)} &\approx \Big\| \phi_m(t) \int_0^t f^*(s) K_k(s,t) \d{s} \Big\|_{X(0, \infty)} \\
&\geq \Big\| \phi_m(t) \int_0^{t/2} f^*(s) K_k\Big( \frac{t}{2}, t \Big) \d{s} \Big\|_{X(0, \infty)} \\
&\approx \Big\| \phi_m(t) \int_0^{t/2} f^*(s) \d{s} \Big\|_{X(0, \infty)} \\
&\approx \Big\| \phi_m(t) \int_0^t f^*(s) \d{s} \Big\|_{X(0, \infty)} \\
&=\lambda_{m, X}(f)
\end{align*}
for every $f\in\Mpl(0, \infty)$.

Now, we shall show that the functionals $\nu_{m, X}$, $\sigma_{m, X}$, $\lambda_{m, X}$, and $\mu_{m, X}$ are subadditive and satisfy the properties (P4) and (P5) of rearrangement\hyp{}invariant function norms\textemdash the remaining properties can be readily verified. We start with the subadditivity. The subadditivity of the functionals follows from \eqref{prel:ri:subadditivityofdoublestar}. Indeed, it implies that
\begin{equation*}
R_\alpha(f + g)^* \leq R_\alpha f^* + R_\alpha g^*
\end{equation*}
for every $f,g\in\Mpl(0, \infty)$ and every $\alpha\in[0, n)$, whence the subadditivity follows.

We now turn our attention to the property (P4). As for the functional $\lambda_{m, X}$, since $\chi_E^* = \chi_{(0, |E|)}$ for every measurable $E\subseteq(0, \infty)$, it is sufficient to prove that
\begin{equation}\label{prop:operator_induced_norms_lambda_finite_on_characteristic_functions}
\lambda_{m, X}(\chi_{(0, a)}) < \infty \quad \text{for every $a\in(1, \infty)$}
\end{equation}
if \eqref{prop:operator_induced_norms_lambda_condition} is satisfied. Note that, for $a\in[1, \infty)$,
\begin{equation*}
R_m(\chi_{(0, a)})(t) = t^{\frac{m}{n}}\chi_{(0, 1)}(t) + \chi_{[1, a)}(t) + \frac{a}{t}\chi_{[a, \infty)}(t) \quad \text{for every $t\in(0, \infty)$}.
\end{equation*}
Using that with the boundedness of the dilation operator $D_a$ on $X(0, \infty)$ (see \eqref{prel:ri:dilation}), it follows that \eqref{prop:operator_induced_norms_lambda_finite_on_characteristic_functions} is satisfied if \eqref{prop:operator_induced_norms_lambda_condition} is satisfied. It also follows that, if \eqref{prop:operator_induced_norms_lambda_condition} is not satisfied, then $\lambda_{m, X}(\chi_{(0, 1)}) = \infty$, and so $\lambda_{m, X}$ does not have the property (P4). As for the functionals $\nu_{m, X}$ and $\sigma_{m, X}$, thanks to \eqref{prop:operator_induced_norms_sigma_lesssim_nu}, it sufficient to prove that
\begin{equation*}
\nu_{m, X}(\chi_{(0, a)}) < \infty \quad \text{for every $a\in(1, \infty)$}
\end{equation*}
if \eqref{prop:operator_induced_norms_nu_sigma_condition} is satisfied.
Note that, for $a\in[1, \infty)$,
\begin{equation*}
P(\chi_{(0, a)})(t) = \chi_{(0, a)}^{**}(t) = \chi_{(0, a)}(t) + \frac{a}{t}\chi_{[a, \infty)}(t) \quad \text{for every $t\in(0, \infty)$}.
\end{equation*}
Carrying out elementary computations, we see that
\begin{align*}
(H_2\circ P)^{k+1}\chi_{(0, a)}(t) &\lesssim (1 + \log a)^{k+1} \Big( \chi_{(0, a)}(t) + a\frac{(1 + \log\big( \frac{t}{a}\big))^k}{t}\chi_{[a, \infty)}(t) \Big) \\
\intertext{and}
(R_\beta\circ (H_2\circ P)^k)\chi_{(0, a)}(t) &\lesssim (1 + \log a)^k \Big( t\phi_\beta(t)\chi_{(0, a)}(t) + a\frac{(1 + \log\big( \frac{t}{a}\big))^k}{t}\chi_{[a, \infty)}(t) \Big)
\end{align*}
for every $t\in(0, \infty)$, in which the multiplicative constants depend only on $m$ and $n$. That together with the boundedness of the dilation operator $D_a$ on $X(0, \infty)$ (see \eqref{prel:ri:dilation}) implies that, if \eqref{prop:operator_induced_norms_nu_sigma_condition} is satisfied, then $\nu_{m, X}$ has the property (P4) (and so does $\sigma_{m, X}$). On the other hand, if \eqref{prop:operator_induced_norms_nu_sigma_condition} is not satisfied, then the functional $\sigma_{m, X}$ does not have the property (P4) (and, in turn, neither does $\nu_{m, X}$). Indeed, using \eqref{prop:properties_of_iterated_operators_Ra_k} and \eqref{prop:properties_of_iterated_operators_Ri_R_jk}, we observe that
\begin{align}
R_2^{k+1}\chi_{(0, 1)}(t) &\geq \frac1{t k!}\Big( \int_0^{\frac{1}{2}} \Big( \int_{\frac{1}{2}}^t \phi_2 \Big)^k \d{s} \Big) \chi_{(1, \infty)}(t) \notag\\
&\approx \frac{(1 + \log t)^k}{t}\chi_{(1, \infty)}(t) \label{prop:operator_induced_norms_lower_bound_on_R2k_chi_0_1}
\intertext{and}
(R_1\circ R_2^k)\chi_{(0, 1)}(t) &\gtrsim \frac{(1 + \log t)^k}{t}\chi_{(1, \infty)}(t) \notag
\end{align}
for every $t\in(0, \infty)$, in which the multiplicative constants depend only on $n$ and $m$; it follows that $\sigma_{m, X}(\chi_{(0, 1)}) = \infty$. As for $\mu_{m, X}$, the proof is similar, so we only sketch it. Using \eqref{prop:properties_of_iterated_operators_Hjk_Hi}, elementary computations show that, for $a\in[1, \infty)$ and $t\in(0, \infty)$,
\begin{align*}
(H_2^{k + 1} \circ P) \chi_{(0, a)}(t) &\approx b_a(t)\chi_{(0, a)}(t) + \frac1{t}\chi_{[a, \infty)}(t) \\
\intertext{and}
(R_1 \circ H_2^k \circ P) \chi_{(0, a)}(t) &\approx \widetilde{b}_a(t)\chi_{(0, a)}(t) + \frac{1 + \log t}{t}\chi_{[a, \infty)}(t),
\end{align*}
in which $b_a$ and $\widetilde{b}_a$ are bounded positive continuous functions on $(0, \infty)$, whose supremum norms depend on $a$, and in which the multiplicative constants depend on $m$, $n$ and $a$. Combining that with the boundedness of the dilation operator $D_a$ on $X(0, \infty)$, it follows that the functional $\mu_{m, X}$ has the property (P4) if and only if, when $m$ is even, \eqref{prop:operator_induced_norms_lambda_condition} is satisfied; when $m$ is odd, \eqref{prop:operator_induced_norms_nu_sigma_condition} with $k = 1$ is satisfied. 

It remains to prove that the functionals $\nu_{m, X}$, $\sigma_{m, X}$, $\lambda_{m, X}$, and $\mu_{m, X}$ have the property (P5). As for the first three, in view of \eqref{prop:operator_induced_norms_sigma_lesssim_nu}, \eqref{prop:operator_induced_norms_lambda_lesssim_sigma} and \eqref{prel:ri:HLg=chiE}, it is sufficient to prove that
\begin{equation*}
\int_0^a f^*(s) \d{s} \leq C_a \lambda_m(f) \quad \text{for every $f\in\Mpl(0, \infty)$},
\end{equation*}
in which the multiplicative constant $C_a$ may depend on $a\in(0, \infty)$ but not on $f$. However, that follows immediately from
\begin{equation*}
\|R_m f^*\|_{X(0, \infty)} \geq \phi_m(2a) \int_0^a f^*(s) \d{s} \|\chi_{(a, 2a)}\|_{X(0, \infty)}
\end{equation*}
for every $f\in\Mpl(0, \infty)$. As for $\mu_{m, X}$, note that $\mu_{1, X}(f)\geq \lambda_{1, X}(f)$ for every $f\in\Mpl(0, \infty)$, and so $\mu_{1, X}$ has the property (P5), too. Assume that $m\geq 2$. Using \eqref{prop:properties_of_iterated_operators_Hjk_Hi}, it can be readily verified that
\begin{align*}
\|H_2^{k + 1}f^{**}\|_{X(0, \infty)} &\geq \frac1{k!}\frac{a}{2} \phi_2(a) \Big(\int_{\frac{a}{4}}^{\frac{a}{2}} \phi_2\Big)^k f^{**}(a) \|\chi_{(0,\frac{a}{4})}\|_{X(0, \infty)}\\
\intertext{and}
\|(P_1 \circ H_2^k)f^{**}\|_{X(0, \infty)} &\geq \frac{\phi_1(\frac{a}{2})}{(k - 1)!}\frac{a^2}{8} \phi_2(a) \Big(\int_{\frac{a}{4}}^{\frac{a}{2}} \phi_2\Big)^{k - 1} f^{**}(a) \|\chi_{(\frac{a}{4},\frac{a}{2})}\|_{X(0, \infty)}
\end{align*}
for every $a\in(0, \infty)$ and every $f\in\Mpl(0, \infty)$. It follows that $\mu_{m, X}$ has the property (P5).
\end{proof}

\begin{remark}\label{rem:operator_induced_norms_closed_forms}\ 
\begin{enumerate}[label=(\roman*), ref=(\roman*)]
\item It is not hard to see that $\nu_{m, X}(f) = \sigma_{m, X}(f) = \lambda_{m, X}(f) = \mu_{m, X}(f) =\infty$ for every function $f\in\Mpl(0, \infty)\setminus\{0\}$ when the functionals lack the property (P4). The functionals have the property (P4) if and only if the respective condition from \cref{prop:operator_induced_norms} is satisfied. Furthermore, note that $\nu_{1, X} = \sigma_{1, X} = \lambda_{1, X}$, that $\sigma_{2, X} = \lambda_{2, X}$, and that $\nu_{2, X} = \mu_{2, X}$. Finally, the statement of \cref{prop:operator_induced_norms} is also true (with virtually the same proof) even when the subscripts $2$ and $\beta$ in the definitions of $\nu_{m, X}$, $\sigma_{m, X}$, and $\mu_{m, X}$ are replaced by any $\alpha, \gamma\in(0, n)$. The subscript $m$, too, can be replaced by any $\alpha\in(0, n)$ in the definition of $\lambda_{m, X}$.
\item\label{rem:operator_induced_norms_closed_forms:item} Assume that $m\geq3$, and define the functionals $\widetilde{\nu}_{m, X}$, $\widetilde{\sigma}_{m, X}$, and $\widetilde{\mu}_{m, X}$ as follows. For every $f\in\Mpl(0, \infty)$, the functional $\widetilde{\sigma}_{m, X}$ is defined as
\begin{equation*}
\widetilde{\sigma}_{m, X}(f) = \Big\| \phi_m(t) \int_0^t f^*(s) K_k(s,t) \d{s} \Big\|_{X(0, \infty)}.
\end{equation*}
For every $f\in\Mpl(0, \infty)$, the functional $\widetilde{\mu}_{m, X}$ is defined as
\begin{equation*}
\widetilde{\mu}_{m, X}(f) = \Big\| \int_t^\infty f^{**}(s) \phi_{m}(s) K_k(t,s) \d{s} \Big\|_{X(0, \infty)}
\end{equation*}
when $m$ is even and as
\begin{align*}
\widetilde{\mu}_{m, X}(f) &= \Big\| \phi_1(t) \int_0^t f^{**}(s) s \phi_{m-1}(s) \d{s} \Big\|_{X(0, \infty)} \\
&\quad+ \Big\| \phi_1(t) t \int_t^\infty f^{**}(s) \phi_{m-1}(s) K_{k-1}(t,s) \d{s} \Big\|_{X(0, \infty)}
\end{align*}
when $m$ is odd. Finally, for every $f\in\Mpl(0, \infty)$, the functional $\widetilde{\nu}_{m, X}$ is defined as
\begin{align*}
\widetilde{\nu}_{m, X}(f) &= \Big\| \phi_m(t) \int_0^t f^*(s) K_k(s,t) \d{s} \Big\|_{X(0, \infty)} \\
&\quad+ \Big\| \int_t^\infty f^{**}(s) \phi_m(s) K_k(t,s) \d{s} \Big\|_{X(0, \infty)}
\end{align*}
when $m$ is even and as
\begin{align*}
\widetilde{\nu}_{m, X}(f) &= \Big\| \phi_1(t) \int_0^t f^{**}(s) s \phi_{m-1}(s) \d{s} \Big\|_{X(0, \infty)}\\
&\quad+ \Big\| \phi_1(t) t \int_t^\infty f^{**}(s) \phi_{m-1}(s) K_{k-1}(t,s) \d{s} \Big\|_{X(0, \infty)}\\
&\quad+ \Big\| \phi_m(t) \int_0^t f^*(s) K_k(s,t) \d{s} \Big\|_{X(0, \infty)}
\end{align*}
when $m$ is odd. It follows from \cref{prop:properties_of_iterated_operators,prop:iterated_operators_closed_form} that
\begin{align*}
\widetilde{\sigma}_{m, X}(f) &\approx \sigma_{m, X}(f) \quad \text{for every $f\in\Mpl(0, \infty)$},\\
\widetilde{\mu}_{m, X}(f) &\approx \mu_{m, X}(f) \quad \text{for every $f\in\Mpl(0, \infty)$},
\intertext{and}
\widetilde{\nu}_{m, X}(f) &\approx \nu_{m, X}(f) \quad \text{for every $f\in\Mpl(0, \infty)$},
\end{align*}
in which the multiplicative constants depend only on $m$ and $n$. Furthermore, the functional $\widetilde{\nu}_{m, X}$, $\widetilde{\sigma}_{m, X}$, or $\widetilde{\mu}_{m, X}$ is a rearrangement\hyp{}invariant function norms if and only if $\nu_{m, X}$, $\sigma_{m, X}$, or $\mu_{m, X}$ is, respectively. The fact that the functionals $\widetilde{\nu}_{m, X}$, $\widetilde{\sigma}_{m, X}$, and $\widetilde{\mu}_{m, X}$ are subadditive follows from the subadditivity of $f\mapsto f^{**}$, \eqref{E:kernel_K_monotonicity_first_var}, and Hardy's lemma \cite[Chapter~2, Proposition~3.6]{BS}.
\end{enumerate}
\end{remark}

The following proposition, which concludes this section, shows that the functionals defined in \cref{prop:operator_induced_norms} are sometimes equivalent.
\begin{proposition}\label{prop:operator_induced_norms_equivalences}
Let $\|\cdot\|_{X(0, \infty)}$ be a rearrangement\hyp{}invariant function norm, $m < n$, $m\in\N$. Let $\nu_{m, X}$, $\sigma_{m, X}$, $\lambda_{m, X}$, and $\mu_{m, X}$ be the functionals defined in \cref{prop:operator_induced_norms}.
\begin{enumerate}[label=(\roman*), ref=(\roman*)]
\item\label{prop:operator_induced_norms_equivalences_nu_sigma_item} If the operator $f\mapsto f^{**}$ is bounded on $X'(0, \infty)$, then we have
\begin{equation}\label{prop:operator_induced_norms_equivalences_nu_sigma}
\nu_{m, X}(f) \approx \sigma_{m, X}(f) \quad \text{for every $f\in\Mpl(0, \infty)$}.
\end{equation}
\item\label{prop:operator_induced_norms_equivalences_sigma_lambda_item} If the operator $f\mapsto f^{**}$ is bounded on $X(0, \infty)$ and \eqref{prop:operator_induced_norms_nu_sigma_condition} is satisfied, then we have
\begin{align}
\sigma_{m, X}(f) &\approx \lambda_{m, X}(f) \quad \text{for every $f\in\Mpl(0, \infty)$} \label{prop:operator_induced_norms_equivalences_sigma_lambda} \\
\intertext{and}
\nu_{m, X}(f) &\approx \mu_{m, X}(f) \quad \text{for every $f\in\Mpl(0, \infty)$}. \label{prop:operator_induced_norms_equivalences_nu_mu}
\end{align}
\end{enumerate}
\end{proposition}
\begin{proof}
We start with \ref{prop:operator_induced_norms_equivalences_nu_sigma_item}. First, it is not hard to see that, if \eqref{thm:optimal_target_general_theorem_condition} is not satisfied, then $\nu_{m, X}(f) = \sigma_{m, X}(f) = \infty$ for every $f\in\Mpl(0, \infty)\setminus\{0\}$. Therefore, we may assume that \eqref{thm:optimal_target_general_theorem_condition} is satisfied. Recall that the operator $f \mapsto f^{**}$ is bounded on $X'(0, \infty)$ if and only if the operator $P$ is (see \eqref{prel:ri:P_bounded_iff_double_star}); moreover, their operator norms are the same. If $m = 1$, there is nothing to prove, for $\nu_{1, X} = \sigma_{1, X}$. Assume now that $m = 2$. Using \eqref{prel:ri:normX''down}, \eqref{prel:ri:Rk_Hk_self_adjoint} and the fact that the operator $P$ is bounded on $X'(0, \infty)$, we obtain
\begin{align*}
\nu_{2, X}(f) &= \|(H_2\circ P)f^*\|_{X(0, \infty)} = \|(Q\circ R_2)f^*\|_{X(0, \infty)}\\
&= \sup_{\substack{g\in\Mpl(0, \infty)\\ \|g\|_{X'(0, \infty)}\leq1}} \int_0^\infty g^*(t) (Q\circ R_2)f^*(t) \d{t} \\
&= \sup_{\substack{g\in\Mpl(0, \infty)\\ \|g\|_{X'(0, \infty)}\leq1}} \int_0^\infty Pg^*(t) R_2f^*(t) \d{t} \\
&\approx \sup_{\substack{g\in\Mpl(0, \infty)\\ \|g\|_{X'(0, \infty)}\leq1}} \int_0^\infty g^*(t) R_2f^*(t) \d{t}
\end{align*}
for every $f\in\Mpl(0, \infty)$, in which the multiplicative constants depend only on $\|P\|_{X'(0, \infty)}$. Here we also used the fact that the function $Pg^* = g^{**}$ is nonincreasing together with \eqref{prel:ri:twostarsdominateonestar}. Owing to \cite[Theorem~3.10]{Pe:20} (see also \cite[Theorem~9.5]{CPS:15}) combined with \eqref{prel:ri:normX''}, we have
\begin{align*}
\sup_{\substack{g\in\Mpl(0, \infty)\\ \|g\|_{X'(0, \infty)}\leq1}} \int_0^\infty g^*(t) R_2f^*(t) \d{t} &\approx \sup_{\substack{g\in\Mpl(0, \infty)\\ \|g\|_{X'(0, \infty)}\leq1}} \int_0^\infty g(t) R_2f^*(t) \d{t} \\
&= \|R_2f^*(t)\|_{X(0, \infty)} = \sigma_{2, X}(f)
\end{align*}
for every $f\in\Mpl(0, \infty)$, in which the multiplicative constants are absolute. Hence \eqref{prop:operator_induced_norms_equivalences_nu_sigma} is true for $m = 2$. Now, let $m$ be either $1$ or $2$ (and so $\beta = m$, where $\beta$ is defined by \eqref{thm:optimal_target_general_theorem_beta}), and let $Z(0, \infty)$ be the rearrangement\hyp{}invariant function space whose function norm is given by $\nu_{m, X}$. We claim that, if the operator $P$ is bounded on $X'(0, \infty)$, then it is also bounded on $Z'(0, \infty)$. In view of \eqref{prel:ri:P_bounded_iff_Q} and \eqref{prel:ri:Q_bounded_iff_Qfstar}, we have the fact that the operator $Q$ is bounded on $X(0, \infty)$, and it is sufficient to show that the operator $g\mapsto Qg^*$ is bounded on $Z(0, \infty)$. Note that
\begin{equation}\label{prop:operator_induced_norms_Rbeta_circ_Q}
(R_\beta\circ Q)g^*(t) = R_\beta g^*(t) + t\phi_\beta(t)Qg^*(t)
\end{equation}
for every $t\in(0, \infty)$ and every $g\in\Mpl(0, \infty)$, thanks to the Fubini theorem. Furthermore, note that the auxiliary operator $T$ defined as
\begin{equation*}
Th(t) = \chi_{(0,1)}(t)t\phi_\beta(t) \int_t^1 \frac{h(s)}{s^2 \phi_\beta(s)} \d{s},\ h\in\M(0, \infty),\ t\in(0, \infty),
\end{equation*}
is bounded on both $L^1(0, \infty)$ and $L^\infty(0, \infty)$. Indeed, we have
\begin{align*}
\|Th\|_{L^1(0, \infty)} &\leq \int_0^1 |h(s)|s^{-1 - \frac{\beta}{n}} \int_0^s t^{\frac{\beta}{n}} \d{t} \d{s} \leq \frac{n}{n + \beta}\|h\|_{L^1(0, \infty)} \\
\intertext{and}
\|Th\|_{L^\infty(0, \infty)} &\leq \|h\|_{L^\infty(0, \infty)} \sup_{t\in (0, 1)} t^{\frac{\beta}{n}} \int_t^1 s^{-1 - \frac{\beta}{n}} \d{s} = \frac{n}{\beta} \|h\|_{L^\infty(0, \infty)}.
\end{align*}
Owing to \cite[Chapter~3, Theorem~2.2]{BS}, the operator $T$ is bounded on every rearrangement\hyp{}invariant function space over $(0, \infty)$, and its norm depends only on $n$ and $\beta$. In particular, $T$ is bounded on $X(0, \infty)$. Using \eqref{prop:operator_induced_norms_Rbeta_circ_Q}, we have
\begin{align}
\|Qg^*\|_{Z(0, \infty)} &\approx \|R_\beta(Qg^*)\|_{X(0, \infty)} \leq \|g\|_{Z(0, \infty)} + \|t\phi_\beta(t)Qg^*(t)\|_{X(0, \infty)} \notag\\
&\leq \|g\|_{Z(0, \infty)} + \|t\phi_\beta(t)Q(g^*\chi_{(0,1)})(t)\|_{X(0, \infty)} \notag\\
&\quad + \|t\phi_\beta(t)Q(g^*\chi_{(1, \infty)})(t)\|_{X(0, \infty)} \label{prop:operator_induced_norms_Q_bounded_on_Z_aux1}
\end{align}
for every $g\in\Mpl(0, \infty)$. Thanks to the boundedness of the operators $Q$ and $T$ on $X(0, \infty)$ and the monotonicity of the nonincreasing rearrangement, we have
\begin{align}
&\|t\phi_\beta(t)Q(g^*\chi_{(0,1)})(t)\|_{X(0, \infty)} + \|t\phi_\beta(t)Q(g^*\chi_{(1, \infty)})(t)\|_{X(0, \infty)} \notag\\
&= \Big\| \chi_{(0,1)}(t)t^\frac{\beta}{n} \int_t^1 \frac{g^*(s)s\phi_\beta(s)}{s^2\phi_\beta(s)} \d{s} \Big\|_{X(0, \infty)} + \|Q(g^*\chi_{(1, \infty)})\|_{X(0, \infty)} \notag\\
&\lesssim \|g^*(t)t\phi_\beta(t)\|_{X(0, \infty)} + \|g^*\chi_{(1, \infty)}\|_{X(0, \infty)} \notag\\
&\leq 2\|g^*(t)t\phi_\beta(t)\|_{X(0, \infty)} \leq 2\|R_\beta g^*\|_{X(0, \infty)} \approx \|g\|_{Z(0, \infty)} \label{prop:operator_induced_norms_Q_bounded_on_Z_aux2}
\end{align}
for every $g\in\Mpl(0, \infty)$. Hence the operator $g\mapsto Qg^*$ is bounded on $Z(0, \infty)$ thanks to \eqref{prop:operator_induced_norms_Q_bounded_on_Z_aux1} and \eqref{prop:operator_induced_norms_Q_bounded_on_Z_aux2}.

Finally, the rest follows by induction. Assume that the statement is valid for both $m\in\N$ and $m + 1$. Let $k = \lceil m/2 - 1 \rceil$. Let $Z_1$ and $Z_2$ be the rearrangement\hyp{}invariant function spaces defined by $\nu_{1, X}$ and $\nu_{2, X}$, respectively. Thanks to the previous step, we know that $\nu_{2, X} \approx \sigma_{2, X}$ and that the operator $f\mapsto f^{**}$ is bounded on $Z_i'(0, \infty)$, $i = 1,2$. Hence, owing to the inductive hypothesis, we have $\nu_{m + 1, Z_1} \approx \sigma_{m + 1, Z_1}$ and $\nu_{m, Z_2} \approx \sigma_{m, Z_2}$. Furthermore, it follows from \cite[Theorem~3.10]{Pe:20} (see also \cite[Theorem~9.5]{CPS:15}) that
\begin{equation*}
\|R_\beta((R_2^{k+1}f^*)^*)\|_{X(0, \infty)} \approx \|R_\beta(R_2^{k+1}f^*)\|_{X(0, \infty)} \quad \text{for every $f\in\Mpl(0, \infty)$},
\end{equation*}
where the multiplicative constants are absolute. Note that
\begin{equation*}
\|R_\beta(R_2^{k+1}f^*)\|_{X(0, \infty)} = \sigma_{m + 2, X}(f) \quad \text{for every $f\in\Mpl(0, \infty)$}.
\end{equation*}
The rest is simple now. If $m$ is odd, we have
\begin{align*}
\nu_{m + 2, X}(f) &= \|(H_2\circ P)^{k+1} f^*\|_{Z_1} = \nu_{m+1, Z_1}(f) \approx \sigma_{m+1, Z_1}(f) \\
&= \|R_1((R_2^{k+1}f^*)^*)\|_{X(0, \infty)} \approx \sigma_{m + 2, X}(f)
\end{align*}
for every $f\in\Mpl(0, \infty)$. If $m$ is even, we have
\begin{align*}
\nu_{m + 2, X}(f) &= \|(H_2\circ P)^{k+1} f^*\|_{Z_2} = \nu_{m, Z_2}(f) \approx \sigma_{m, Z_2}(f) \\
&\approx \|R_2((R_2^{k+1}f^*)^*)\|_{X(0, \infty)} \approx \sigma_{m + 2, X}(f)
\end{align*}
for every $f\in\Mpl(0, \infty)$, which finishes the proof of \ref{prop:operator_induced_norms_equivalences_nu_sigma_item}.

We now turn our attention to \ref{prop:operator_induced_norms_equivalences_sigma_lambda_item}. We start with \eqref{prop:operator_induced_norms_equivalences_sigma_lambda}. Note that there is nothing to prove when $m\in\{1, 2\}$, for $\sigma_{m, X} = \lambda_{m, X}$ in that case. Let $m\geq 3$. In view of \eqref{prop:operator_induced_norms_lambda_lesssim_sigma}, we only need to prove that
\begin{equation}\label{prop:operator_induced_norms_equivalences_sigma_lambda_sigma_smaller_lambda}
\sigma_{m, X}(f) \lesssim \lambda_{m, X}(f) \quad \text{for every $f\in\Mpl(0, \infty)$}.
\end{equation}
Note that the validity of \eqref{prop:operator_induced_norms_nu_sigma_condition} implies the validity of \eqref{prop:operator_induced_norms_lambda_condition}. Hence, the functional $\lambda_{m, X}$ is a rearrangement\hyp{}invariant function norm thanks to \cref{prop:operator_induced_norms}. Using \eqref{prop:iterated_operators_closed_form:E-RR} and the fact that $\lambda_{m, X}$ is a rearrangement\hyp{}invariant function norm, we have
\begin{align}
\|(R_\beta\circ R_2^k)(f^*\chi_{(0, 1)})\|_{X(0, \infty)} &\approx \Big\| \chi_{(0,1)}(t) \phi_m(t) \int_0^t f^*(s) \d{s} \Big\|_{X(0, \infty)} \nonumber\\
&\quad+ \int_0^1 f^*(s) \d{s} \Big\| \frac{\chi_{(1,\infty)}(t)}{t} \log(et)^k \Big\|_{X(0, \infty)} \nonumber\\
&\leq \lambda_{m, X}(f) + C_{m,X}\Big\| \frac{\chi_{(1,\infty)}(t)}{t} \log(et)^k \Big\|_{X(0, \infty)} \lambda_{m, X}(f) \label{prop:operator_induced_norms_equivalences_sigma_lambda_near_infty_aux1}
\end{align}
for every $f\in\Mpl(0, \infty)$, where $C_{m,X}\in(0, \infty)$ is the constant from the property (P5) of $\lambda_{m, X}$. Furthermore, since we assume that the operator $P$ is bounded on $X(0, \infty)$ and $R_\alpha g \leq Pg$ for every $g\in\Mpl(0, \infty)$ and $\alpha\in[0, n)$, we also have
\begin{align}
\|(R_\beta\circ R_2^k)(f^*\chi_{(1, \infty)})\|_{X(0, \infty)} &\leq \|P^{k + 1}(f^*\chi_{(1, \infty)})\|_{X(0, \infty)} \notag\\
&\leq \|P\|_{X(0, \infty)}^{k+1} \|f^*\chi_{(1, \infty)}\|_{X(0, \infty)} \notag\\
&\leq \|P\|_{X(0, \infty)}^{k+1} \|f^{**}\chi_{(1, \infty)}\|_{X(0, \infty)} \notag\\
&\leq \|P\|_{X(0, \infty)}^{k+1} \lambda_{m, X}(f) \label{prop:operator_induced_norms_equivalences_sigma_lambda_near_infty_aux2}
\end{align}
for every $f\in\Mpl(0, \infty)$. Hence, combining \eqref{prop:operator_induced_norms_equivalences_sigma_lambda_near_infty_aux1} and \eqref{prop:operator_induced_norms_equivalences_sigma_lambda_near_infty_aux2}, we obtain \eqref{prop:operator_induced_norms_equivalences_sigma_lambda_sigma_smaller_lambda}.

Last, we shall prove the equivalence \eqref{prop:operator_induced_norms_equivalences_nu_mu}. First, assume that $m$ is even. If $m = 2$, there is nothing to prove, for $\nu_{2, X} = \mu_{2, X}$. Let $m\geq 4$ be even. Owing to \cref{prop:operator_induced_norms}, both functionals $\nu_{m, X}$ and $\mu_{m, X}$ are rearrangement\hyp{}invariant function norms. Thanks to \eqref{prop:properties_of_iterated_operators_(Hj_P)k}, we have
\begin{equation*}
\nu_{m, X}(f) \approx \mu_{m, X}(f) + \|R_2^{k + 1} f^*\|_{X(0,\infty)} \quad \text{for every $f\in\Mpl(0, \infty)$}.
\end{equation*}
Hence it is sufficient to prove that
\begin{equation}\label{prop:operator_induced_norms_equivalences_nu_mu_missing_ineq}
\|R_2^{k + 1} f^*\|_{X(0,\infty)} \lesssim \mu_{m, X}(f) \quad \text{for every $f\in\Mpl(0, \infty)$}.
\end{equation}
Note that
\begin{align}
\|R_2^{k + 1} f^*\|_{X(0,\infty)} &\approx \|\chi_{(0,1)}(t) \phi_m(t) t f^{**}(t)\|_{X(0, \infty)} \nonumber\\
&\quad+ \int_0^1 f^*(s) \d{s} \Big\| \frac{\chi_{(1,\infty)}(t)}{t} \log(et)^k \Big\|_{X(0, \infty)} \nonumber\\
&\quad+ \Big\| \chi_{(1,\infty)}(t) \frac1{t} \int_1^t f^*(s) \log\Big( \frac{t}{s} \Big)^k \d{s} \Big\|_{X(0, \infty)} \label{prop:operator_induced_norms_equivalences_nu_mu_Rk_split}
\end{align}
for every $f\in\Mpl(0, \infty)$ thanks to \eqref{prop:iterated_operators_closed_form:E-RR}. First, using \eqref{prel:ri:dilation} and \eqref{prop:iterated_operators_closed_form:E-HH}, we have
\begin{align}
\|\chi_{(0,1)}(t) \phi_m(t) t f^{**}(t)\|_{X(0, \infty)} &\approx \Big\| \chi_{(0,1)}(t) \phi_m(t) \int_{t/2}^t K_k\Big( \frac{t}{2}, s \Big) \d{s} f^{**}(t) \Big\|_{X(0, \infty)} \nonumber\\
&\leq  \Big\| \int_{t/2}^t f^{**}(s)\phi_m(s)K_k\Big( \frac{t}{2}, s \Big) \d{s} \Big\|_{X(0, \infty)} \nonumber\\
&\leq  \Big\| \int_{t/2}^\infty f^{**}(s)\phi_m(s)K_k\Big( \frac{t}{2}, s \Big) \d{s} \Big\|_{X(0, \infty)} \nonumber\\
&\approx \mu_{m, X}(f) \label{prop:operator_induced_norms_equivalences_nu_mu_missing_ineq:aux1}
\end{align}
for every $f\in\Mpl(0, \infty)$. Second, note that
\begin{equation}\label{prop:operator_induced_norms_equivalences_nu_mu_missing_ineq:aux2}
\int_0^1 f^*(s) \d{s} \lesssim \mu_{m, X}(f) \quad \text{for every $f\in\Mpl(0, \infty)$}
\end{equation}
thanks to the property (P5) of $\mu_{m, X}$. Finally, thanks to the Fubini theorem, \eqref{prop:iterated_operators_closed_form:E5}, \eqref{E:kernel_K_monotonicity_second_var}, \eqref{prel:ri:dilation}, and \eqref{prop:iterated_operators_closed_form:E-HH}, we have
\begin{align}
&\Big\| \chi_{(1,\infty)}(t) \frac1{t} \int_1^t f^*(s) \log\Big( \frac{t}{s} \Big)^k \d{s} \Big\|_{X(0, \infty)} \nonumber\\
&= \|\chi_{(1,\infty)} P^{k+1}(f^*\chi_{(1, \infty)})\|_{X(0, \infty)} \nonumber\\
&\leq \|P\|_{X(0, \infty)}^k \|f^*\chi_{(1, \infty)}\|_{X(0, \infty)} \nonumber\\
&\approx \|P\|_{X(0, \infty)}^k \Big\| f^*(t)\chi_{(1, \infty)}(t) \int_{t/2}^t \phi_m(s) K_k\Big( \frac{t}{4}, \frac{t}{2} \Big) \d{s} \Big\|_{X(0, \infty)} \nonumber\\
&\leq \|P\|_{X(0, \infty)}^k \Big\| \int_{t/2}^t f^*(s) \phi_m(s) K_k\Big( \frac{t}{4}, s \Big) \d{s} \Big\|_{X(0, \infty)} \nonumber\\
&\leq \|P\|_{X(0, \infty)}^k \Big\| \int_{t/4}^\infty f^{**}(s) \phi_m(s) K_k\Big( \frac{t}{4}, s \Big) \d{s} \Big\|_{X(0, \infty)} \nonumber\\
&\approx \|P\|_{X(0, \infty)}^k \mu_{m, X}(f) \label{prop:operator_induced_norms_equivalences_nu_mu_missing_ineq:aux3}
\end{align}
for every $f\in\Mpl(0, \infty)$. Hence, combining \eqref{prop:operator_induced_norms_equivalences_nu_mu_Rk_split} with \eqref{prop:operator_induced_norms_equivalences_nu_mu_missing_ineq:aux1}, \eqref{prop:operator_induced_norms_equivalences_nu_mu_missing_ineq:aux2}, and \eqref{prop:operator_induced_norms_equivalences_nu_mu_missing_ineq:aux3}, we obtain \eqref{prop:operator_induced_norms_equivalences_nu_mu_missing_ineq}. Therefore, all that remains for us is to prove \eqref{prop:operator_induced_norms_equivalences_nu_mu} for $m$ odd. Let $Z(0, \infty)$ be the rearrangement\hyp{}invariant function space defined by $\nu_{1, X}$, which is a rearrangement\hyp{}invariant function norm thanks to \cref{prop:operator_induced_norms}. We claim that the operator $f\mapsto f^{**}$ is bounded on $Z(0, \infty)$. Owing to \eqref{prel:ri:P_bounded_iff_double_star}, that is the same as showing that the operator $P$ is bounded on $Z(0, \infty)$. By \cite[Chapter~3, Theorem~5.15]{BS}, the operator $P$ is bounded on a rearrangement\hyp{}invariant function space $W(0, \infty)$ if and only if
\begin{equation*}
\inf_{a\in(1, \infty)} \frac{\log \|D_a\|_{W(0, \infty)}}{\log a} < 1.
\end{equation*}
Moreover, the infimum is equal to the limit as $a$ goes to $\infty$, and the quantity is called the upper Boyd index of $W(0, \infty)$. In view of that and the boundedness of the operator $P$ on $X(0, \infty)$, it is sufficient to show that
\begin{equation}\label{prop:operator_induced_norms_equivalences_nu_mu_norm_of_dilation_nu_1X_leq_X}
\|D_a\|_{Z(0, \infty)} \leq \|D_a\|_{X(0, \infty)}.
\end{equation}
To that end, note that
\begin{equation*}
(D_af)^*(t) = f^*\Big( \frac{t}{a} \Big) \quad \text{for every $t\in(0, \infty)$ and $a\in(0, \infty)$}.
\end{equation*}
Using that, we have, for every $a\in(0, \infty)$ and $f\in\Mpl(0, \infty)$,
\begin{align*}
\|D_a f\|_{Z(0, \infty)} &= \Big\| \phi_1(t) \int_0^t f^*\Big( \frac{s}{a} \Big) \d{s} \Big\|_{X(0, \infty)} = a \Big\| \phi_1(t) \int_0^{\frac{t}{a}} f^*(s) \d{s} \Big\|_{X(0, \infty)} \\
&= a\min\{a^{-1 + \frac{1}{n}}, a^{-1}\} \|D_a(R_1f^*)\|_{X(0, \infty)} \\
&\leq \|D_a\|_{X(0, \infty)} \|f\|_{Z(0, \infty)}.
\end{align*}
Hence \eqref{prop:operator_induced_norms_equivalences_nu_mu_norm_of_dilation_nu_1X_leq_X} is valid. Now, note that the fact that the operator $f\mapsto f^{**}$ is bounded on $Z(0, \infty)$ immediately implies the validity of \eqref{prop:operator_induced_norms_equivalences_nu_mu} for $m = 1$. Finally, let $m\geq 3$ be odd. Note that $\lceil (m - 1)/2 - 1 \rceil = k - 1$, and that
\begin{equation*}
\nu_{m, X} = \nu_{m - 1, Z} \qquad \text{and} \qquad \mu_{m, X} = \mu_{m - 1, Z}.
\end{equation*}
Therefore, the validity of \eqref{prop:operator_induced_norms_equivalences_nu_mu} for $m\geq 3$ odd will follow once we observe that
\begin{equation*}
\frac{(1 + \log t)^{k-1}}{t}\chi_{(1, \infty)}(t) \in Z(0, \infty).
\end{equation*}
However, that follows immediately from the validity of \eqref{prop:operator_induced_norms_nu_sigma_condition}. Indeed, using the fact that the function $(0, \infty) \ni t \mapsto \chi_{(0, 1)}(t) + (1 + \log t)^{k-1}\chi_{[1, \infty)}(t)/t$ is equivalent to a nonincreasing function, we have
\begin{align*}
\Big\| \frac{(1 + \log t)^{k-1}}{t}\chi_{(1, \infty)}(t) \Big\|_{Z(0, 1)} &\lesssim \| t^{\frac1{n}}\chi_{(0, 1)}(t) \|_{X(0, \infty)} \\
&\quad+ \Big\| \frac{(1 + \log t)^k}{t}\chi_{(1, \infty)}(t) \Big\|_{X(0, \infty)} < \infty. \qedhere
\end{align*}
\end{proof}

\begin{remark}
Since $\nu_{1, X} = \sigma_{1, X} = \lambda_{1, X}$, \cref{prop:operator_induced_norms_equivalences}\ref{prop:operator_induced_norms_equivalences_nu_sigma_item} is interesting only when $m\geq2$. Furthermore, the equivalence \eqref{prop:operator_induced_norms_equivalences_sigma_lambda} is interesting only when $m\geq 3$, for $\sigma_{2, X} = \lambda_{2, X}$. Finally, the equivalence \eqref{prop:operator_induced_norms_equivalences_nu_mu} is interesting only when $m\neq 2$, for $\nu_{2, X} = \mu_{2, X}$.
\end{remark}

\section{Proofs of Main Results}\label{sec:proofs_main}

\begin{proof}[Proof of \cref{thm:reduction_principle}]
We begin with a few definitions. We define the operators $\widetilde{H}_1$ and $\widetilde{H}_2$ as
\begin{align*}
\widetilde{H}_1f(t) &= \int_t^\infty f(s) \sinh(\varrho(s))^{1-n} \d{s}, \\
\widetilde{H}_2f(t) &= \int_t^\infty f(s)s\sinh(\varrho(s))^{2-2n} \d{s},
\end{align*}
for $f\in\Mpl(0,\infty)$ and  $t\in[0,\infty)$. The function $\varrho$ is the inverse function to the function $V$ defined by \eqref{prel:volume_of_hyperbolic_ball}. We also set
\begin{equation*}
\widetilde{T}_m = \begin{cases}
	(\widetilde{H}_2 \circ P)^k \circ \widetilde{H}_1 \quad &\text{if $m$ is odd}, \\
	(\widetilde{H}_2 \circ P)^{k + 1} \quad &\text{if $m$ is even},	
\end{cases}
\end{equation*}
where $k = \lceil m/2 - 1 \rceil$.
Using L'Hospital's rule, we immediately see that
\begin{align*}
\lim_{t\to0^+} \frac{\sinh(\varrho(t))^{1 - n}}{t^{-1 + \frac1{n}}} &= \left(\lim_{t\to0^+} \frac{\sinh(t)^n}{V(t)}\right)^\frac{1-n}{n} = \omega_n^\frac{n-1}{n}\\
\intertext{and}
\lim_{t\to\infty} \frac{\sinh(\varrho(t))^{1 - n}}{t^{-1}} &= \left(\lim_{t\to\infty} \frac{\sinh(t)^{n-1}}{V(t)}\right)^{-1} = \frac{n\omega_n}{n-1};
\end{align*}
hence
\begin{align}
\sinh(\varrho(t))^{1 - n} &\approx \phi_1(t), \label{thm:reduction_principle_asymp_behavior_1}\\
t\sinh(\varrho(t))^{2 - 2n} &\approx \phi_2(t), \label{thm:reduction_principle_asymp_behavior_2}
\end{align}
for every $t\in(0, \infty)$, in which the multiplicative constants depend only on $n$. It follows that
\begin{equation}\label{thm:reduction_principle_equivalence_Tm}
\widetilde{T}_mf(t) \approx \Tm f(t) \quad \text{for every $f\in\Mpl(0, \infty)$, $t\in(0, \infty)$ and $m\in\N$},
\end{equation}
in which the multiplicative constants depend only on $n$ and $m$. Recall that the operator $T_m$ is defined by \eqref{operator_Tm_def}.

\underline{\ref{thm:reduction_principle_Sobolev_item} implies \ref{thm:reduction_principle_Hardy_item}:} We start with some observations. First, if $f\in\Mpl(0, \infty)\cap \mathcal C_b(0, \infty) \cap L^2(0,\infty)$, then $\widetilde{H}_1 f \in \mathcal C^1(0, \infty)\cap L^2(0, \infty)\cap L^\infty(0, \infty)$ and $\lim_{t\to \infty} \widetilde{H}_1 f(t) = 0$. The fact that $\widetilde{H}_1 f \in L^\infty(0, \infty)$ follows from H\"{o}lder's inequality combined with \eqref{thm:reduction_principle_asymp_behavior_1}. The fact that $\widetilde{H}_1 f \in L^2(0, \infty)$ follows from weighted Hardy's inequality \cite[Theorem~2]{M:72}, inasmuch as
\begin{equation*}
\sup_{r > 0} r^{\frac1{2}}\Big(\int_r^\infty\phi_1^2\Big)^\frac1{2} \approx \sup_{r > 0} r^{\frac1{2}}\phi_2(r)^\frac1{2} < \infty.
\end{equation*}
Second, if $f$ is as above and has support inside $(0, \infty)$, then the function $v_f$ defined as
\begin{equation*}
v_f(x) = \widetilde{H}_1 f(V(\distH(x))),\ x\in\Hn,
\end{equation*}
belongs to $\mathcal C^1(\Hn)\cap L^2(\Hn)\cap L^\infty(\Hn)$ and $\lim_{\distH(x) \to \infty} v_f(x) = 0$. Moreover,
\begin{align*}
\absH{\gradH v_f(x)} &= \absH{f(V(\distH(x))) \sinh(\varrho(V(\distH(x))))^{1-n} n\omega_n \sinh(\distH(x))^{n - 1} \gradH \distH(x)} \\
	&= n\omega_n f(V(\distH(x)))
\end{align*}
for every $x\in\Hn$, thanks to the fact that $\absH{\gradH \distH(x)} = 1$ for every $x\in\Hn\setminus\{0\}$. Therefore, $v_f\in V_0^1X(\Hn)$ for every such function $f$ that belongs to $X(0, \infty)$. Next,  if $f\in\Mpl(0, \infty)\cap \mathcal C_b(0, \infty) \cap L^2(0,\infty)$, then $(\widetilde{H}_2\circ P) f \in \mathcal C^2(0, \infty)\cap L^2(0, \infty)\cap L^\infty(0, \infty)$ and $\lim_{t\to \infty} \widetilde{H}_2 f(t) = 0$. The fact that $(\widetilde{H}_2 \circ P) f \in L^\infty(0, \infty)$ follows from H\"{o}lder's inequality combined with \eqref{thm:reduction_principle_asymp_behavior_2}. As for the fact that $(\widetilde{H}_2 \circ P) f \in L^2(0, \infty)$, we have
\begin{equation*}
\|(\widetilde{H}_2 \circ P) f\|_{L^2(0, \infty)}\lesssim \|Pf\|_{L^2(0, \infty)}\lesssim \|f\|_{L^2(0, \infty)}
\end{equation*}
by weighted Hardy's inequalities \cite{M:72}. Furthermore, it is easy to see that the function $w_f$ defined as
\begin{equation*}
w_f(x) = (\widetilde{H}_2\circ P) f(V(\distH(x))),\ x\in\Hn,
\end{equation*}
belongs to $\mathcal C^1(\Hn)\cap \mathcal C^2(\Hn\setminus\{0\})\cap L^2(\Hn)\cap L^\infty(\Hn)$ and $\lim_{\distH(x) \to \infty} w_f(x) = 0$. Now, considering $w_f$ as a function of $r = |x|$, we have
\begin{align}
\lapH w_f(x) &= \left(\frac{1 - r^2}{2}\right)^2 \left(\frac{\mathrm d^2\tilde{w}_f}{\mathrm d r^2}(r) + \frac{n - 1}{r}\frac{\mathrm d\tilde{w}_f}{\mathrm d r}(r) \right) + (n-2)\frac{1 - r^2}{2}r\frac{\mathrm d\tilde{w}_f}{\mathrm d r}(r) \notag\\
&= \left(\frac{1 - r^2}{2}\right)^2 \frac{\mathrm d^2\tilde{w}_f}{\mathrm d r^2}(r) + \left( \frac{1 - r^2}{2} \frac{n - 1}{r} + (n - 2)r \right) \frac{1 - r^2}{2} \frac{\mathrm d\tilde{w}_f}{\mathrm d r}(r) \label{thm:reduction_principle_LapH_radial}
\end{align}
for every $x\in\Hn\setminus\{0\}$, where
\begin{equation*}
\tilde{w}_f(r) = (\widetilde{H}_2\circ P) f(V(2\arctanh r)),\ r\in(0, 1).
\end{equation*}
Computing the derivatives of $\tilde{w}_f$, we obtain
\begin{equation}\label{thm:reduction_principle_first_derivative_of_T2}
\frac{1 - r^2}{2}\frac{\mathrm d\tilde{w}_f}{\mathrm d r}(r) = -n\omega_n \left(\int_0^{V(2\arctanh r)} f\right) \sinh(2\arctanh r)^{1 - n},
\end{equation}
in which we used the fact that $\varrho(V((2\arctanh r))) = 2\arctanh r$. We also obtain
\begin{align}
\left(\frac{1 - r^2}{2}\right)^2\frac{\mathrm d^2\tilde{w}_f}{\mathrm d r^2}(r) &= -(n\omega_n)^2 f(V(2\arctanh r)) \notag\\
&\quad- \left(\frac{n - 1}{\tanh(2\arctanh r)} - r\right)\frac{1 - r^2}{2}\frac{\mathrm d\tilde{w}_f}{\mathrm d r}(r) \notag\\
&= -(n\omega_n)^2 f(V(2\arctanh r)) \notag\\
&\quad - \left((n - 1)\frac{1 + r^2}{2r} - r\right) \frac{1 - r^2}{2} \frac{\mathrm d\tilde{w}_f}{\mathrm d r}(r), \label{thm:reduction_principle_second_derivative_of_T2}
\end{align}
in which we used \eqref{thm:reduction_principle_first_derivative_of_T2} and the identity $\tanh(2\arctanh r) = 2r/(1 + r^2)$. Plugging \eqref{thm:reduction_principle_first_derivative_of_T2} and \eqref{thm:reduction_principle_second_derivative_of_T2} into \eqref{thm:reduction_principle_LapH_radial} and observing that
\begin{equation*}
\frac{1 - r^2}{2} \frac{n - 1}{r} + (n - 2)r - (n - 1)\frac{1 + r^2}{2r} + r = 0,
\end{equation*}
we arrive at
\begin{equation*}
\lapH w_f(x) = -(n\omega_n)^2 f(V(2\arctanh |x|)) = -(n\omega_n)^2 f(V(\distH(x))).
\end{equation*}

Finally, putting all those observations together, we have proved the following fact: If $f\in\Mpl(0, \infty)\cap \mathcal C_b(0, \infty) \cap L^2(0,\infty) \cap X(0, \infty)$ has support inside $(0, \infty)$, then the function $u_f$ defined as
\begin{equation}\label{thm:reduction_principle_test_function}
u_f(x) = \widetilde{T}_mf(V(\distH(x))),\ x\in\Hn,
\end{equation}
belongs to $V_0^m X(\Hn)$, and we have
\begin{equation}\label{thm:reduction_principle_test_function_norm_of_derivative}
\|\lapH[m]u_f\|_{X(\Hn)} = c_{n,m} \|f(V(\distH(x)))\|_{X(\Hn)} = c_{n,m} \|f\|_{X(0, \infty)},
\end{equation}
where $c_{n,m} = (n\omega_n)^{\lceil m/2\rceil}$. In the last equality, we used the fact that the mapping $\Hn \ni x \mapsto V(\distH(x)) \in [0, \infty)$ is measure preserving (in the sense of \cite[Chapter~2, Definition~7.1]{BS}).

We are now in a position to prove the desired implication. First, we claim that it is sufficient to prove \eqref{thm:reduction_principle_Hardy_ineq} only for nonincreasing functions $f\in\Mpl(0,\infty)$; in other words, it is sufficient to prove
\begin{equation}\label{thm:reduction_principle_Hardy_ineq_restricted}
\|\Tm f^*\|_{Y(0, \infty)} \leq \widetilde{C}_2 \|f\|_{X(0, \infty)} \quad \text{for every $f\in\Mpl(0, \infty)$}.
\end{equation}
If $m$ is even, then \eqref{thm:reduction_principle_Hardy_ineq} reads as
\begin{equation}\label{thm:reduction_principle_Hardy_ineq_m_even}
\|(H_2\circ P)^{k + 1} f\|_{Y(0, \infty)} \leq C_2 \|f\|_{X(0, \infty)} \quad \text{for every $f\in\Mpl(0, \infty)$}.
\end{equation}
The fact that \eqref{thm:reduction_principle_Hardy_ineq_restricted} implies \eqref{thm:reduction_principle_Hardy_ineq_m_even} (with the same constant $C_2 = \widetilde{C}_2$) is an immediate consequence of \eqref{prel:ri:HLg=chiE}. If $m$ is odd, then \eqref{thm:reduction_principle_Hardy_ineq} reads as
\begin{equation}\label{thm:reduction_principle_Hardy_ineq_m_odd}
\|((H_2\circ P)^k\circ H_1) f\|_{Y(0, \infty)} \leq C_2 \|f\|_{X(0, \infty)} \quad \text{for every $f\in\Mpl(0, \infty)$}.
\end{equation}
The situation is considerably more complicated now because the integration in $H_1 f$ is carried out on intervals whose left endpoints are not $0$; therefore, we cannot simply use \eqref{prel:ri:HLg=chiE} this time. Using \eqref{prel:ri:normX''down}, \eqref{prop:properties_of_iterated_operators_H2P_self_adjoint}, \eqref{prel:ri:Rk_Hk_self_adjoint} and \eqref{prel:ri:normX'}, we have
\begin{align}
&\sup_{\substack{f\in\Mpl(0, \infty)\\ \|f\|_{X(0, \infty)}\leq1}} \|((H_2\circ P)^k\circ H_1) f\|_{Y(0, \infty)} \notag\\
&\quad= \sup_{\substack{f, g\in\Mpl(0, \infty)\\ \|f\|_{X(0, \infty)}\leq1, \|g\|_{Y'(0, \infty)}\leq1}} \int_0^\infty g^*(t) ((H_2\circ P)^k\circ H_1) f(t) \d{t} \notag\\
&\quad\leq \Big( \frac{n}{n-2} \Big)^k \sup_{\substack{f, g\in\Mpl(0, \infty)\\ \|f\|_{X(0, \infty)}\leq1, \|g\|_{Y'(0, \infty)}\leq1}} \int_0^\infty (H_2\circ P)^k g^*(t) H_1 f(t) \d{t} \notag\\
&\quad= \Big( \frac{n}{n-2} \Big)^k \sup_{\substack{f, g\in\Mpl(0, \infty)\\ \|f\|_{X(0, \infty)}\leq1, \|g\|_{Y'(0, \infty)}\leq1}} \int_0^\infty (R_1 \circ (H_2\circ P)^k) g^*(t) f(t) \d{t} \notag\\
&\quad= \Big( \frac{n}{n-2} \Big)^k \sup_{\substack{g\in\Mpl(0, \infty)\\ \|g\|_{Y'(0, \infty)}\leq1}} \|(R_1\circ (H_2\circ P)^k) g^*\|_{X'(0, \infty)}. \label{thm:reduction_principle_Hardy_ineq_restricted_eq1}
\end{align}
Since the function $(H_2\circ P)^k g^*$ is nonincreasing (even for $k = 0$), we have
\begin{equation}\label{thm:reduction_principle_Hardy_ineq_restricted_eq2}
\|(R_1\circ (H_2\circ P)^k) g^*\|_{X'(0, \infty)} \leq 4 \sup_{\substack{f\in\Mpl(0, \infty)\\ \|f\|_{X(0, \infty)}\leq1}} \int_0^\infty (R_1 \circ (H_2\circ P)^k) g^*(t) f^*(t) \d{t}
\end{equation}
for every $g\in\Mpl(0, \infty)$ by \cite[Theorem~3.10]{Pe:20} (see also \cite[Theorem~9.5]{CPS:15} and \cref{rem:reduction_principle_Hardy_ineq_equiv_to_reduced}\ref{rem:reduction_principle_Hardy_ineq_equiv_to_reduced_item})\textemdash note that we could not simply use \eqref{prel:ri:normX''down}, because the function $(R_1 \circ (H_2\circ P)^k) g^*$ need not be nonincreasing. Plugging \eqref{thm:reduction_principle_Hardy_ineq_restricted_eq2} into \eqref{thm:reduction_principle_Hardy_ineq_restricted_eq1}, we obtain
\begin{align*}
&\sup_{\substack{f\in\Mpl(0, \infty)\\ \|f\|_{X(0, \infty)}\leq1}} \|((H_2\circ P)^k\circ H_1) f\|_{Y(0, \infty)} \\
&\quad\leq 4 \Big( \frac{n}{n-2} \Big)^k \sup_{\substack{f, g\in\Mpl(0, \infty)\\ \|f\|_{X(0, \infty)}\leq1, \|g\|_{Y'(0, \infty)}\leq1}} \int_0^\infty (R_1 \circ (H_2\circ P)^k) g^*(t) f^*(t) \d{t}.
\end{align*}
Now, going from the end to the beginning in \eqref{thm:reduction_principle_Hardy_ineq_restricted_eq1} with $f$ replaced by $f^*$, we arrive at
\begin{align*}
&\sup_{\substack{f\in\Mpl(0, \infty)\\ \|f\|_{X(0, \infty)}\leq1}} \|((H_2\circ P)^k\circ H_1) f\|_{Y(0, \infty)} \\
&\quad\leq 4\Big( \frac{n}{n-2} \Big)^{2k} \sup_{\substack{f\in\Mpl(0, \infty)\\ \|f\|_{X(0, \infty)}\leq1}} \|((H_2\circ P)^k\circ H_1) f^*\|_{Y(0, \infty)}.
\end{align*}
Therefore, if \eqref{thm:reduction_principle_Hardy_ineq_restricted} holds, so does \eqref{thm:reduction_principle_Hardy_ineq_m_odd} (with $C_2 = 4(n/(n-2))^{2k}\widetilde{C}_2$). Furthermore, it is clearly sufficient to prove \eqref{thm:reduction_principle_Hardy_ineq_restricted} for bounded nonincreasing functions $f\in\Mpl(0, \infty)\cap X(0, \infty)$, for $\{\min\{j, f^*\}\}_{j = 1}^\infty$ is a nondecreasing sequence converging pointwise to $f^*$. Moreover, we may also assume that $f^*\in\mathcal C(0, \infty)$ because there is a nondecreasing sequence of nonincreasing nonnegative continuous functions converging pointwise a.e.~to $f^*$ (e.g., \cite[Proposition~2.1]{GS:13}). Finally, for $j\in\N$, let $\eta_j\in\mathcal C(0, \infty)$ be a cutoff function such that $0\leq\eta_j\leq 1$, $\eta_j\equiv1$ in $[1/j, j]$ and $\eta_j\equiv 0$ outside $[1/(j+1), j+1]$. Note that, if $f^*\in\mathcal C_b(0, \infty)$, then the functions $\eta_j f^*$, $j\in\N$, belong to $\Mpl(0, \infty) \cap \mathcal C_b(0, \infty) \cap X(0, \infty) \cap L^2(0, \infty)$ and have support inside $(0, \infty)$; moreover, the sequence $\{\eta_j f^*\}_{j = 1}^\infty$ is nondecreasing and converges pointwise to $f^*$. Therefore, \eqref{thm:reduction_principle_Hardy_ineq_restricted} will follow as soon as we prove that
\begin{equation}\label{thm:reduction_principle_Hardy_ineq_restricted_to_regular}
\|\Tm h\|_{Y(0, \infty)} \leq \widetilde{C}_2 \|h\|_{X(0, \infty)}
\end{equation}
for every $h\in\Mpl(0, \infty) \cap X(0, \infty) \cap \mathcal C_b(0, \infty) \cap L^2(0, \infty)$ having support inside $(0, \infty)$. The rest is easy now. Assume that \eqref{thm:reduction_principle_Sob_ineq} is valid. Let $h\in\Mpl(0, \infty) \cap X(0, \infty) \cap \mathcal C_b(0, \infty) \cap L^2(0, \infty)$ having support inside $(0, \infty)$, and let $u_h$ be the function defined by \eqref{thm:reduction_principle_test_function} with $f$ replaced by $h$. We have $u_h\in V_0^m X(\Hn)$. Thanks to \eqref{thm:reduction_principle_equivalence_Tm}, \eqref{thm:reduction_principle_test_function_norm_of_derivative} and \eqref{thm:reduction_principle_Sob_ineq}, we arrive at
\begin{align*}
\|\Tm h\|_{Y(0, \infty)} &\approx \|\widetilde{T}_m h\|_{Y(0, \infty)} = \|u_h\|_{Y(\Hn)} \\
&\leq C_1 \|\gradH[m] u_h\|_{X(\Hn)} = (n\omega_n)^{\lceil \frac{m}{2} \rceil} C_1 \|h\|_{X(0, \infty)}.
\end{align*}
Hence, \eqref{thm:reduction_principle_Hardy_ineq_restricted_to_regular} is valid with a constant $\widetilde{C}_2$, depending only on $C_1$, $m$ and $n$.

\underline{\ref{thm:reduction_principle_Hardy_item} implies \ref{thm:reduction_principle_Sobolev_item}:} First, let $m$ be even. Recall that $k = \lceil m/2 - 1 \rceil$. Iterating \eqref{prel:potential_estimate} $m/2 = (k + 1)$ times and using \eqref{thm:reduction_principle_Hardy_ineq} together with \eqref{thm:reduction_principle_equivalence_Tm}, we obtain
\begin{align*}
\|u\|_{Y(\Hn)} &= \|u^*\|_{Y(0, \infty)} \leq (n\omega_n)^{-m} \|(\widetilde{H}_2\circ P)^{k + 1} ((\lapH[k+1] u)^*) \|_{Y(0, \infty)} \\
&\approx \|\Tm ((\lapH[k+1] u)^*)\|_{Y(0, \infty)} \leq C_2 \|(\lapH[k+1] u)^*\|_{X(0, \infty)} = \|\gradH[m] u\|_{X(\Hn)}
\end{align*}
for every $u\in V_0^m X(\Hn)$; that is, \eqref{thm:reduction_principle_Sob_ineq} is valid with a constant $C_1$, which depends only on $C_2$, $m$ and $n$. Now, let $m$ be odd, and let $u\in V_0^m X(\Hn)$. By iterating \eqref{prel:potential_estimate} $k$ times, we have
\begin{equation}\label{thm:reduction_principle_odd_case_potential_estimate}
\|u\|_{Y(\Hn)} \leq (n\omega_n)^{-2k} \|(\widetilde{H}_2\circ P)^k ((\lapH[k] u)^*) \|_{Y(0, \infty)}.
\end{equation}
Since $\lapH[k] u\in V_0^1 X(\Hn)$, we may use \cref{prop:PS_inequality} to obtain
\begin{align}
(\lapH[k] u)^*(t) &= \int_t^\infty \Big( -\frac{\mathrm d (\lapH[k] u)^*}{\mathrm ds}(s) \Big) \d{s} \notag\\
&= \widetilde{H}_1\Big( -\frac{\mathrm d (\lapH[k] u)^*}{\mathrm ds}(s) \sinh(\varrho(s))^{n - 1} \Big)(t) \label{thm:reduction_principle_rearr_laplace_as_integral}
\end{align}
for every $t\in(0, \infty)$, and
\begin{align}
\Big\| -\frac{\mathrm d (\lapH[k] u)^*}{\mathrm dt}(t)\sinh(\varrho(t))^{n - 1} \Big\|_{X(0, \infty)} &\leq (n\omega_n)^{-1} \|\gradH \lapH[k] u\|_{X(\Hn)} \notag\\
&= (n\omega_n)^{-1} \|\gradH[m] u\|_{X(\Hn)}. \label{thm:reduction_principle_rearr_laplace_PS_used}
\end{align}
Finally, combining \eqref{thm:reduction_principle_odd_case_potential_estimate} and \eqref{thm:reduction_principle_rearr_laplace_as_integral} with \eqref{thm:reduction_principle_equivalence_Tm} and \eqref{thm:reduction_principle_Hardy_ineq}, and using \eqref{thm:reduction_principle_rearr_laplace_PS_used} in the last step, we arrive at
\begin{align*}
\|u\|_{Y(\Hn)} &\leq (n\omega_n)^{-2k} \Big\| \widetilde{T}_m\Big( -\frac{\mathrm d (\lapH[k] u)^*}{\mathrm ds}(s)\sinh(\varrho(s))^{n - 1}\Big) \Big\|_{Y(0, \infty)} \\
&\lesssim (n\omega_n)^{-2k} C_2 \Big\| -\frac{\mathrm d (\lapH[k] u)^*}{\mathrm dt}(t)\sinh(\varrho(t))^{n - 1} \Big\|_{X(0, \infty)} \\
&\leq (n\omega_n)^{-(2k + 1)} C_2 \|\gradH[m] u\|_{X(\Hn)}.
\end{align*}
Hence \eqref{thm:reduction_principle_Sob_ineq} is valid with a constant $C_1$, which depends only on $C_2$, $m$ and $n$.

\underline{\ref{thm:reduction_principle_Hardy_item} implies \ref{thm:reduction_principle_Hardy_simplified_item}:} Let $\nu_{m, X'}$ and $\sigma_{m, X'}$ be the functionals from \cref{prop:operator_induced_norms}. Note that, thanks to \eqref{prop:properties_of_iterated_operators_H2P_self_adjoint} and \eqref{prel:ri:Rk_Hk_self_adjoint}, combined with \eqref{prel:ri:normX''down} and \eqref{prel:ri:normX''}, we have
\begin{align}
\sup_{\substack{f\in\Mpl(0, \infty)\\ \|f\|_{X(0, \infty)}\leq1}} \|T_m f\|_{Y(0, \infty)} &\approx \sup_{\substack{g\in\Mpl(0, \infty)\\ \|g\|_{Y'(0, \infty)}\leq1}} \nu_{m, X'}(g), \label{thm:reduction_principle_Hardy_ineq_dual}
\intertext{in which the multiplicative constants depend only on $m$ and $n$, and}
\sup_{\substack{f\in\Mpl(0, \infty)\\ \|f\|_{X(0, \infty)}\leq1}} \|S_m f\|_{Y(0, \infty)} &= \sup_{\substack{g\in\Mpl(0, \infty)\\ \|g\|_{Y'(0, \infty)}\leq1}} \sigma_{m, X'}(g). \label{thm:reduction_principle_Hardy_ineq_simplified_dual}
\end{align}
Owing to \eqref{prop:operator_induced_norms_sigma_lesssim_nu}, we have $\sigma_{m, X'}(g) \lesssim \nu_{m, X'}(g)$ for every $g\in\Mpl(0, \infty)$. Hence \ref{thm:reduction_principle_Hardy_item} implies \ref{thm:reduction_principle_Hardy_simplified_item}.

\uline{\ref{thm:reduction_principle_Hardy_simplified_item} implies \ref{thm:reduction_principle_Hardy_item} provided that either $m\geq2$ and $f\mapsto f^{**}$ is bounded on $X(0, \infty)$ or $m = 1$:} When $m = 1$, there is nothing to prove, for $T_1 = S_1$. When $m\geq2$ and $f\mapsto f^{**}$ is bounded on $X(0, \infty)$, the statement follows immediately from \cref{prop:operator_induced_norms_equivalences}\ref{prop:operator_induced_norms_equivalences_nu_sigma_item} combined with \eqref{thm:reduction_principle_Hardy_ineq_dual} and \eqref{thm:reduction_principle_Hardy_ineq_simplified_dual}.
\end{proof}

\begin{proof}[Proof of \cref{thm:optimal_target_general_theorem}]
First, note that the functional $\|\cdot\|_{Y_{m, X}'(0, \infty)}$ is indeed a rearrangement\hyp{}invariant function norms thanks to \cref{prop:operator_induced_norms}. Moreover, we have $\|\cdot\|_{Y_{m, X}'(0, \infty)} = \nu_{m, X'}(\cdot)$. In particular, the rearrangement\hyp{}invariant function space $Y_{m, X}(\Hn)$ is well defined.

Second, owing to \eqref{thm:reduction_principle_Hardy_ineq_dual} with $Y(\Hn) = Y_{m, X}(\Hn)$, we have
\begin{equation*}
\sup_{\substack{f\in\Mpl(0, \infty)\\ \|f\|_{X(0, \infty)}\leq1}} \|T_m f\|_{Y_{m, X}(0, \infty)} < \infty.
\end{equation*}
In other words, \eqref{thm:reduction_principle_Hardy_ineq} is valid with $Y(\Hn) = Y_{m, X}(\Hn)$. Hence, by \cref{thm:reduction_principle}, the inequality \eqref{intro:Sob_ineq} is valid with $Y(\Hn) = Y_{m, X}(\Hn)$.

Third, let $Y(\Hn)$ be a rearrangement\hyp{}invariant function space such that \eqref{intro:Sob_ineq} is valid. We need to show that $Y_{m, X}(\Hn) \hookrightarrow Y(\Hn)$. Owing to \cref{thm:reduction_principle} combined with \eqref{thm:reduction_principle_Hardy_ineq_dual}, we have
\begin{equation}\label{thm:optimal_target_general_theorem_aux}
\sup_{\substack{g\in\Mpl(0, \infty)\\ \|g\|_{Y'(0, \infty)}\leq1}} \|g\|_{Y_{m, X}'(0, \infty)} \approx \sup_{\substack{f\in\Mpl(0, \infty)\\ \|f\|_{X(0, \infty)}\leq1}} \|T_m f\|_{Y(0, \infty)} < \infty;
\end{equation}
in other words, $Y'(\Hn) \hookrightarrow Y_{m, X}'(\Hn)$. Hence $Y_{m, X}(\Hn) \hookrightarrow Y(\Hn)$ by \eqref{prel:ri:XtoYiffY'toX'}. Furthermore, it follows from \eqref{thm:optimal_target_general_theorem_aux} that $\|\chi_{(0, 1)}\|_{Y_{m, X}'(0, \infty)} < \infty$. Combining that with \eqref{prop:operator_induced_norms_lower_bound_on_R2k_chi_0_1} and \eqref{prop:operator_induced_norms_sigma_lesssim_nu}, we see that the condition \eqref{thm:optimal_target_general_theorem_condition} is satisfied if there is a rearrangement\hyp{}invariant function space $Y(\Hn)$ such that \eqref{intro:Sob_ineq} is valid.
\end{proof}

\begin{proof}[Proof of \cref{thm:optimal_target_general_theorem_simplified_M_bounded_X,thm:optimal_target_general_theorem_simplified_M_bounded_Xasoc,thm:optimal_target_general_theorem_simplified_M_bounded_on_both_X_and_Xasoc}]
In all three theorems, the fact that the functional $\|\cdot\|_{Y_{m, X}'(0, \infty)}$ defined in the respective theorem is a rearrangement\hyp{}invariant function norm follows from \cref{prop:operator_induced_norms}.

We start with \cref{thm:optimal_target_general_theorem_simplified_M_bounded_X}. For $m = 1$, the statement of the theorem is the same as that of \cref{thm:optimal_target_general_theorem}, which was already proved. For $m\geq2$, the statement follows from \cref{thm:optimal_target_general_theorem} combined with \cref{prop:operator_induced_norms_equivalences}\ref{prop:operator_induced_norms_equivalences_nu_sigma_item}.

Next, we turn our attention to \cref{thm:optimal_target_general_theorem_simplified_M_bounded_Xasoc}. For $m\neq 2$, the statement follows from \cref{thm:optimal_target_general_theorem} combined with \eqref{prop:operator_induced_norms_equivalences_nu_mu}.

Finally, as for \cref{thm:optimal_target_general_theorem_simplified_M_bounded_on_both_X_and_Xasoc}, note that the condition \eqref{thm:optimal_target_general_theorem_condition} is satisfied (recall  \cref{rem:boundedness_P_on_Xasoc_implies_necessary_cond}\ref{rem:boundedness_P_on_Xasoc_implies_necessary_cond_explanation_item}). Now, combining \cref{thm:optimal_target_general_theorem} with \eqref{prop:operator_induced_norms_equivalences_nu_sigma} and \eqref{prop:operator_induced_norms_equivalences_sigma_lambda}, we obtain the statement.
\end{proof}

\begin{proof}[Proof of \cref{thm:optimal_target_eucl_and_supercritical}]
We start with \ref{thm:optimal_target_eucl_and_supercritical_subcritical_item}. Note that the functional $\|\cdot\|_{Z_{m, X}'(0, \infty)}$ is indeed a rearrangement\hyp{}invariant function norm (e.g., see \cite[Theorem~4.4]{EMMP:20}), and so the rearrangement\hyp{}invariant function space $Z_{m, X}(\Hn)$ (in turn, $Y_{m, X}(\Hn)$, too), is well defined. Set $k = \lceil m/2 - 1 \rceil$ and let $\beta$ be defined by \eqref{thm:optimal_target_general_theorem_beta}. Note that $2k + \beta = m$. Recall that the operator $S_m$ is defined by \eqref{operator_Sm_def}, and the kernel $K_k(\cdot, \cdot)$ is defined by \eqref{E:def_kernel_K} with $j$ replaced by $k$.

Now, we will show that \eqref{intro:Sob_ineq} is valid with $Y_{m, X}(\Hn) = X(\Hn) \cap Z_{m, X}(\Hn)$. First, since $\phi_\alpha(t)\leq t^{-1}$, $\alpha\in[0, n)$, for every $t\in(0, \infty)$, we have
\begin{equation*}
S_m f(t) \leq Q^{\lceil \frac{m}{2} \rceil} f(t) \quad \text{for every $t\in(0, \infty)$ and every $f\in\Mpl(0, \infty)$}.
\end{equation*}
Hence
\begin{equation}\label{thm:optimal_target_eucl_and_supercritical_bounded_into_X}
\|S_m f\|_{X(0, \infty)} \leq \|Q^{\lceil \frac{m}{2} \rceil} f\|_{X(0, \infty)} \leq \|Q\|^{\lceil \frac{m}{2} \rceil}_{X(0, \infty)} \|f\|_{X(0, \infty)}
\end{equation}
for every $f\in\Mpl(0, \infty)$, in which $\|Q\|_{X(0, \infty)}$ stands for the operator norm of $Q$ on $X(0, \infty)$. Second, since $\phi_\alpha(t)\leq t^{-1 + \frac{\alpha}{n}}$, $\alpha\in[0, n)$, for every $t\in(0, \infty)$, we also have
\begin{equation*}
S_m f(t) \leq \frac1{k!} \int_t^\infty f(s) s^{-1 + \frac{\beta}{n}} \Big( \int_t^s \tau^{-1 + \frac{2}{n}} \d{\tau} \Big)^k \leq \frac{(\frac{n}{2})^k}{k!} \int_t^\infty f(s)s^{-1 + \frac{2k + \beta}{n}} \d{s}
\end{equation*}
for every $f\in\Mpl(0, \infty)$ and every $t\in(0, \infty)$. Hence
\begin{equation}\label{thm:optimal_target_eucl_and_supercritical_bounded_into_Eucl_optimal_upper_bound_on_Sm}
\|S_m f\|_{Z_{m, X}(0, \infty)} \leq c_{m, n} \Big\| \int_t^\infty f(s) s^{-1 + \frac{m}{n}} \d{s} \Big\|_{Z_{m, X}(0, \infty)}
\end{equation}
for every $f\in\Mpl(0, \infty)$. Here $c_{m, n}$ is a constant depending only on $m$ and $n$. Furthermore, we also have
\begin{equation}\label{thm:optimal_target_eucl_and_supercritical_bounded_into_Eucl_optimal_norm_of_Eucl_Hardy}
\sup_{\substack{f\in\Mpl(0, \infty)\\ \|f\|_{X(0, \infty)}\leq1}} \Big \|\int_t^\infty f(s) s^{-1 + \frac{m}{n}} \d{s} \Big\|_{Z_{m, X}(0, \infty)} = 1.
\end{equation}
Indeed, using \eqref{prel:ri:normX''down}, \eqref{prel:ri:normX'} and the definition of $\|\cdot\|_{Z_{m, X}'(0, \infty)}$, we observe that
\begin{align*}
&\sup_{\substack{f\in\Mpl(0, \infty)\\ \|f\|_{X(0, \infty)}\leq1}} \Big \|\int_t^\infty f(s) s^{-1 + \frac{m}{n}} \d{s} \Big\|_{Z_{m, X}(0, \infty)} \notag\\
&= \sup_{\substack{f, g\in\Mpl(0, \infty)\\ \|f\|_{X(0, \infty)}\leq 1, \|g\|_{Z_{m, X}'(0, \infty)}\leq1}} \int_0^\infty g^*(t) \int_t^\infty f(s) s^{-1 + \frac{m}{n}} \d{s} \d{t} \notag\\
&= \sup_{\substack{f, g\in\Mpl(0, \infty)\\ \|f\|_{X(0, \infty)}\leq 1, \|g\|_{Z_{m, X}'(0, \infty)}\leq1}} \int_0^\infty f(s) s^{-1 + \frac{m}{n}} \int_0^s g^*(t) \d{t} \d{s} \notag\\
&= \sup_{\substack{g\in\Mpl(0, \infty)\\ \|g\|_{Z_{m, X}'(0, \infty)}\leq1}} \|t^{\frac{m}{n}}g^{**}(t)\|_{X'(0, \infty)} = \sup_{\substack{g\in\Mpl(0, \infty)\\ \|g\|_{Z_{m, X}'(0, \infty)}\leq1}} \|g\|_{Z_{m, X}'(0, \infty)}.
\end{align*}
Hence, thanks to \eqref{thm:optimal_target_eucl_and_supercritical_bounded_into_Eucl_optimal_upper_bound_on_Sm} and \eqref{thm:optimal_target_eucl_and_supercritical_bounded_into_Eucl_optimal_norm_of_Eucl_Hardy}, we have
\begin{equation}\label{thm:optimal_target_eucl_and_supercritical_bounded_into_Z}
\|S_m f\|_{Z_{m, X}(0, \infty)} \leq c_{m, n} \|f\|_{X(0, \infty)} \quad \text{for every $f\in\Mpl(0, \infty)$}.
\end{equation}
Finally, combining \eqref{thm:optimal_target_eucl_and_supercritical_bounded_into_X} and \eqref{thm:optimal_target_eucl_and_supercritical_bounded_into_Z} with \cref{thm:reduction_principle}, we obtain
\begin{equation*}
\|u\|_{X(\Hn)} + \|u\|_{Z_{m, X}(\Hn)} \leq C \|\gradH[m] u\|_{X(\Hn)} \quad \text{for every $u\in V_0^mX(\Hn)$}.
\end{equation*}

It remains to prove the optimality of $Y_{m, X}(\Hn)$. Owing to \cref{thm:optimal_target_general_theorem_simplified_M_bounded_on_both_X_and_Xasoc} (and \cref{rem:optimal_target_general_theorem_simplified_M_bounded_on_both_X_and_Xasoc_m_one_two_already_known}\ref{rem:optimal_target_general_theorem_simplified_M_bounded_on_both_X_and_Xasoc_m_one_two_already_known_m_1_2_item}), we know that the rearrangement\hyp{}invariant function space $\widetilde{Y}_{m, X}(\Hn)$ whose associate function norm is defined as
\begin{equation*}
\|g\|_{\widetilde{Y}_{m, X}'(0, \infty)} = \|R_m g^*\|_{X'(0, \infty)},\ g\in\Mpl(0, \infty),
\end{equation*}
is the optimal target space in \eqref{intro:Sob_ineq}. Therefore, it is sufficient to show that
\begin{equation*}
X(0, \infty) \cap Z_{m, X}(0, \infty) \hookrightarrow \widetilde{Y}_{m, X}(0, \infty).
\end{equation*}
To that end, it is clearly sufficient to prove that
\begin{align}
\|f^*\chi_{(0, 1)}\|_{\widetilde{Y}_{m, X}(0, \infty)} &\lesssim \|f\|_{Z_{m, X}(0, \infty)} \quad \text{for every $f\in\Mpl(0, \infty)$} \label{thm:optimal_target_eucl_and_supercritical_W_near_0_into_Z}\\
\intertext{and}
\|f^*\chi_{(1, \infty)}\|_{\widetilde{Y}_{m, X}(0, \infty)} &\lesssim \|f\|_{X(0, \infty)} \quad \text{for every $f\in\Mpl(0, \infty)$}. \label{thm:optimal_target_eucl_and_supercritical_W_near_infty_into_X}
\end{align}
We start by proving the latter. Thanks to \cref{thm:reduction_principle}, we have
\begin{equation}\label{thm:optimal_target_eucl_and_supercritical_optimality_Hardy_bounded_from_X_to_W}
\|S_m f\|_{\widetilde{Y}_{m, X}(0, \infty)} \lesssim \|f\|_{X(0, \infty)} \quad \text{for every $f\in\Mpl(0, \infty)$}.
\end{equation}
Using \eqref{prop:iterated_operators_closed_form:E5} and \eqref{E:kernel_K_monotonicity_second_var}, we have
\begin{align*}
f^*(t)\chi_{(1, \infty)}(t) &\approx f^*(t)\chi_{(1, \infty)}(t) \int_{\frac{t}{2}}^t \phi_m(s) \d{s} \\
&\approx f^*(t)\chi_{(1, \infty)}(t) \int_{\frac{t}{2}}^t \phi_m(s) K_k\Big( \frac{t}{4}, \frac{t}{2} \Big) \d{s} \\
&\leq \int_{\frac{t}{4}}^\infty f^*(s) \phi_m(s) K_k\Big( \frac{t}{4}, s \Big) \d{s}
\end{align*}
for all $f\in\Mpl(0, \infty)$ and $t\in(0, \infty)$. Furthermore, thanks to this, the boundedness of the dilation operator $D_4$ on $\widetilde{Y}_{m, X}(0, \infty)$ (see \eqref{prel:ri:dilation}), and \eqref{prop:iterated_operators_closed_form:E-HH}, we obtain
\begin{equation}\label{thm:optimal_target_eucl_and_supercritical_optimality_norm_W_near_infinity}
\|f^*\chi_{(1, \infty)}\|_{\widetilde{Y}_{m, X}(0, \infty)} \lesssim \| S_m f^* \|_{\widetilde{Y}_{m, X}(0, \infty)} \quad \text{for every $f\in\Mpl(0, \infty)$}.
\end{equation}
Hence, by combining \eqref{thm:optimal_target_eucl_and_supercritical_optimality_Hardy_bounded_from_X_to_W} with \eqref{thm:optimal_target_eucl_and_supercritical_optimality_norm_W_near_infinity}, we obtain \eqref{thm:optimal_target_eucl_and_supercritical_W_near_infty_into_X}. Now, the proof of the optimality of $Y_{m, X}(\Hn)$ will be complete once we prove \eqref{thm:optimal_target_eucl_and_supercritical_W_near_0_into_Z}. To that end, using \eqref{prel:ri:normX''down} twice together with the definitions of $\|\cdot\|_{Z_{m, X}'(0, \infty)}$ and $\|\cdot\|_{\widetilde{Y}_{m, X}'(0, \infty)}$, we observe that
\begin{align}
\sup_{\substack{f\in\Mpl(0, \infty)\\ \|f\|_{Z_{m, X}(0, \infty)}\leq1}}\|f^*\chi_{(0,1)}\|_{\widetilde{Y}_{m, X}(0, \infty)} &= \sup_{\substack{g\in\Mpl(0, \infty)\\ \|g\|_{\widetilde{Y}_{m, X}'(0, \infty)}\leq1}}\|g^*\chi_{(0,1)}\|_{Z_{m, X}'(0, \infty)} \notag \\
&\leq \sup_{\substack{g\in\Mpl(0, \infty)\\ \|g\|_{\widetilde{Y}_{m, X}'(0, \infty)}\leq1}}\|t^{\frac{m}{n}}g^{**}(t)\chi_{(0,1)}(t)\|_{X'(0, \infty)} \notag \\
&\quad + \sup_{\substack{g\in\Mpl(0, \infty)\\ \|g\|_{\widetilde{Y}_{m, X}'(0, \infty)}\leq1}}\Big\| t^{-1 + \frac{m}{n}}\Big(\int_0^1 g^*\Big) \chi_{(1, \infty)}(t) \Big\|_{X'(0, \infty)} \notag\\
&\leq 1 + \sup_{\substack{g\in\Mpl(0, \infty)\\ \|g\|_{\widetilde{Y}_{m, X}'(0, \infty)}\leq1}}\Big\| t^{-1 + \frac{m}{n}} \Big(\int_0^1 g^*\Big) \chi_{(1, \infty)}(t) \Big\|_{X'(0, \infty)}. \label{thm:optimal_target_eucl_and_supercritical_W_near_0_into_Z_upper_bound}
\end{align}
Using the H\"older inequality \eqref{prel:ri:holder} and the assumption \eqref{thm:optimal_target_eucl_and_supercritical_p_(1_n/m)_condition}, we obtain
\begin{align*}
&\sup_{\substack{g\in\Mpl(0, \infty)\\ \|g\|_{\widetilde{Y}_{m, X}'(0, \infty)}\leq1}}\Big\| t^{-1 + \frac{m}{n}} \Big(\int_0^1 g^*\Big) \chi_{(1, \infty)}(t) \Big\|_{X'(0, \infty)} \\
&\leq \|t^{-1 + \frac{m}{n}}\chi_{(1, \infty)}(t)\|_{X'(0, \infty)} \|\chi_{(0, 1)}\|_{\widetilde{Y}_{m, X}(0, \infty)} \sup_{\substack{g\in\Mpl(0, \infty)\\ \|g\|_{\widetilde{Y}_{m, X}'(0, \infty)}\leq1}} \|g\|_{\widetilde{Y}_{m, X}'(0, \infty)} \\
&< \infty.
\end{align*}
Hence, combining that with \eqref{thm:optimal_target_eucl_and_supercritical_W_near_0_into_Z_upper_bound}, we obtain \eqref{thm:optimal_target_eucl_and_supercritical_W_near_0_into_Z}.

Finally, assume that \eqref{thm:optimal_target_general_theorem_condition} is satisfied but \eqref{thm:optimal_target_eucl_and_supercritical_p_(1_n/m)_condition} is not. Thanks to \cref{thm:optimal_target_general_theorem} (or, under some additional assumptions, \cref{thm:optimal_target_general_theorem_simplified_M_bounded_X,thm:optimal_target_general_theorem_simplified_M_bounded_Xasoc,thm:optimal_target_general_theorem_simplified_M_bounded_on_both_X_and_Xasoc}), there is an optimal target space for $X(\Hn)$ in \eqref{intro:Sob_ineq}. There are several ways of proving that the optimal target space is not equivalent to $Z_{m, X}(\Hn) \cap X(\Hn)$, but arguably the most straightforward one is to observe that $\|\chi_{(0, a)}\|_{Z_{m, X}'(0, \infty)} = \infty$ for every $a>0$ if \eqref{thm:optimal_target_eucl_and_supercritical_p_(1_n/m)_condition} is not satisfied. In other words, $\|\cdot\|_{Z_{m, X}'(0, \infty)}$ is not even a rearrangement\hyp{}invariant function norm (in turn, neither is $\|\cdot\|_{Z_{m, X}(0, \infty)}$). This completes the proof of \ref{thm:optimal_target_eucl_and_supercritical_subcritical_item}.

We now turn our attention to \ref{thm:optimal_target_eucl_and_supercritical_supercritical_item}. For the most part, the proof is similar to that of \ref{thm:optimal_target_eucl_and_supercritical_subcritical_item}\textemdash in fact, simpler. First, we know from \eqref{thm:optimal_target_eucl_and_supercritical_bounded_into_X} that the operator $S_m$ is bounded on $X(0, \infty)$. We need to show that $S_m$ is also bounded from $X(0, \infty)$ to $L^\infty(0, \infty)$. To that end, note that
\begin{equation}\label{thm:optimal_target_eucl_and_supercritical_optimality_Sm_in_Linfty}
\|S_mf\|_{L^\infty(0, \infty)} \approx \int_0^\infty f(t) \phi_m(t) \big(\chi_{(0,1)}(t) + \log(et)^k\chi_{(1, \infty)}(t)\big) \d{t}
\end{equation}
for every $f\in\Mpl(0, \infty)$, thanks to \eqref{prop:iterated_operators_closed_form:E-HH} and \eqref{E:kernel_K_monotonicity_first_var}.
Using this, the H\"older inequality \eqref{prel:ri:holder} and the assumption \eqref{thm:optimal_target_eucl_and_supercritical_p_(n/m_infty)_condition}, we obtain
\begin{align*}
\sup_{\substack{f\in\Mpl(0, \infty)\\ \|f\|_{X(0, \infty)}\leq1}} \|S_mf\|_{L^\infty(0, \infty)} \lesssim \Big\|\phi_m(t) \big( \chi_{(0,1)}(t) + \log(es)^k\chi_{(1, \infty)}(t) \big) \Big\|_{X'(0, \infty)} < \infty.
\end{align*}
Owing to \cref{thm:reduction_principle}, it follows that \eqref{intro:Sob_ineq} is valid with $Y_{m, X}(\Hn) = X(\Hn) \cap L^\infty(\Hn)$. As for the optimality, let $\widetilde{Y}_{m, X}(\Hn)$ be as in the proof of \ref{thm:optimal_target_eucl_and_supercritical_subcritical_item}. We need to show that $X(\Hn) \cap L^\infty(\Hn) \hookrightarrow \widetilde{Y}_{m, X}(\Hn)$. As in the proof of \ref{thm:optimal_target_eucl_and_supercritical_subcritical_item}, it is sufficient to prove \eqref{thm:optimal_target_eucl_and_supercritical_W_near_infty_into_X} and \eqref{thm:optimal_target_eucl_and_supercritical_W_near_0_into_Z} with $Z_{m, X}(0, \infty)$ replaced by $L^\infty(0, \infty)$. The validity of \eqref{thm:optimal_target_eucl_and_supercritical_W_near_infty_into_X} can be proved exactly the same way as before. Therefore, all that remains is to prove that
\begin{equation*}
\sup_{\substack{f\in\Mpl(0, \infty)\\ \|f\|_{L^\infty(0, \infty)}\leq1}}\|f^*\chi_{(0, 1)}\|_{\widetilde{Y}_{m, X}(0, \infty)} < \infty,
\end{equation*}
which is easy, for
\begin{equation*}
\sup_{\substack{f\in\Mpl(0, \infty)\\ \|f\|_{L^\infty(0, \infty)}\leq1}}\|f^*\chi_{(0, 1)}\|_{\widetilde{Y}_{m, X}(0, \infty)} = \|\chi_{(0, 1)}\|_{\widetilde{Y}_{m, X}(0, \infty)} < \infty.
\end{equation*}

Finally, assume that \eqref{thm:optimal_target_general_theorem_condition} is satisfied. Thanks to \cref{thm:optimal_target_general_theorem} (or, under some additional assumptions, \cref{thm:optimal_target_general_theorem_simplified_M_bounded_X,thm:optimal_target_general_theorem_simplified_M_bounded_Xasoc,thm:optimal_target_general_theorem_simplified_M_bounded_on_both_X_and_Xasoc}), there is an optimal target space for $X(\Hn)$ in \eqref{intro:Sob_ineq}. Assume that \eqref{intro:Sob_ineq} is valid with a rearrangement\hyp{}invariant function space $Y(\Hn)$ such that $Y(\Hn) \subseteq L^\infty(\Hn)$. In particular, it is valid with $L^\infty(\Hn)$ instead of $Y(\Hn)$. Hence, owing to \eqref{prel:ri:normX'}, \eqref{thm:optimal_target_eucl_and_supercritical_optimality_Sm_in_Linfty}, and \cref{thm:reduction_principle}, we have
\begin{align*}
\|t^{-1 + \frac{m}{n}}\chi_{(0,1)}(t)\|_{X'(0, \infty)} &=  \sup_{\substack{f\in\Mpl(0, \infty)\\ \|f\|_{X(0, \infty)}\leq1}} \int_0^\infty f(t) t^{-1 + \frac{m}{n}}\chi_{(0,1)}(t) \d{t}\\
&\lesssim \sup_{\substack{f\in\Mpl(0, \infty)\\ \|f\|_{X(0, \infty)}\leq1}} \|S_mf\|_{L^\infty(0, \infty)} < \infty.
\end{align*}
This together with the validity of \eqref{thm:optimal_target_general_theorem_condition} implies that \eqref{thm:optimal_target_eucl_and_supercritical_p_(n/m_infty)_condition} is satisfied.
\end{proof}

\begin{proof}[Proof of \cref{thm:examples_LZ}]
We begin by noting that, in all the cases considered in the theorem, $X(\Hn) = L^{p, q; \A}(\Hn)$ is (equivalent to) a rearrangement\hyp{}invariant function space, and that we have (recall \eqref{thm:examples_LZ_X_asoc})
\begin{equation*}
X'(0, \infty) = L^{p', q'; -\A}(0, \infty).
\end{equation*}
Set $k = \lceil m/2 - 1 \rceil$.

First, assume that either $m = 1$, $p=q=1$, $\alpha_0\geq0$, and $\alpha_\infty\leq 0$ or $p\in(1, n/m)$. In this case, we may use \cref{thm:optimal_target_eucl_and_supercritical}\ref{thm:optimal_target_eucl_and_supercritical_subcritical_item} to obtain
\begin{equation*}
 Y_{m, X}(\Hn) = Z_{m, X}(\Hn) \cap X(\Hn),
\end{equation*}
where $Z_{m, X}(\Hn)$ is a rearrangement\hyp{}invariant function space whose associate function norm is defined as
\begin{equation*}
\|g\|_{Z_{m, X}'(0, \infty)} = \|t^{\frac{m}{n}} g^{**}(t)\|_{L^{p', q'; -\A}(0, \infty)},\ g\in\Mpl(0, \infty).
\end{equation*}
Combining that with \cite[Theorem~5.1]{M:21}, it is easy to see that
\begin{equation*}
\|f\|_{Y_{m, X}(0, \infty)} \approx \|f\|_{\Lambda^q_{v_1}(0, \infty)} \quad \text{for every $f\in\Mpl(0, \infty)$}.
\end{equation*}

Second, assume that either $p = n/m$, $q = 1$, and $\alpha_0\geq0$ or $p = n/m$, $q \in (1, \infty]$, and $\alpha_0 > \frac1{q'}$, or $p\in(n/m, \infty)$. In this case, the statement follows immediately from \cref{thm:optimal_target_eucl_and_supercritical}\ref{thm:optimal_target_eucl_and_supercritical_supercritical_item}.

Third, assume that $p = n/m$ and either $q = 1$ and $\alpha_0 < 0$ or $q\in(1, \infty]$ and $\alpha_0 \leq \frac1{q'}$. We claim that
\begin{equation}\label{thm:examples_LZ_p=n/m_not_supercritical_norm_asoc}
\|g\|_{Y_{m, X}'(0, \infty)} \approx \big\| \min\{t^{1 - \frac1{q'}} \ell^{-\alpha_0}(t), t^{1 - \frac{m}{n} - \frac1{q'}} \ell^{-\alpha_\infty}(t)\}g^{**}(t) \big\|_{L^{q'}(0, \infty)}
\end{equation}
for every $g\in\Mpl(0, \infty)$. To that end, thanks to \cref{thm:optimal_target_general_theorem_simplified_M_bounded_on_both_X_and_Xasoc} (see also \cref{rem:optimal_target_general_theorem_simplified_M_bounded_on_both_X_and_Xasoc_m_one_two_already_known}\ref{rem:optimal_target_general_theorem_simplified_M_bounded_on_both_X_and_Xasoc_m_one_two_already_known_m_1_2_item}), we have
\begin{equation}\label{thm:examples_LZ_p=n/m_not_supercritical_norm_asoc_with_Rm}
\|g\|_{Y_{m, X}'(0, \infty)} = \big\| t^{1 - \frac{m}{n} - \frac1{q'}} \ell^{-\A}(t) (R_m g^*)^*(t) \big\|_{L^{q'}(0, \infty)}
\end{equation}
for every $g\in\Mpl(0, \infty)$. Now, if $q = 1$ (and so $q' = \infty$), using the fact that the function $(0, \infty)\ni t \mapsto t^{1 - \frac{m}{n}} \ell^{-\A}(t) $ is equivalent to a nondecreasing function, we have
\begin{align}
\big\| t^{1 - \frac{m}{n} - \frac1{q'}} \ell^{-\A}(t) (R_m g^*)^*(t) \big\|_{L^{q'}(0, \infty)} &\leq \sup_{t\in(0, \infty)} t^{1 - \frac{m}{n}} \ell^{-\A}(t) \sup_{\tau\in[t, \infty)} \tau \phi_m(\tau)  g^{**}(\tau) \notag\\
&= \sup_{\tau\in(0, \infty)} \tau \phi_m(\tau)  g^{**}(\tau) \sup_{t\in(0, \tau]} t^{1 - \frac{m}{n}} \ell^{-\A}(t) \notag\\
&\approx \sup_{\tau\in(0, \infty)} \tau^{1 - \frac{m}{n}} \ell^{-\A}(\tau) \tau \phi_m(\tau)  g^{**}(\tau) \label{thm:examples_LZ_rid_of_star_q=1}
\end{align}
for every $g\in\Mpl(0, \infty)$. When $q\in (1, \infty]$ (and so $q'\in[1, \infty)$), we have
\begin{align}
\big\| t^{1 - \frac{m}{n} - \frac1{q'}} \ell^{-\A}(t) (R_m g^*)^*(t) \big\|_{L^{q'}(0, \infty)} &\leq \big\| t^{1 - \frac{m}{n} - \frac1{q'}} \ell^{-\A}(t) \sup_{\tau\in[t, \infty)} \tau \phi_m(\tau)  g^{**}(\tau) \big\|_{L^{q'}(0, \infty)} \notag\\
&\lesssim \big\| t^{1 - \frac{m}{n} - \frac1{q'}} \ell^{-\A}(t) t \phi_m(t)  g^{**}(t) \big\|_{L^{q'}(0, \infty)} \label{thm:examples_LZ_rid_of_star_q>1}
\end{align}
for every $g\in\Mpl(0, \infty)$, where we used \cite[Theorem~3.2]{GOP:06} in the second inequality. Therefore, thanks to \eqref{thm:examples_LZ_p=n/m_not_supercritical_norm_asoc_with_Rm} combined with \eqref{thm:examples_LZ_rid_of_star_q=1} or \eqref{thm:examples_LZ_rid_of_star_q>1}, we have
\begin{equation}\label{thm:examples_LZ_p=n/m_not_supercritical_norm_asoc_with_Rm_upper}
\|g\|_{Y_{m, X}'(0, \infty)} \lesssim \big\| t^{1 - \frac{m}{n} - \frac1{q'}} \ell^{-\A}(t) t \phi_m(t) g^{**}(t) \big\|_{L^{q'}(0, \infty)}
\end{equation}
for every $g\in\Mpl(0, \infty)$. As for the opposite inequality, note that the function $(0, \infty)\ni t \mapsto t \phi_m(t) = \min\{t^{\frac{m}{n}}, 1\}$ is quasiconcave\textemdash a function $\varphi\colon(0, \infty) \to (0, \infty)$ is quasiconcave if it is nondecreasing and the function $(0, \infty)\ni t\mapsto \varphi(t)/t$ is nonincreasing. Hence, in virtue of \cite[Lemma~4.10]{EMMP:20}, we have
\begin{equation*}
\big\| \sup_{\tau\in[t, \infty)} \tau \phi_m(\tau) g^{**}(\tau) \big\|_{X'(0, \infty)} \leq 6 \|t\phi_m(t)g^{**}(t)\|_{X'(0, \infty)}
\end{equation*}
for every $g\in\Mpl(0, \infty)$. Using that, we have
\begin{align}
\big\| t^{1 - \frac{m}{n} - \frac1{q'}} \ell^{-\A}(t) t \phi_m(t) g^{**}(t) \big\|_{L^{q'}(0, \infty)} &\leq \big\| t^{1 - \frac{m}{n} - \frac1{q'}} \ell^{-\A}(t) \sup_{\tau\in[t, \infty)} \tau \phi_m(\tau) g^{**}(\tau) \big\|_{L^{q'}(0, \infty)} \notag\\
&= \|\sup_{\tau\in[t, \infty)} \tau \phi_m(\tau) g^{**}(\tau)\|_{L^{p', q'; -\A}(0, \infty)} \notag\\
&\lesssim \|t \phi_m(t) g^{**}(t)\|_{L^{p', q'; -\A}(0, \infty)} \notag\\
&= \|g\|_{Y_{m, X}'(0, \infty)} \label{thm:examples_LZ_p=n/m_not_supercritical_norm_asoc_with_Rm_lower}
\end{align}
for every $g\in\Mpl(0, \infty)$. Hence, by observing that
\begin{equation*}
\ell^{-\A}(t) t \phi_m(t) \approx \min\{t^{m/n} \ell^{-\alpha_0}(t), \ell^{-\alpha_\infty}(t)\} \quad \text{for every $t\in(0, \infty)$}
\end{equation*}
and by combining \eqref{thm:examples_LZ_p=n/m_not_supercritical_norm_asoc_with_Rm_upper} and \eqref{thm:examples_LZ_p=n/m_not_supercritical_norm_asoc_with_Rm_lower}, we obtain \eqref{thm:examples_LZ_p=n/m_not_supercritical_norm_asoc}. Now, when $q = 1$ (and so $q'=\infty$), it follows from \eqref{thm:examples_LZ_p=n/m_not_supercritical_norm_asoc} combined with \cite[Section~3]{S:72}
that
\begin{equation*}
\|f\|_{Y_{m, X}(0, \infty)} \approx \|f\|_{\Lambda^1_{v_2}(0, \infty)} \quad \text{for every $f\in\Mpl(0, \infty)$}.
\end{equation*}
When $q \in (1, \infty]$ (and so $q'\in[1, \infty)$), it follows from \eqref{thm:examples_LZ_p=n/m_not_supercritical_norm_asoc} combined with \cite[Theorem~6.2]{GP:03} that
\begin{equation*}
\|f\|_{Y_{m, X}(0, \infty)} \approx \|f^{**}\|_{\Lambda^q_{v_j}(0, \infty)} \quad \text{for every $f\in\Mpl(0, \infty)$},
\end{equation*}
where $j = 2$ if $\alpha_0 < 1/q'$; $j = 3$ if $\alpha_0 = 1/q'$. Moreover, in view of \eqref{prel:ri:twostarsdominateonestar} and weighted Hardy's inequality \cite[Theorem~1]{M:72}, we have
\begin{equation*}
\|f^{**}\|_{\Lambda^q_{v_j}(0, \infty)} \approx \|f\|_{\Lambda^q_{v_j}(0, \infty)} \quad \text{for every $f\in\Mpl(0, \infty)$}.
\end{equation*}

Next, assume that $p = q = \infty$, $\alpha_0\leq 0$, and $\alpha_\infty > \lceil m/2 \rceil$. We start with two observations. First, when $\widetilde{X}(0, \infty) = L^{\infty,\infty; [\beta_0, \beta_\infty]}(0, \infty)$ is a Lorentz--Zygmund space (that is, $\beta_0\leq 0$), then the condition \eqref{thm:optimal_target_general_theorem_condition} with $m \in \{1, 2\}$ is satisfied if and only if $\beta_\infty > 1$. Second, thanks to \cref{thm:optimal_target_general_theorem} and \eqref{prel:ri:X''=X}, we see that
\begin{equation*}
Y_{m, \tilde{X}}(0, \infty) = \begin{cases}
Y_{m-2, Y_{2, \tilde{X}}}(0, \infty) \quad &\text{if $m\geq4$ is even},\\
Y_{m-1, Y_{1, \tilde{X}}}(0, \infty) \quad &\text{if $m\geq3$ is odd}.
\end{cases} 
\end{equation*}
In view of these two observations, it will follow that
\begin{equation}\label{thm:examples_LZ_p=q=infty}
Y_{m, L^{\infty, \infty; [\alpha_0, \alpha_\infty]}}(0, \infty) = L^{\infty, \infty; [0, \alpha_\infty - \lceil m/2 \rceil]}(0, \infty)
\end{equation}
once we show that
\begin{equation}\label{thm:examples_LZ_p=q=infty_m_1_or_2}
Y_{j, L^{\infty, \infty; [\beta_0, \beta_\infty]}}(0, \infty) = L^{\infty, \infty; [0, \beta_\infty - 1]}(0, \infty)
\end{equation}
for $j\in\{1, 2\}$, $\beta_0\leq0$, and $\beta_\infty > 1$. To that end, thanks to \cref{thm:optimal_target_general_theorem_simplified_M_bounded_X}, we have
\begin{equation*}
\|g\|_{Y_{j, L^{\infty, \infty; [\beta_0, \beta_\infty]}}'(0, \infty)} = \big\| \ell^{[-\beta_0, -\beta_\infty]}(t)(R_jg^*)^*(t) \big\|_{L^1(0, \infty)}
\end{equation*}
for every $g\in\Mpl(0, \infty)$. Arguing as in the previous case, we obtain that
\begin{align*}
\big\| \ell^{[-\beta_0, -\beta_\infty]}(t)(R_jg^*)^*(t) \big\|_{L^1(0, \infty)} &\approx \|t\phi_j(t) \ell^{[-\beta_0, -\beta_\infty]}(t) g^{**}(t)\|_{L^1(0, \infty)} \\
&\approx \|\min\{t^{\frac{j}{n}} \ell^{-\beta_0}(t), \ell^{-\beta_\infty}(t)\} g^{**}(t)\|_{L^1(0, \infty)}
\end{align*}
for every $g\in\Mpl(0, \infty)$. Combining that, \cite[Theorem~6.2]{GP:03}, and \cite[Theorem~3.8]{OP:99}, we obtain \eqref{thm:examples_LZ_p=q=infty_m_1_or_2}. Now that we have established \eqref{thm:examples_LZ_p=q=infty}, it only remains to observe that
\begin{align*}
\|f\|_{L^{\infty, \infty; [0, \alpha_\infty - k - 1]}(0, \infty)} &\approx \|f^{**}\|_{L^{\infty, \infty; [0, \alpha_\infty - k - 1]}(0, \infty)} \\
&\approx \|f\|_{L^\infty(0, \infty)}  + \sup_{t\in[1, \infty)}\ell^{\alpha_\infty - k - 1}(t)f^{**}(t) \\
&= \|f\|_{Z_3(0, \infty)}
\end{align*}
for every $f\in(0, \infty)$; as for the first equivalence, see \cite[Theorem~3.8]{OP:99}.

Next, assume that $m$ is even and that $p=q=1$ and $\alpha_0 = \alpha_\infty = 0$; that is, $X(0, \infty) = L^1(0, \infty)$.
Using \cref{thm:optimal_target_general_theorem_simplified_M_bounded_Xasoc} and \eqref{thm:optimal_target_eucl_and_supercritical_optimality_Sm_in_Linfty}, we obtain
\begin{align*}
\|g\|_{Y_{m, L^1}'(0, \infty)} &= \|S_m(g^{**})\|_{L^\infty(0, \infty)} \\
&\approx  \int_0^\infty g^{**}(t) \phi_m(t) \big(\chi_{(0,1)}(t) + \log(et)^k\chi_{(1, \infty)}(t)\big) \d{t} \\
&\approx \int_0^\infty g^{**}(t) \min\{t^{-1 + \frac{m}{n}}, t^{-1}\ell^k(t)\} \d{t}
\end{align*}
for every $g \in \Mpl(0, \infty)$.
Moreover, it follows from \eqref{prel:ri:twostarsdominateonestar} and weighted Hardy's inequality \cite[Theorem~1]{M:72} that
\begin{equation*}
\|g\|_{Y_{m, L^1}'(0, \infty)} \approx \int_0^\infty g^*(t) \min\{t^{-1 + \frac{m}{n}}, t^{-1}\ell^k(t)\} \d{t} \quad \text{for every $g\in\Mpl(0, \infty)$}.
\end{equation*}
Combining this with \cite[Section~3]{S:72}, we obtain
\begin{equation*}
\|f\|_{Y_{m, L^1}(0, \infty)} \approx \sup_{t\in(0, \infty)} \max\{t^{1 - \frac{m}{n}}, t \ell^{-k - 1}(t)\} f^{**}(t) \approx \|f\|_{Z_2(0, \infty)}
\end{equation*}
for every $f\in\Mpl(0, \infty)$.

Lastly, assume that $m\geq3$ (in particular, $k\geq 1$) is odd and $X(0, \infty) = L^1(0, \infty)$. Owing to \cref{thm:optimal_target_general_theorem_simplified_M_bounded_Xasoc}, we have
\begin{equation*}
\|g\|_{Y_{m, L^1}'(0, \infty)} = \|R_1(H_2^k(g^{**}))\|_{L^\infty(0, \infty)} \quad \text{for every $g\in\Mpl(0, \infty)$}.
\end{equation*}
Arguing as in the third case, we obtain
\begin{equation*}
\|R_1(H_2^k(g^{**}))\|_{L^\infty(0, \infty)} \approx \|t\phi_1(t)(H_2^k(g^{**}))^{**}(t)\|_{L^\infty(0, \infty)}
\end{equation*}
for every $g\in\Mpl(0, \infty)$. Moreover, thanks to \eqref{prel:ri:twostarsdominateonestar} and \cite[Theorem~1]{M:72}, we also obtain
\begin{equation*}
\|t\phi_1(t)(H_2^k(g^{**}))^{**}(t)\|_{L^\infty(0, \infty)} \approx \|t\phi_1(t)H_2^k(g^{**})(t)\|_{L^\infty(0, \infty)}
\end{equation*}
for every $g\in\Mpl(0, \infty)$. Therefore, we have
\begin{equation*}
\|g\|_{Y_{m, L^1}'(0, \infty)} \approx \|\min\{t^{\frac1{n}}, 1\}H_2^k(g^{**})(t)\|_{L^\infty(0, \infty)} \quad \text{for every $g\in\Mpl(0, \infty)$}.
\end{equation*}
Using this and \eqref{prop:properties_of_iterated_operators_Hjk_Hi}, we arrive at
\begin{equation}\label{thm:examples_LZ_p=q=1_m_odd_gtr3_asoc_norm_explicit_with_kernel}
\|g\|_{Y_{m, L^1}'(0, \infty)} \approx \sup_{t\in(0, \infty)}\min\{t^{\frac1{n}}, 1\} \int_t^\infty g^{**}(s) \phi_2(s) \Big( \int_t^s \phi_2 \Big)^{k-1} \d{s}
\end{equation}
for every $g\in\Mpl(0, \infty)$. To simplify the notation in the rest of the proof, set
\begin{equation*}
\psi(t) = \min\{t^{-1 + \frac{2k}{n}}, t^{-1}\ell^{k-1}(t)\},\ t\in(0, \infty).
\end{equation*}
Note that $2k = m-1$. We claim that
\begin{equation}\label{thm:examples_LZ_p=q=1_m_odd_gtr3_asoc_norm_explicit_copson}
\|g\|_{Y_{m, L^1}'(0, \infty)} \approx \sup_{t\in(0, 1)}t^{\frac1{n}} \int_t^\infty g^{**}(s) \psi(s) \d{s} + \int_1^\infty g^{**}(s) \psi(s) \d{s}
\end{equation}
for every $g\in\Mpl(0, \infty)$. To that end, in view of \eqref{thm:examples_LZ_p=q=1_m_odd_gtr3_asoc_norm_explicit_with_kernel} and \eqref{prop:properties_of_iterated_operators_kernel_phi_alpha_explicit}, we clearly have
\begin{equation}
\|g\|_{Y_{m, L^1}'(0, \infty)} \lesssim \sup_{t\in(0, 1)}t^{\frac1{n}} \int_t^\infty g^{**}(s) \psi(s) \d{s} + \int_1^\infty g^{**}(s) \psi(s) \d{s} \label{thm:examples_LZ_p=q=1_m_odd_gtr3_asoc_norm_explicit_copson_upper}
\end{equation}
for every $g\in\Mpl(0, \infty)$. As for the opposite inequality, we observe that, for every $t\in(0, 1)$ and $g\in\Mpl(0, \infty)$,
\begin{align*}
\int_t^\infty g^{**}(s) \phi_2(s) \Big( \int_t^s \phi_2 \Big)^{k-1} \d{s} &\gtrsim \int_{2t}^2 g^{**}(s) s^{-1 + \frac{2}{n}} \Big( \int_{\frac{s}{2}}^s \tau^{-1 + \frac{2}{n}} \d{\tau} \Big)^{k-1} \d{s} \\
&\quad+ \int_2^\infty g^{**}(s) s^{-1} \Big( \int_{\sqrt{s}}^s \tau^{-1} \d{\tau} \Big)^{k-1} \d{s} \\
&\approx \int_{2t}^\infty g^{**}(s) \psi(s) \d{s}.
\end{align*}
Using that and \eqref{thm:examples_LZ_p=q=1_m_odd_gtr3_asoc_norm_explicit_with_kernel}, we have
\begin{align}
\|g\|_{Y_{m, L^1}'(0, \infty)} &\gtrsim \sup_{t\in(0, 2)} t^{\frac1{n}}\int_t^\infty g^{**}(s) \psi(s) \d{s} \notag\\
&\gtrsim \sup_{t\in(0, 1)} t^{\frac1{n}}\int_t^\infty g^{**}(s) \psi(s) \d{s} + \int_1^\infty g^{**}(s) \psi(s) \d{s} \label{thm:examples_LZ_p=q=1_m_odd_gtr3_asoc_norm_explicit_copson_lower}
\end{align}
for every $g\in\Mpl(0, \infty)$. Hence we obtain \eqref{thm:examples_LZ_p=q=1_m_odd_gtr3_asoc_norm_explicit_copson} by combining \eqref{thm:examples_LZ_p=q=1_m_odd_gtr3_asoc_norm_explicit_copson_upper} and \eqref{thm:examples_LZ_p=q=1_m_odd_gtr3_asoc_norm_explicit_copson_lower}. Now, we claim that
\begin{align}
&\sup_{t\in(0, 1)}t^{\frac1{n}} \int_t^\infty g^{**}(s) \psi(s) \d{s} + \int_1^\infty g^{**}(s) \psi(s) \d{s} \notag\\
&\approx \sup_{t\in(0, 1)}t^{-1 + \frac{m}{n}} g^{**}(s) + \int_1^\infty g^{**}(s) \psi(s) \d{s} \label{thm:examples_LZ_p=q=1_m_odd_gtr3_asoc_norm_explicit_with_two_stars}
\end{align}
for every $g\in\Mpl(0, \infty)$. On the one hand, for every $t\in(0, 1)$ and $g\in\Mpl(0, \infty)$, we have
\begin{align}
t^{\frac1{n}} \int_t^\infty g^{**}(s) \psi(s) \d{s} &\geq t^{\frac1{n}} \int_t^{2t} g^{**}(s) \psi(s) \d{s} \approx t^{\frac1{n}} \int_t^{2t} g^{**}(s) s^{-1 + \frac{2k}{n}} \d{s} \notag\\
&\geq t^{\frac1{n}} g^{**}(2t) \int_t^{2t} s^{-1 + \frac{2k}{n}} \d{s} \approx t^{\frac1{n}} g^{**}(2t) t^{\frac{2k}{n}} \geq \frac1{2}t^{\frac{m}{n}}g^{**}(t) \label{thm:examples_LZ_p=q=1_m_odd_gtr3_asoc_norm_explicit_with_two_stars_lower}.
\end{align}
On the other hand, thanks to \cite[Theorem~2]{M:72}, we have
\begin{equation*}
\sup_{t\in(0,1)}t^{\frac1{n}}\int_t^1 h(s)\d{s} \lesssim \sup_{t\in(0, 1)} t^{\frac{m}{n} + 1 - \frac{2k}{n}} h(t) \quad \text{for every $h\in\Mpl(0, \infty)$}.
\end{equation*}
It follows that
\begin{align}
\sup_{t\in(0, 1)}t^{\frac1{n}} \int_t^\infty g^{**}(s) \psi(s) \d{s} &\lesssim \sup_{t\in(0, 1)} t^{\frac{m}{n} + 1 - \frac{2k}{n}} g^{**}(t) \psi(t) + \int_1^\infty g^{**}(s) \psi(s) \d{s} \notag\\
&\approx \sup_{t\in(0, 1)} t^{\frac{m}{n}} g^{**}(t) + \int_1^\infty g^{**}(s) \psi(s) \d{s} \label{thm:examples_LZ_p=q=1_m_odd_gtr3_asoc_norm_explicit_with_two_stars_upper}
\end{align}
for every $g\in\Mpl(0, \infty)$. Hence \eqref{thm:examples_LZ_p=q=1_m_odd_gtr3_asoc_norm_explicit_with_two_stars} follows from \eqref{thm:examples_LZ_p=q=1_m_odd_gtr3_asoc_norm_explicit_with_two_stars_lower} and \eqref{thm:examples_LZ_p=q=1_m_odd_gtr3_asoc_norm_explicit_with_two_stars_upper}. Finally, we claim that
\begin{equation}\label{thm:examples_LZ_p=q=1_m_odd_gtr3_asoc_norm_explicit_with_one_star}
\sup_{t\in(0, 1)} t^{\frac{m}{n}} g^{**}(t) + \int_1^\infty g^{**}(s) \psi(s) \d{s} \approx \sup_{t\in(0, 1)} t^{\frac{m}{n}} g^*(t) + \int_1^\infty g^*(s) \psi(s) \d{s}
\end{equation}
for every $g\in\Mpl(0, \infty)$. To that end, owing to weighted Hardy's inequalities \cite[Theorems~1~and~2]{M:72} again, we have
\begin{align}
\sup_{t\in(0, 1)} t^{\frac{m}{n}} g^{**}(t) &\lesssim \sup_{t\in(0, 1)} t^{\frac{m}{n}} g^*(t) \label{thm:examples_LZ_p=q=1_m_odd_gtr3_asoc_norm_explicit_with_one_star_upper_sup}\\
\intertext{and}
\int_1^\infty s^{-2}\ell^{k-1}(s) \Big( \int_1^s g^* \Big) \d{s} &\lesssim \int_1^\infty s^{-1}\ell^{k-1}(s) g^*(s) \d{s} \notag
\end{align}
for every $g\in\Mpl(0, \infty)$. Using the latter, we have
\begin{align}
\int_1^\infty g^{**}(s) \psi(s) \d{s} &= \Big(\int_0^1 g^*\Big)\int_1^\infty s^{-2}\ell^{k-1}(s) \d{s} + \int_1^\infty s^{-2}\ell^{k-1}(s) \Big( \int_1^s g^* \Big) \d{s} \notag\\
&\lesssim \sup_{t\in(0, 1)} t^{\frac{m}{n}} g^*(t) \Big(\int_0^1 s^{-\frac{m}{n}}\d{s}\Big) + \int_1^\infty s^{-1}\ell^{k-1}(s) g^*(s) \d{s} \notag\\
&\approx \sup_{t\in(0, 1)} t^{\frac{m}{n}} g^*(t) + \int_1^\infty g^*(s) \psi(s) \d{s} \label{thm:examples_LZ_p=q=1_m_odd_gtr3_asoc_norm_explicit_with_one_star_upper_intergral}
\end{align}
for every $g\in\Mpl(0, \infty)$. Hence \eqref{thm:examples_LZ_p=q=1_m_odd_gtr3_asoc_norm_explicit_with_one_star} follows from \eqref{thm:examples_LZ_p=q=1_m_odd_gtr3_asoc_norm_explicit_with_one_star_upper_sup} and \eqref{thm:examples_LZ_p=q=1_m_odd_gtr3_asoc_norm_explicit_with_one_star_upper_intergral} combined with \eqref{prel:ri:twostarsdominateonestar}. Now, by combining \eqref{thm:examples_LZ_p=q=1_m_odd_gtr3_asoc_norm_explicit_copson}, \eqref{thm:examples_LZ_p=q=1_m_odd_gtr3_asoc_norm_explicit_with_two_stars}, and \eqref{thm:examples_LZ_p=q=1_m_odd_gtr3_asoc_norm_explicit_with_one_star}, we obtain
\begin{equation*}
\|g\|_{Y_{m, L^1}'(0, \infty)} \approx \sup_{t\in(0, 1)} t^{\frac{m}{n}} g^*(t) + \int_1^\infty s^{-1}\ell^{k-1}(s) g^*(s) \d{s}
\end{equation*}
for every $g\in\Mpl(0, \infty)$. Moreover, it follows from \cite[Theorem~4.1]{M:79} together with \cite[Section~3]{S:72} that
\begin{equation*}
\|g\|_{(L^{\frac{n}{m}, \infty} + W')(0, \infty)} \approx \sup_{t\in(0, 1)} t^{\frac{m}{n}} g^*(t) + \int_1^\infty s^{-1}\ell^{k-1}(s) g^*(s) \d{s}
\end{equation*}
for every $g\in\Mpl(0, \infty)$, where $W(0, \infty) = L^{(1, \infty; [0, -k])}(0, \infty)$. Therefore, we have proved that
\begin{equation*}
Y_{m, L^1}'(0, \infty) = (L^{\frac{n}{m}, \infty} + W')(0, \infty).
\end{equation*}
Hence, combining that with \eqref{prel:ri:dual_sum_and_inter} and \eqref{prel:ri:X''=X}, we obtain
\begin{equation*}
Y_{m, L^1}(0, \infty) = L^{\frac{n}{n - m}, 1}(0, \infty) \cap L^{(1, \infty; [0, -k])}(0, \infty).
\end{equation*}
At last, all that remains is to observe that
\begin{equation*}
\max\Big\{ \int_0^\infty t^{\frac{n - m}{n} - 1} f^*(t) \d{t}, \sup_{t\in(0, \infty)} t \ell^{[0, -k]}(t)f^{**}(t) \Big\} \approx \|f\|_{Z_1(0, \infty)}
\end{equation*}
for every $f\in\Mpl(0, \infty)$. That follows from the following two inequalities, which are valid for every $f\in\Mpl(0, \infty)$ and which can be readily verified:
\begin{align*}
\int_1^\infty t^{\frac{n - m}{n} - 1} f^*(t) \d{t} &\leq \Big( \int_1^\infty t^{-2 + \frac{n - m}{n}} \ell^k(t)\d{t}\Big) \sup_{t\in[1, \infty)} t\ell^{-k}(t)f^{**}(t) \\
\intertext{and}
\sup_{t\in(0, 1)}t f^{**}(t) &= f^{**}(1) \leq \ell^k(1) \sup_{t\in[1, \infty)} t\ell^{-k}(t)f^{**}(t). \qedhere
\end{align*}
\end{proof}

\section*{Acknowledgment}
The author would like to thank the referee for their very valuable and insightful comments.


\begin{thebibliography}{35}
\providecommand{\natexlab}[1]{#1}
\providecommand{\url}[1]{\texttt{#1}}
\expandafter\ifx\csname urlstyle\endcsname\relax
  \providecommand{\doi}[1]{doi: #1}\else
  \providecommand{\doi}{doi: \begingroup \urlstyle{rm}\Url}\fi

\bibitem[Alberico et~al.(2018)Alberico, Cianchi, Pick, and
  Slav\'{\i}kov\'{a}]{ACPS:18}
A.~Alberico, A.~Cianchi, L.~Pick, and L.~Slav\'{\i}kov\'{a}.
\newblock Sharp {S}obolev type embeddings on the entire {E}uclidean space.
\newblock \emph{Commun. Pure Appl. Anal.}, 17\penalty0 (5):\penalty0
  2011--2037, 2018.
\newblock \doi{10.3934/cpaa.2018096}.

\bibitem[Bennett and Sharpley(1988)]{BS}
C.~Bennett and R.~Sharpley.
\newblock \emph{{Interpolation of operators}}, volume 129 of \emph{Pure and
  Applied Mathematics}.
\newblock Academic Press, Inc., Boston, MA, 1988.

\bibitem[B\"{o}gelein et~al.(2015)B\"{o}gelein, Duzaar, and Scheven]{BDS:15}
V.~B\"{o}gelein, F.~Duzaar, and C.~Scheven.
\newblock A sharp quantitative isoperimetric inequality in hyperbolic
  {$n$}-space.
\newblock \emph{Calc. Var. Partial Differential Equations}, 54\penalty0
  (4):\penalty0 3967--4017, 2015.
\newblock \doi{10.1007/s00526-015-0928-9}.

\bibitem[Breit and Cianchi(2021)]{BC:21}
D.~Breit and A.~Cianchi.
\newblock Symmetric gradient {S}obolev spaces endowed with
  rearrangement-invariant norms.
\newblock \emph{Adv. Math.}, 391:\penalty0 107954, 101 pp., 2021.
\newblock \doi{10.1016/j.aim.2021.107954}.

\bibitem[Br\'{e}zis and Wainger(1980)]{BW:80}
H.~Br\'{e}zis and S.~Wainger.
\newblock A note on limiting cases of {S}obolev embeddings and convolution
  inequalities.
\newblock \emph{Comm. Partial Differential Equations}, 5\penalty0 (7):\penalty0
  773--789, 1980.
\newblock \doi{10.1080/03605308008820154}.

\bibitem[Carro et~al.(1996)Carro, Garc\'{\i}a~del Amo, and Soria]{CGS:96}
M.J. Carro, A.~Garc\'{\i}a~del Amo, and J.~Soria.
\newblock Weak-type weights and normable {L}orentz spaces.
\newblock \emph{Proc. Amer. Math. Soc.}, 124\penalty0 (3):\penalty0 849--857,
  1996.
\newblock \doi{10.1090/S0002-9939-96-03214-5}.

\bibitem[Cianchi(1997)]{C:97}
A.~Cianchi.
\newblock A note on two-weight inequalities for maximal functions and singular
  integrals.
\newblock \emph{Bull. London Math. Soc.}, 29\penalty0 (1):\penalty0 53--59,
  1997.
\newblock \doi{10.1112/S0024609396001798}.

\bibitem[Cianchi(2004)]{C:04}
A.~Cianchi.
\newblock Symmetrization and second-order {S}obolev inequalities.
\newblock \emph{Ann. Mat. Pura Appl. (4)}, 183\penalty0 (1):\penalty0 45--77,
  2004.
\newblock \doi{10.1007/s10231-003-0080-6}.

\bibitem[Cianchi and Pick(1998)]{CP:98}
A.~Cianchi and L.~Pick.
\newblock Sobolev embeddings into {BMO}, {VMO}, and {$L_\infty$}.
\newblock \emph{Ark. Mat.}, 36\penalty0 (2):\penalty0 317--340, 1998.
\newblock \doi{10.1007/BF02384772}.

\bibitem[Cianchi and Pick(2009)]{CP:09}
A.~Cianchi and L.~Pick.
\newblock Optimal {G}aussian {S}obolev embeddings.
\newblock \emph{J. Funct. Anal.}, 256\penalty0 (11):\penalty0 3588--3642, 2009.
\newblock \doi{10.1016/j.jfa.2009.03.001}.

\bibitem[Cianchi and Pick(2016)]{CP:16}
A.~Cianchi and L.~Pick.
\newblock Optimal {S}obolev trace embeddings.
\newblock \emph{Trans. Amer. Math. Soc.}, 368\penalty0 (12):\penalty0
  8349--8382, 2016.
\newblock \doi{10.1090/tran/6606}.

\bibitem[Cianchi et~al.(2008)Cianchi, Kerman, and Pick]{CKP:08}
A.~Cianchi, R.~Kerman, and L.~Pick.
\newblock Boundary trace inequalities and rearrangements.
\newblock \emph{J. Anal. Math.}, 105:\penalty0 241--265, 2008.
\newblock \doi{10.1007/s11854-008-0036-2}.

\bibitem[Cianchi et~al.(2015)Cianchi, Pick, and Slav\'{\i}kov\'{a}]{CPS:15}
A.~Cianchi, L.~Pick, and L.~Slav\'{\i}kov\'{a}.
\newblock Higher-order {S}obolev embeddings and isoperimetric inequalities.
\newblock \emph{Adv. Math.}, 273:\penalty0 568--650, 2015.
\newblock \doi{10.1016/j.aim.2014.12.027}.

\bibitem[Cianchi et~al.(2020)Cianchi, Pick, and Slav\'{\i}kov\'{a}]{CPS:20}
A.~Cianchi, L.~Pick, and L.~Slav\'{\i}kov\'{a}.
\newblock Sobolev embeddings, rearrangement-invariant spaces and {F}rostman
  measures.
\newblock \emph{Ann. Inst. H. Poincar\'{e} C Anal. Non Lin\'{e}aire},
  37\penalty0 (1):\penalty0 105--144, 2020.
\newblock \doi{10.1016/j.anihpc.2019.06.004}.

\bibitem[Cwikel et~al.(2003)Cwikel, Nilsson, and Schechtman]{CNS:03}
M.~Cwikel, P.G. Nilsson, and G.~Schechtman.
\newblock Interpolation of weighted {B}anach lattices. {A} characterization of
  relatively decomposable {B}anach lattices.
\newblock \emph{Mem. Amer. Math. Soc.}, 165\penalty0 (787):\penalty0 vi+127,
  2003.
\newblock \doi{10.1090/memo/0787}.

\bibitem[Edmunds et~al.(2020)Edmunds, Mihula, Musil, and Pick]{EMMP:20}
D.E. Edmunds, Z.~Mihula, V.~Musil, and L.~Pick.
\newblock Boundedness of classical operators on rearrangement-invariant spaces.
\newblock \emph{J. Funct. Anal.}, 278\penalty0 (4):\penalty0 108341, 56 pp.,
  2020.
\newblock \doi{10.1016/j.jfa.2019.108341}.

\bibitem[Gogatishvili and Pick(2003)]{GP:03}
A.~Gogatishvili and L.~Pick.
\newblock Discretization and anti-discretization of rearrangement-invariant
  norms.
\newblock \emph{Publ. Mat.}, 47\penalty0 (2):\penalty0 311--358, 2003.
\newblock \doi{10.5565/PUBLMAT\_47203\_02}.

\bibitem[Gogatishvili and Soudsk\'{y}(2014)]{GS:14}
A.~Gogatishvili and F.~Soudsk\'{y}.
\newblock Normability of {L}orentz spaces---an alternative approach.
\newblock \emph{Czechoslovak Math. J.}, 64(139)\penalty0 (3):\penalty0
  581--597, 2014.
\newblock \doi{10.1007/s10587-014-0120-y}.

\bibitem[Gogatishvili and Stepanov(2013)]{GS:13}
A.~Gogatishvili and V.D. Stepanov.
\newblock Reduction theorems for weighted integral inequalities on the cone of
  monotone functions.
\newblock \emph{Russian Math. Surveys}, 68\penalty0 (4):\penalty0 597--664,
  2013.
\newblock \doi{10.1070/RM2013v068n04ABEH004849}.

\bibitem[Gogatishvili et~al.(2006)Gogatishvili, Opic, and Pick]{GOP:06}
A.~Gogatishvili, B.~Opic, and L.~Pick.
\newblock {Weighted inequalities for Hardy-type operators involving suprema}.
\newblock \emph{Collect. Math.}, 57\penalty0 (3):\penalty0 227--255, 2006.

\bibitem[Hansson(1979)]{H:79}
K.~Hansson.
\newblock Imbedding theorems of {S}obolev type in potential theory.
\newblock \emph{Math. Scand.}, 45\penalty0 (1):\penalty0 77--102, 1979.
\newblock \doi{10.7146/math.scand.a-11827}.

\bibitem[Kerman and Pick(2006)]{KP:06}
R.~Kerman and L.~Pick.
\newblock Optimal {S}obolev imbeddings.
\newblock \emph{Forum Math.}, 18\penalty0 (4):\penalty0 535--570, 2006.
\newblock \doi{10.1515/FORUM.2006.028}.

\bibitem[Lee(2018)]{L:18}
J.M. Lee.
\newblock \emph{Introduction to {R}iemannian manifolds}, volume 176 of
  \emph{Graduate Texts in Mathematics}.
\newblock Springer, Cham, 2018.

\bibitem[Mihula(2021{\natexlab{a}})]{M:21}
Z.~Mihula.
\newblock Embeddings of homogeneous {S}obolev spaces on the entire space.
\newblock \emph{Proc. Roy. Soc. Edinburgh Sect. A}, 151\penalty0 (1):\penalty0
  296--328, 2021{\natexlab{a}}.
\newblock \doi{10.1017/prm.2020.14}.

\bibitem[Mihula(2021{\natexlab{b}})]{M:21b}
Z.~Mihula.
\newblock Poincar\'{e}-{S}obolev inequalities with rearrangement\hyp{}invariant
  norms on the entire space.
\newblock \emph{Math. Z.}, 298\penalty0 (3-4):\penalty0 1623--1640,
  2021{\natexlab{b}}.
\newblock \doi{10.1007/s00209-020-02652-z}.

\bibitem[Milman(1979)]{M:79}
M.~Milman.
\newblock Interpolation of operators of mixed weak-strong type between
  rearrangement invariant spaces.
\newblock \emph{Indiana Univ. Math. J.}, 28\penalty0 (6):\penalty0 985--992,
  1979.
\newblock \doi{10.1512/iumj.1979.28.28071}.

\bibitem[Muckenhoupt(1972)]{M:72}
B.~Muckenhoupt.
\newblock Hardy's inequality with weights.
\newblock \emph{Studia Math.}, 44:\penalty0 31--38, 1972.
\newblock \doi{10.4064/sm-44-1-31-38}.

\bibitem[Ng\^{o} and Nguyen(2020)]{NN:20}
Q.A. Ng\^{o} and V.H. Nguyen.
\newblock Sharp {A}dams-{M}oser-{T}rudinger type inequalities in the hyperbolic
  space.
\newblock \emph{Rev. Mat. Iberoam.}, 36\penalty0 (5):\penalty0 1409--1467,
  2020.
\newblock \doi{10.4171/rmi/1171}.

\bibitem[Nguyen(2018)]{N:18}
V.H. Nguyen.
\newblock The sharp {P}oincar\'{e}-{S}obolev type inequalities in the
  hyperbolic spaces {$\mathbb{H}^n$}.
\newblock \emph{J. Math. Anal. Appl.}, 462\penalty0 (2):\penalty0 1570--1584,
  2018.
\newblock \doi{10.1016/j.jmaa.2018.02.054}.

\bibitem[Nguyen(2020)]{N:20}
V.H. Nguyen.
\newblock The sharp {S}obolev type inequalities in the {L}orentz-{S}obolev
  spaces in the hyperbolic spaces.
\newblock \emph{J. Math. Anal. Appl.}, 490\penalty0 (1):\penalty0 124197, 21
  pp., 2020.
\newblock \doi{10.1016/j.jmaa.2020.124197}.

\bibitem[Nguyen(2021)]{N:21}
V.H. Nguyen.
\newblock The sharp higher-order {L}orentz-{P}oincar\'{e} and
  {L}orentz-{S}obolev inequalities in the hyperbolic spaces.
\newblock \emph{Ann. Mat. Pura Appl. (4)}, 200\penalty0 (5):\penalty0
  2133--2153, 2021.
\newblock \doi{10.1007/s10231-021-01072-y}.

\bibitem[Opic and Pick(1999)]{OP:99}
B.~Opic and L.~Pick.
\newblock On generalized {L}orentz-{Z}ygmund spaces.
\newblock \emph{Math. Inequal. Appl.}, 2\penalty0 (3):\penalty0 391--467, 1999.
\newblock \doi{10.7153/mia-02-35}.

\bibitem[Pe\v{s}a(2020)]{Pe:20}
D.~Pe\v{s}a.
\newblock Reduction principle for a certain class of kernel-type operators.
\newblock \emph{Math. Nachr.}, 293\penalty0 (4):\penalty0 761--773, 2020.
\newblock \doi{10.1002/mana.201800510}.

\bibitem[Sawyer(1990)]{S:90}
E.~Sawyer.
\newblock Boundedness of classical operators on classical {L}orentz spaces.
\newblock \emph{Studia Math.}, 96\penalty0 (2):\penalty0 145--158, 1990.
\newblock \doi{10.4064/sm-96-2-145-158}.

\bibitem[Sharpley(1972)]{S:72}
R.~Sharpley.
\newblock {Spaces $\Lambda_\alpha(X)$ and interpolation}.
\newblock \emph{J. Funct. Anal.}, 11\penalty0 (4):\penalty0 479--513, 1972.
\newblock \doi{10.1016/0022-1236(72)90068-7}.

\end{thebibliography}
\end{document}